\documentclass[11pt]{article}
\usepackage{amsfonts}
\usepackage{mathrsfs}
\usepackage{bbm}
\usepackage{amsfonts}
\usepackage{amssymb,amsmath,amsthm,graphicx}
\usepackage{float}
\usepackage[colorlinks, linkcolor=red]{hyperref}
\usepackage{color, xcolor}
\usepackage{cite}
\usepackage[paper=a4paper, left=1.6cm, right=1.6cm, top=1.8cm, bottom=1.6cm, headheight=5.5pt, footskip=0.8cm, footnotesep=0.8cm, centering, includefoot]{geometry}
\hypersetup{urlcolor=red, citecolor=red}

\DeclareMathOperator{\divv}{div}
\DeclareMathOperator{\curl}{curl}

\allowdisplaybreaks

\begin{document}
\title{Entropy-bounded solutions to the 3D compressible heat-conducting
magnetohydrodynamic equations with vacuum at infinity
\thanks{
This research was partially supported by National Natural Science Foundation of China (Nos. 11901288, 11901474, 12071359), Scientific Research Foundation of Jilin Provincial Education Department (No. JJKH20210873KJ), Postdoctoral Science Foundation of
China (No. 2021M691219), and Exceptional Young Talents Project of Chongqing Talent (No. cstc2021ycjh-bgzxm0153).}
}
\author{Yang Liu$\,^{\rm 1, 2}\,$\quad Xin Zhong$\,^{\rm 3}\,${\thanks{Corresponding author. E-mail address: liuyang0405@ccsfu.edu.cn (Y. Liu), xzhong1014@amss.ac.cn (X. Zhong).}}
\date{}\\
\footnotesize $^{\rm 1}\,$
School of Mathematics, Jilin University, Changchun 130012, P. R. China\\
\footnotesize $^{\rm 2}\,$ College of Mathematics, Changchun Normal
University, Changchun 130032, P. R. China\\
\footnotesize $^{\rm 3}\,$ School of Mathematics and Statistics, Southwest University, Chongqing 400715, P. R. China} \maketitle
\newtheorem{theorem}{Theorem}[section]
\newtheorem{definition}{Definition}[section]
\newtheorem{lemma}{Lemma}[section]
\newtheorem{proposition}{Proposition}[section]
\newtheorem{corollary}{Corollary}[section]
\newtheorem{remark}{Remark}[section]
\renewcommand{\theequation}{\thesection.\arabic{equation}}
\catcode`@=11 \@addtoreset{equation}{section} \catcode`@=12
\maketitle{}

\begin{abstract}
The mathematical analysis on the behavior of the entropy for viscous, compressible, and heat conducting magnetohydrodynamic flows near the vacuum region is a challenging problem as the governing equation for entropy is highly degenerate and singular in the vacuum region. In particular, it is unknown whether the entropy remains its boundedness.
In the present paper, we investigate the Cauchy problem to the three-dimensional (3D) compressible heat-conducting magnetohydrodynamic equations with vacuum at infinity only. We show that the uniform boundedness of the entropy and the $L^2$ regularities of the velocity and temperature can be propagated provided that the initial density decays suitably slow at infinity. The main tools are based on singularly weighted energy estimates and De Giorgi type iteration techniques developed by Li and Xin (https://arxiv.org/abs/2111.14057) for the 3D full compressible Navier-Stokes system. Some new mathematical techniques and useful estimates are developed to deduce the lower and upper bounds on the entropy.
%The main ingredient of this paper is to overcome the strong coupling effects between the velocity and magnetic field.
\end{abstract}



%%%%%%%%%%%%%%%%%%%%%%%%%%%%%%%%%%%%%%%%%%%%%%%%%%%%%%%%%%%%%%%%%%%%%%%%%%%%%%%%%%%%%%%%%%%%%%%%%%

\section{Introduction}
The motion of a compressible, viscous, and heat conducting magnetohydrodynamic (MHD) flow in $\mathbb{R}^3$ can be described by the following equations (cf. \cite[Chapter 3]{LQ2012}):
\begin{align}\label{a1}
\begin{cases}
\rho_t+\divv(\rho u)=0,\\
\rho u_t+\rho u\cdot\nabla u-\mu\Delta u-(\lambda+\mu)\nabla\divv u+\nabla p
=b\cdot\nabla b-\frac12\nabla|b|^2,\\
c_v\rho(\theta_t+u\cdot\nabla\theta)+p\divv u-\kappa\Delta\theta=\mathcal{Q}(\nabla u)
+\nu|\curl b|^2,\\
b_t-b\cdot\nabla u+u\cdot\nabla b+b\divv u=\nu\Delta b,\\
\divv b=0,
\end{cases}
\end{align}
where the unknowns $\rho\ge 0$, $u\in\mathbb{R}^3$, $\theta\ge 0$, $p\ge 0$, and $b\in\mathbb{R}^3$ are
the density, velocity, pressure, absolute temperature, pressure, and magnetic field,
respectively. Here, $c_v$ is a positive constant, $\nu$ is the magnetic diffusive coefficient, and $\mu$ and $\lambda$ are the viscous coefficients
satisfying the physical constraints
\begin{align*}
\mu>0, \quad 2\mu+3\lambda\ge 0.
\end{align*}
$\kappa>0$ is the heat conductive coefficient, while
\begin{align*}
\mathcal{Q}(\nabla u)=\frac{\mu}{2}|\nabla u+(\nabla u)^\top|^2
+\lambda(\divv u)^2,
\end{align*}
with $(\nabla u)^\top$ being the transpose of $\nabla u$.

In order to close the system \eqref{a1}, one needs first to postulate a constitutive equation relating the pressure to the other variables.
We assume that the pressure $p$ takes the form
\begin{align*}
p=\widetilde{R}\rho\theta
\end{align*}
for a positive constant $\widetilde{R}$. Let $e=c_v\theta\ (c_v>0)$ be the specific internal energy and $s$ be the specific entropy. The specific entropy $s$, the specific internal energy $e$, and the pressure $p$ are interrelated through the second principle of thermodynamics
\begin{align}
\theta Ds=De+D\Big(\frac{1}{\rho}\Big)p,
\end{align}
where $D$ stands for the differential with respect to the variables $\rho$ and $\theta$, which implies some compatibility conditions between $e$ and $p$ ( Maxwell's relation).
Then it holds that
\begin{align*}
p=Ae^{\frac{s}{c_v}\rho^\gamma}
\end{align*}
for some positive constant $A$, where $\gamma=1+\frac{\widetilde{R}}{c_v}$. In terms of
$\rho$ and $\theta$, $s$ can be expressed as
\begin{align}\label{1.2}
s=c_v\Big(\log\frac{\widetilde{R}}{A}+\log\theta-(\gamma-1)\log\rho\Big),
\end{align}
and it will be more convenient to replace \eqref{a1}$_3$ by the
\textit{entropy balance}
\begin{align}\label{1.3}
\rho(s_t+u\cdot\nabla s)-\frac{\kappa}{c_v}\Delta s=\kappa(\gamma-1)\divv\Big(\frac{\nabla\rho}{\rho}\Big)
+\frac{1}{\theta}\Big(\mathcal Q(\nabla u)
+\nu|\curl b|^2+\kappa\frac{|\nabla\theta|^2}{\theta}\Big)
\end{align}
in the region where both $\rho$ and $\theta$ are positive.

When the temperature fluctuations can be neglected, \eqref{a1} becomes the compressible isentropic MHD system
\begin{align}\label{a3}
\begin{cases}
\rho_t+\divv(\rho u)=0,\\
\rho u_t+\rho u\cdot\nabla u-\mu\Delta u-(\lambda+\mu)\nabla\divv u+\nabla p
=b\cdot\nabla b-\frac12\nabla|b|^2,\\
b_t-b\cdot\nabla u+u\cdot\nabla b+b\divv u=\nu\Delta b,\\
\divv b=0,\ p(\rho)=A\rho^\gamma.
\end{cases}
\end{align}
Hu and Wang \cite{HW10} studied the global existence and large-time behavior of finite energy weak solutions to \eqref{a3} in a bounded domain $\Omega\subseteq\mathbb{R}^3$ with Dirichlet boundary condition $u|_{\partial\Omega}=b|_{\partial\Omega}=0$.
Motivated by \cite{H95,H06} where D. Hoff constructed the so-called \textit{intermediate weak solutions},
Suen and Hoff \cite{SH12} proved the global-in-time existence of weak solutions of \eqref{a3} in $\mathbb{R}^3$ with initial data small in $L^2$ and initial density positive and essentially bounded if
the viscosity coefficients $\mu$ and $\lambda$ fulfill the additional assumption
\begin{align}\label{a4}
\mu<2\mu+\lambda<\Big(\frac32+\frac{\sqrt{21}}{6}\Big)\mu.
\end{align}
This result was later improved by Liu-Yu-Zhang \cite{LYZ13} where the condition \eqref{a4} was removed and vacuum was allowed initially. Moreover, the uniqueness and continuous dependence of these weak solutions were established in \cite{S20} by modifying the initial conditions. Meanwhile, Wu-Zhang-Zou \cite{WZZ21} obtained the optimal time-decay rates of weak solutions with discontinuous initial data, i.e., for $t\geq1$,
\begin{align*}
\begin{cases}
\|(\rho-\tilde{\rho},u,b)(t)\|_{L^r}\leq Ct^{-\frac32(1-\frac1r)},\ 2\leq r\leq\infty,\\
\|(\nabla u,\nabla b)(t)\|_{L^2}\leq Ct^{-\frac54}, \\
\|(\nabla^2b,b_t)(t)\|_{L^2}\leq Ct^{-\frac74}, \\
\end{cases}
\end{align*}
if $\|(\rho_0-\tilde{\rho},u_0,b_0)\|_{L^1}$ is bounded. Apart from \textit{large-energy weak solutions} \cite{HW10} and \textit{intermediate weak solutions} \cite{SH12,LYZ13,S20,WZZ21}, the third type of solutions to \eqref{a3} are the \textit{small-smooth solutions}. More precisely, Li-Xu-Zhang \cite{LXZ13} showed the global existence of classical solutions to the 3D Cauchy problem of \eqref{a3} with small initial energy but possibly large oscillations and vacuum states at both the interior domain and the far field which generalized the previous result proved by Huang-Li-Xin \cite{HLX12} for compressible Navier-Stokes system. A similar result for 2D Cauchy problem has been established by L{\"u}-Shi-Xu \cite{LSX2016} with the help of spatial weighted energy method. Later on, Hong-Hou-Peng-Zhu \cite{HHPZ17} obtained global smooth solutions by requiring $\big[(\gamma-1)^{-\frac19}+\nu^{-\frac14}\big]E_0$ to be suitably small, where $E_0$ is the initial energy. Very recently, Chen-Huang-Shi \cite{CHS21,CHS22} investigated global strong solutions to the initial-boundary value problem of \eqref{a3} with Navier slip boundary conditions. We refer the readers to \cite{FL20,LS19,WW17,TW18} for more results on global solutions to multi-dimensional compressible non-resistive or inviscid MHD equations. There are some interesting results on other studies of \eqref{a3}, such as blow-up criterion of solutions \cite{DW15,XZ12}, asymptotic limits of solutions \cite{GLX19,L2012,JJL10}, the large time behavior of solutions \cite{LY11,ZZ21}, and so on.

Recently, there have been numerous studies on compressible heat conducting MHD equations \eqref{a1} due to its physical importance, complexity, rich phenomena, and mathematical challenges.
By using the method of weak convergence developed by Lions \cite{L1998} and Feireisl \cite{F04,EF04}, Ducomet and Feireisl \cite{DF06} obtained the global existence of variational weak solutions in a bounded domain of $\mathbb{R}^3$ under certain conditions upon the equations of state. This result was improved by Li and Guo \cite{LG2014} by allowing a rather more general constitutive relationship. At the same time, Hu and Wang \cite{HW2008} also proved the existence of a global variational weak solution of \eqref{a1} with large data. The main difference between \cite{DF06} and \cite{HW2008} is that they used the entropy equation \eqref{1.3} and the thermal equation \eqref{a1}$_3$, respectively. On the other hand, for global existence of strong solutions, Kawashima \cite{K83} first established the global-in-time existence of $H^3$ solutions to \eqref{a1} when the initial data was taken to be properly small in $H^3$ modulo a constant state.
His analysis consists an iterative procedure based on asymptotic decay rates for the corresponding linearized equations developed by Matsumura and Nishida \cite{MN80}.
It is worth mentioning that the initial data is small close to a constant state in $H^3$-norm indicates that there is absent of vacuum states. However, as emphasized in many papers (see, e.g., \cite{HLX12,LX20,LX22,LX23,LXZ13,HHPZ17}), the possible presence of vacuum is one of the major difficulties in the study of mathematical theory of compressible fluids. Very recently, several results are devoted to investigating global well-posedness of strong solutions to the 3D Cauchy problem of \eqref{a1} with vacuum both at interior region and far field. Liu and Zhong \cite{LZ20} proved global strong solutions provided that $\|\rho_0\|_{L^\infty}+\|b_0\|_{L^3}$ is suitably small and the viscosity coefficients satisfy $3\mu>\lambda$. This result was later improved in \cite{LZ22} where the authors showed the global existence and uniqueness of strong solutions, which may be of possibly large oscillations, when the initial data are of small total energy. Moreover, they also derived the algebraic time decay estimates of the solution. Hou-Jiang-Peng \cite{HJP22} obtained global strong solutions under the condition that $\|\rho_0\|_{L^1}+\|b_0\|_{L^2}$ is suitably small. Meanwhile, Liu and Zhong \cite{LZ2022} deduced global existence of strong solutions under a smallness condition on scaling invariant quantity independent of any norms of the initial data.
We refer to \cite{L22,LS21} for more results on global solutions to multi-dimensional full compressible non-resistive MHD equations. There are some interesting results on other studies of \eqref{a1}, such as blow-up criterion of solutions \cite{HL13,W21,Z20}, asymptotic limits of solutions \cite{JJL12,JJLX14,JLL13,KT11}, and so on.

Although there are abundant researches to the compressible heat-conducting MHD equations \eqref{a1}, yet it should be pointed out that all the results mentioned above tell us relatively little about the entropy. One of the main reasons is that the lack of the expression of the entropy in the vacuum region and the high singularity and degeneracy of the entropy equation \eqref{1.3} near the vacuum region. Very recently, Li and Xin \cite{LX20,LX22,LX23} investigated the uniform boundedness of the entropy to the Cauchy problem of one-dimensional and three-dimensional full compressible Navier-Stokes equations in a very technical and subtle way. Their results reveal that the uniform boundedness of the entropy can be propagated up to any finite time as long as the initial vacuum presents only at far fields and the initial density decays slowly enough at the far field.

We will consider the Cauchy problem of \eqref{a1} with the initial condition
\begin{align}\label{a2}
(\rho, u, \theta, b)|_{t=0}=(\rho_0, u_0, \theta_0, b_0).
\end{align}
Motivated by \cite{LX23}, the main purpose of this paper is to study the
existence of entropy-bounded solutions of \eqref{a1} and \eqref{a2} with only vacuum at infinity. The main novelty of this current work is that we obtain some new estimates on velocity and magnetic field which are important in controlling the strong coupling effects between velocity and magnetic fields. Those estimates will be used in deriving some \textit{a priori} energy estimates with singular weights.

Before stating our main result, let us first introduce some notations. For $1\le q\le \infty$ and positive integer $m$,
$L^q=L^q(\mathbb{R}^3)$ and $W^{m, q}=W^{m, q}(\mathbb{R}^3)$, respectively, are the standard
Lebesgue and Sobolev spaces, and $H^m=W^{m, 2}$. $D_0^1=D_0^1(\mathbb{R}^3)$ and $D^{m, q}=D^{m, q}(\mathbb{R}^3)$
are the homogeneous Sobolev spaces defined, respectively, as
\begin{align*}
D_0^1=\{u\in L^6|\ \nabla u\in L^2\},\
D^{m, q}=\{u\in L_{loc}^1|\ \nabla^\alpha u\in L^q, 1\le |\alpha|\le m\},
\end{align*}
and $D^{m, 2}=D^m$, while
$\dot{f}\triangleq f_t+u\cdot\nabla f$ denotes the material derivative of $f$.

Strong solutions considered in this paper are defined as follows.
\begin{definition}
Let $T>0$, assume that
\begin{align}\label{1.5}
0\le\rho_0\in H^1\cap W^{1, q}, \quad u_0\in D_0^1\cap D^2,\quad 0\le \theta_0\in D_0^1\cap D^2, \quad
b_0\in H^2,
\end{align}
for some $q\in (3, 6]$. A quadruple $(\rho, u, \theta, b)$ is called a strong solution to system \eqref{a1}
in $\mathbb{R}^3\times (0, T)$ with initial condition \eqref{a2}, if it has the following regularities
\begin{align*}
\begin{cases}
\rho\in C([0, T]; H^1\cap W^{1, q}),\ \rho_t\in C([0, T]; L^2\cap L^q),\\
(u, \theta)\in C([0, T]; D^1\cap D^2)\cap L^2(0, T; D^{2, q}),\ (u_t, \theta_t)\in L^2(0, T; D_0^1), \\
b\in C([0, T]; H^2)\cap L^2(0, T; H^3),\ b_t\in C([0, T];L^2)\cap L^2(0, T; H^1),\\
(\sqrt{\rho}u_t, \sqrt{\rho}\theta_t)\in L^\infty(0, T; L^2).
\end{cases}
\end{align*}
satisfies \eqref{a1} a.e. in $\mathbb{R}^3\times (0, T)$ and fulfills the initial condition \eqref{a2}.
\end{definition}

\begin{definition}
A quadruple $(\rho, u, \theta, b)$ is called a global strong solution to system \eqref{a1}, subject to \eqref{a2},
if it is a strong solution to the same system in $\mathbb{R}^3\times (0, T)$ for
any finite time $T$.
\end{definition}

Before stating the main result of this paper, let us recall the following two theorems on the local and global well-posedness of strong solutions
to system \eqref{a1} with initial condition \eqref{a2}, respectively.
\begin{theorem}[Local well-posedness, see \cite{FY09}]\label{t1}
Let $q\in (3, 6]$ and assume in addition to \eqref{1.5} that
\begin{align}
\begin{cases}
-\mu\Delta u_0-(\lambda+\mu)\nabla\divv u_0+\widetilde{R}\nabla(\rho_0\theta_0)
-b_0\cdot\nabla b_0+\frac12\nabla|b_0|^2=\sqrt{\rho_0}g_1,\\
-\kappa\Delta\theta_0-\mathcal{Q}(\nabla u_0)-\nu|\curl b_0|^2=\sqrt{\rho_0}g_2,\label{1.6}
\end{cases}
\end{align}
for given functions $g_1$, $g_2\in L^2$. Then there exists a positive time
$T_*$ depending only on the initial data, such that system \eqref{a1}, subject to \eqref{a2},
admits a unique strong solution $(\rho, u, \theta, b)$ in $\mathbb{R}^3\times (0, T_*)$.
\end{theorem}

\begin{theorem}[Global well-posedness, see \cite{LZ2022}]\label{t2}
Assume, in addition to \eqref{1.5} and \eqref{1.6}, that $3\mu>\lambda$. Then there exists
a positive constant $\varepsilon_0$ depending only on $\widetilde{R}$, $\gamma$, $c_v$, $\mu$, $\lambda$, and $\kappa$
 such that system \eqref{a1}, subject to \eqref{a2}, has a unique global
strong solution provided that
\begin{align*}
N_0\triangleq \bar{\rho}\big[\|\rho_0\|_{L^3}+\bar{\rho}^2(\|\sqrt{\rho_0}u_0\|_{L^2}^2
+\|b_0\|_{L^2}^2)\big]
\big[\|\nabla u_0\|_{L^2}^2+\bar{\rho}(\|\sqrt{\rho_0}E_0\|_{L^2}^2+\|\nabla b_0\|_{L^2}^2)\big]\le \varepsilon_0,
\end{align*}
where $E_0=\frac{|u_0|^2}{2}+c_v\theta_0$ and $\bar{\rho}=\|\rho_0\|_{L^\infty}+1$.
\end{theorem}

Now we can state our main result of the present paper.
\begin{theorem}\label{thm1}
Assume, in addition to \eqref{1.5} and \eqref{1.6}, that the initial density $\rho_0$
is positive on $\mathbb{R}^3$ and satisfies
\begin{align}\label{h1}\tag{$\rm H1$}
|\nabla\rho_0(x)|\le K_1\rho_0^\frac32(x), \quad \forall x\in\mathbb{R}^3,
\end{align}
for a positive constant $K_1$. Denote
\begin{align*}
&s_0=c_v\Big(\log\frac{\widetilde{R}}{A}+\log\theta_0-(\gamma-1)\log\rho_0\Big), \ \underline s_0=\inf_{x\in\mathbb{R}^3}s_0(x), \
\overline s_0=\sup_{x\in\mathbb{R}^3}s_0(x),\\
 \quad
&\wp_0=\mu\Delta u_0+(\mu+\lambda)\nabla \divv u_0-\widetilde{R}\nabla(\rho_0\theta_0)+b_0\cdot\nabla b_0-\frac12\nabla|b_0|^2.
\end{align*}
Let $(\rho, u, \theta, b)$  be an arbitrary strong solution to system \eqref{a1} in $\mathbb{R}^3\times(0, T)$, subject to \eqref{a2},
and $s$ be the entropy given by \eqref{1.2}. Then the following statements hold true.

(i) If $u_0\in L^2$, then $u\in L^\infty(0, T; L^2)$.

(ii) If $\Big(\frac{u_0}{\sqrt{\rho_0}}, \theta_0\Big)\in L^2$,
then $\Big(\frac{u}{\sqrt{\rho}}, \theta\Big)\in L^\infty(0, T; L^2)$.

(iii) If there exists some positive constant $K_2$ such that
\begin{align}\label{h2}\tag{$\rm H2$}
|\nabla^2\rho_0(x)|\le K_2\rho_0^2(x), \quad \forall x\in\mathbb{R}^3,
\end{align}
then it holds that
\begin{align*}
\inf_{\mathbb{R}^3\times(0, T)}s(x, t)>-\infty, \quad\text{as~long~as}~\underline s_0>-\infty.
\end{align*}

(iv)  Assume in addition that \eqref{h2} holds true, then one has
\begin{align*}
\sup_{\mathbb{R}^3\times(0, T)}s(x, t)<+\infty,
\end{align*}
as long as $\overline s_0<+\infty$ and $\Big(\rho_0^\frac{1-\gamma}{2}u_0$,
$\rho_0^{-\frac{\gamma}{2}}\nabla u_0$, $\rho_0^{-\frac{\gamma}{2}}b_0$, $\rho_0^{-\frac{\gamma+2}{2}}\nabla b_0$,
$\rho_0^{1-\frac{\gamma}{2}}\theta_0$, $\rho_0^{1-\frac{\gamma}{2}}\nabla\theta_0$,
 $\rho_0^{-\frac{\gamma}{2}}\wp_0\Big)\in L^2$.
\end{theorem}

\begin{remark}
It is worth mentioning that there is no need to require the assumptions \eqref{h1} and \eqref{h2} in the local and global well-posedness theory of strong solutions of \eqref{a1} (see Theorems \ref{t1} and \ref{t2}). If, in addition to \eqref{1.5} and \eqref{1.6}, assume that the conditions in Theorem \ref{t2} hold true, then the entropy given in Theorem \ref{thm1} is uniformly bounded for any finite time.
\end{remark}

\begin{remark}
When $b=0$, i.e., there is no electromagnetic field effect, the compressible heat-conducting MHD equations \eqref{a1} becomes the full compressible Navier-Stokes equations, and Theorem \ref{thm1} is similar as the result of Li and Xin \cite{LX23}. Thus, we successfully extend the theory of entropy-bounded solutions given in \cite{LX23} for compressible heat-conducting Navier-Stokes system to the full compressible MHD system.
\end{remark}

The proof of Theorem \ref{thm1} is based on singularly weighted energy estimates and the De Giorgi type iterations developed by Li and Xin \cite{LX20,LX22,LX23}. However, because of the interaction between magnetic field and the hydrodynamic motion, the problem of MHD equations considered here is much more complicated than that of Navier-Stokes equations \cite{LX23}. For example, some additional difficulties will arise when we deal with the magnetic force and the convection term. In particular, the usual singularly weighted $L^2(0,T;L^2(\mathbb{R}^3))$-norm estimates of $b_t$ cannot be directly estimated because of the strong coupled term $u\cdot\nabla b$. To this end, we attempt to borrow some ideas used in \cite{LSX2016}, where the authors investigated global strong solutions to the 2D Cauchy problem of compressible isentropic MHD equations \eqref{a3}. Even so, new difficulties arise in our case since the crucial analysis in \cite{LSX2016} heavily relies on the basic fact that the initial density and the initial magnetic field need fast vanishing at far field.

To overcome these difficulties stated above, we need to make use of important relations among the velocity, magnetic field, and the effective viscous flux $F\triangleq(2\mu+\lambda)\divv u-p-\frac12|b|^2$. The following key observations help us to deal with the interaction of the velocity field and the magnetic field very well. First of all, we try to obtain the estimates on the $L^\infty(0,T;L^2(\mathbb{R}^3))$-norm of the gradients of velocity and  magnetic field with singular weights. On the one hand, multiplying \eqref{a1}$_2$ by $\rho_0^{-\gamma} u_t$, we find that the key point is to control the terms $\|\rho_0^{-\frac{\gamma+1}{2}}\nabla F\|_{L^2}$ and $\|\rho_0^{-\frac{\gamma+2}{2}}|b||\nabla b|\|_{L^2}$ (see \eqref{2.27}).
Using some facts on the effective viscous flux, we succeed in bounding the term $\|\rho_0^{-\frac{\gamma+1}{2}}\nabla F\|_{L^2}$ by
$\|\rho_0^{-\frac{\gamma+1}{2}}|b|^2\|_{L^2}$ (see \eqref{2.34}).
On the other hand, the usual $L^2(0,T;L^2(\mathbb{R}^3))$-norm of $b_t$ with singular weight cannot be directly estimated due to the strong
coupled term  $u\cdot \nabla b$. Motivated by \cite{LSX2016}, multiplying   \eqref{a1}$_4$ by $\rho_0^{-(\gamma+2)}\Delta b$ instead of $\rho_0^{-\gamma}b_t$ (see \eqref{2.28}), the coupled term $u\cdot\nabla b$ can be controlled after integration by parts (see \eqref{z2.28}).
Next, using the structure of the magnetic equation,
we multiply \eqref{a1}$_4$ by $\rho_0^{-(\gamma+2)}|b|^2 b$ and
thus obtain some the desired estimates on
$\|\rho_0^{-\frac{\gamma+1}{2}}|b|^2\|_{L^2}$ and $\|\rho_0^{-\frac{\gamma+2}{2}}|b||\nabla b|\|_{L^2}$ (see \eqref{2.33}).
Secondly, motivated by \cite{LSX2016}, we multiply the magnetic equation by $\rho_0^{-\frac{\gamma+2}{2}}b\Delta|b|^2$ and succeed in controlling the coupled term $u\cdot \nabla b$ by the gradients of both the velocity and the magnetic field after integration by parts. This yields some new \textit{a priori} estimate of the $L^2(0,T;L^2)$-norm of $\rho_0^{-\frac{\gamma+2}{2}}|b||\Delta b|$ (see \eqref{2.47}),
which plays a crucial role in deriving the estimates on the material derivative $\dot u\triangleq u_t+u\cdot\nabla u$ in
both the $L^\infty(0,T;L^2(\mathbb{R}^3))$-norm of $\rho_0^{1-\frac{\gamma}{2}}\dot u$
and the $L^2(\mathbb{R}^3\times(0,T ))$-norm of $\rho_0^{\frac{1-\gamma}{2}}\nabla \dot u$ (see \eqref{z2.49}).
Thirdly, the estimates on the $L^\infty(0,T;L^6(\mathbb{R}^3))$-norm of $\rho_0^{-\frac{\gamma}{2}}\nabla u$ and
$\rho_0^{-\frac{\gamma}{2}}\nabla b$ are the key point to deal with the upper bound of iterations. However, it seems difficult to obtain $L^\infty(0,T;L^6(\mathbb{R}^3))$-norm of $\rho_0^{-\frac{\gamma}{2}}\nabla b$
in terms of $\rho_0^{1-\frac{\gamma}{2}}\dot{u}$. Noticing that
\begin{align*}
\|\nabla (\rho_0^{-\frac{\gamma}{2}}b)\|_{L^6}
&\le C\|\rho_0^{-\frac{\gamma}{2}}b_t\|_{L^2}+C\|b\|_{L^4}\|\rho_0^{-\frac{\gamma}{2}}\nabla u\|_{L^4}
+C\|\rho_0^{-\frac{\gamma}{2}}u\|_{L^6}\|u\|_{L^6}\|\rho_0^{-\frac{\gamma}{2}}\nabla b\|_{L^2}\|\nabla b\|_{L^6}\nonumber\\
&\quad+C\|\rho_0^\frac{1-\gamma}{2}\nabla b\|_{L^2}
+C\|\nabla\rho_0^\frac{1-\gamma}{2}b\|_{L^2}.
\end{align*}
Thus, it suffices to estimate the $L^\infty(0,T;L^2(\mathbb{R}^3))$-norm of $\rho_0^{-\frac{\gamma}{2}}b_t$. To this end, we take advantage of the ideas used in \cite{LSX2016}. Roughly speaking, we may use the estimates of $\rho u_t$ and $\nabla u_t$ to control $b_t$. In fact, by virtue of the condition \eqref{h2} and integration by parts, we successfully deduce the estimates on $L^\infty(0,T)$-norm of
$\|(\rho_0^{-\frac{\gamma}{2}}b_t,\rho_0^{1-\frac{\gamma}{2}}u_t)\|_{L^2}$ and $L^2(0,T)$-norm of
$\|(\rho_0^{-\frac{\gamma}{2}}\nabla b_t,\rho_0^{1-\frac{\gamma}{2}}\nabla u_t)\|_{L^2}$ (see Lemma \ref{l26}), which is absent in \cite{LX23}. This
new singularly weighted estimate is very important to obtain the upper bound of iterations.
%More specifically, we have
%\begin{align*}
%\sup_{0\le t\le T}\Big\|\Big(\rho_0^{1-\frac{\gamma}{2}}u_t, \rho_0^{-\frac{\gamma}{2}}b_t\Big)\Big\|_{L^2}^2
%+\int_0^T\Big\|\Big(\rho_0^\frac{1-\gamma}{2}\nabla u_t, \rho_0^{-\frac{\gamma}{2}}\nabla b_t\Big)\Big\|_{L^2}^2dt\le C.
%\end{align*}
Finally, motivated by \cite{LX23}, with the help of all the \textit{a priori} estimates stated above and modified De Giorgi type iteration (see Appendix),
 we obtain the desired the estimates
on the lower and upper bounds of the entropy.

The rest of this paper is arranged as follow. In Section \ref{sec2}, we derive
important singularly weighted \textit{a priori} energy estimates. Section \ref{sec3} is devoted to carrying out the De Giorgi type iterations
with singular weights, which are used to prove Theorem \ref{thm1} in
Section \ref{sec4}.

\section{Singularly weighted \textit{a priori} estimates}\label{sec2}
This section is devoted to deriving some \textit{a priori} energy estimates with singular weights
which will be used in next section to carrying out suitable De Giorgi iterations. Throughout this section,
as well as the next one, we always assume that $(\rho, u, \theta, b)$ is a strong solution to system \eqref{a1}
in $\mathbb{R}^3\times (0, T)$, subject to \eqref{a2}, for a given positive time $T$. The initial density
is assumed to satisfy \eqref{h1}. For simplicity, we denote
\begin{align*}
\int \cdot dx=\int_{\mathbb{R}^3}\cdot dx.
\end{align*}
Moreover, we write
\begin{align}\label{2.1}
\phi(t)\triangleq1+\|\sqrt{\rho_0}u\|_{L^\infty}^2+\|\nabla u\|_{L^\infty}^2
+\|\nabla b\|_{L^\infty}^2+\|b\|_{L^\infty}^2, \ \Phi_T\triangleq\int_0^T\phi(t) dt.
\end{align}

The following lemma gives the lower and upper bounds of the density, whose proof can be found in \cite{LX22}.
\begin{lemma}\label{l21}
There exists a positive constant $C_*$ depending only on $K_1$ such that
\begin{align*}
e^{-C_*\Phi_T}\rho_0(x)\le \rho(x, t)\le e^{C_*\Phi_T}\rho_0(x), \quad \forall (x, t)\in \mathbb{R}^3\times (0, T).
\end{align*}
\end{lemma}

\begin{lemma}\label{l22}
For any $\alpha>0$, there exists a positive constant $C$ depending only on $\alpha$, $\mu$, $\lambda$, $\widetilde{R}$, $c_v$, $\kappa$, $\nu$, $K_1$, and $\Phi_T$ such that
\begin{align}\label{2.2}
&\sup_{0\le t\le T}\Big\|\Big(\rho_0^\frac{1-\alpha}{2}u, \rho_0^{1-\frac{\alpha}{2}}\theta,
 \rho_0^{-\frac{\alpha}{2}}b\Big)\Big\|_{L^2}^2
+\int_0^T\Big\|\Big(\rho_0^{-\frac{\alpha}{2}}\nabla u,
\sqrt{\rho_0}\rho_0^{-\frac{\alpha}{2}}\nabla\theta,
\rho_0^{-\frac{\alpha}{2}}\nabla b\Big)\Big\|_{L^2}^2dt\nonumber\\
&\le C\Big\|\Big(\rho_0^\frac{1-\alpha}{2}u_0,
\rho_0^{1-\frac{\alpha}{2}}\theta_0,
\rho_0^{-\frac{\alpha}{2}}b_0\Big)\Big\|_{L^2}^2.
\end{align}
\end{lemma}
\begin{proof}[Proof]
1. For any fixed $0<\delta<1$, set $\rho_{0\delta}=\rho_0+\delta$. Choose a nonnegative cut off function $\chi\in C_0^\infty(B_2)$
satisfying $\chi\equiv1$ on $B_1$ and set $\chi_R(x)=\chi(\frac{x}{R})$ for $R>0$. For $\alpha>0$, multiplying $\eqref{a1}_2$ by $\rho_{0\delta}^{-\alpha}u\chi_R^2$ and integrating the resulting equation over $\mathbb{R}^3$ yield that
\begin{align}\label{2.3}
&\frac12\frac{d}{dt}\int\rho|u|^2\rho_{0\delta}^{-\alpha}\chi_R^2dx+\mu\int|\nabla u|^2\rho_{0\delta}^{-\alpha}\chi_R^2dx
+(\mu+\lambda)\int(\divv u)^2\rho_{0\delta}^{-\alpha}\chi_R^2dx\nonumber\\
&=\frac{1}{2}\int\rho u\cdot\nabla\big(\rho_{0\delta}^{-\alpha}\chi_R^2\big)|u|^2dx
-\mu\int\partial_iu\cdot u\partial_i\big(\rho_{0\delta}^{-\alpha}\chi_R^2\big)dx
\nonumber\\
&\quad-(\mu+\lambda)\int\divv u u\cdot\nabla\big(\rho_{0\delta}^{-\alpha}\chi_R^2\big)dx-\int\nabla p\cdot\rho_{0\delta}^{-\alpha}u\chi_R^2dx\nonumber\\
&\quad+\int (b\cdot\nabla)b\cdot\rho_{0\delta}^{-\alpha}u\chi_R^2dx
-\frac12\int\nabla|b|^2\cdot\rho_{0\delta}^{-\alpha}u\chi_R^2dx
\triangleq\sum_{i=1}^6I_i.
\end{align}
Owing to
\begin{align}\label{bb}
|\nabla(\rho_{0\delta}^{-\alpha}\chi_R^2)|&=\Big|-\alpha\rho_{0\delta}^{-(\alpha+1)}
\nabla\rho_{0\delta}\chi_R^2+2\rho_{0\delta}^{-\alpha}\chi_R\nabla\chi_R\Big|\nonumber\\
&\le C\Big(\rho_{0\delta}^{-(\alpha+1)}\rho_0^\frac32\chi_R^2+(R\delta^\alpha)^{-1}\chi_R1_{\mathcal{C}_R}\Big)\nonumber\\
&\le C\big(\rho_{0\delta}^{-\alpha}\sqrt{\rho_0}\chi_R^2
+(R\delta^\alpha)^{-1}\chi_R1_{\mathcal{C}_R}\big),
\end{align}
where $1_{\mathcal C_R}$ is the characteristic function of the set $\mathcal C_R\triangleq B_{2R}\backslash B_R$, we derive from \eqref{h1} that
\begin{align}\label{2.4}
I_1&\le C\int\rho|u|^3\big(\rho_{0\delta}^{-\alpha}\sqrt{\rho_0}\chi_R^2
+(R\delta^\alpha)^{-1}\chi_R1_{\mathcal{C}_R}\big)dx\nonumber\\
&\le C\|\sqrt{\rho_0}u\|_{L^\infty}\|\sqrt{\rho}\rho_{0\delta}^{-\frac{\alpha}{2}}u\chi_R\|_{L^2}^2
+C(R\delta^\alpha)^{-1}\|u\|_{L^\infty}\|\sqrt{\rho}u\|_{L^2}.
\end{align}
It follows from \eqref{bb}, \eqref{h1}, H\"older's inequality, and Cauchy-Schwarz inequality that
\begin{align}
I_2&\le C\int|u||\nabla u|\big(\rho_{0\delta}^{-\alpha}\sqrt{\rho_0}\chi_R^2
+(R\delta^\alpha)^{-1}\chi_R1_{\mathcal{C}_R}\big)dx\nonumber\\
&\le C\|\rho_{0\delta}^{-\frac{\alpha}{2}}\nabla u\chi_R\|_{L^2}\|\sqrt{\rho_0}\rho_{0\delta}^{-\frac{\alpha}{2}}u\chi_R\|_{L^2}
+C(R\delta^\alpha)^{-1}\|u\|_{L^6}\|\nabla u\|_{L^2(\mathcal C_R)}\|\chi_R\|_{L^3}\nonumber\\
&\le \frac{\mu}{8}\|\rho_{0\delta}^{-\frac{\alpha}{2}}\nabla u\chi_R\|_{L^2}^2
+C\|\sqrt{\rho_0}\rho_{0\delta}^{-\frac{\alpha}{2}}u\chi_R\|_{L^2}^2
+C\delta^{-\alpha}\|\nabla u\|_{L^2}\|\nabla u\|_{L^2(\mathcal C_R)},
\end{align}
due to $\|\chi_R\|_{L^3}\le CR$. Similarly, one deduces that
\begin{align}
I_3&\le \frac{\mu+\lambda}{6}\|\rho_{0\delta}^{-\frac{\alpha}{2}}\divv u\chi_R\|_{L^2}^2
+C\|\sqrt{\rho_0}\rho_{0\delta}^{-\frac{\alpha}{2}}u\chi_R\|_{L^2}^2
+C\delta^{-\alpha}\|\nabla u\|_{L^2}\|\nabla u\|_{L^2(\mathcal C_R)}.
\end{align}
By virtue of \eqref{bb} and $\divv b=0$, we infer from integration by parts and Cauchy-Schwarz inequality that
\begin{align}
I_4&=\int p\big(\divv u\rho_{0\delta}^{-\alpha}\chi_R^2
+u\cdot\nabla(\rho_{0\delta}^{-\alpha}\chi_R^2)\big)dx\nonumber\\
&\le C\int\rho\theta\big[|\divv u|\rho_{0\delta}^{-\alpha}\chi_R^2
+|u|\big(\rho_{0\delta}^{-\alpha}\sqrt{\rho_0}\chi_R^2
+(R\delta^\alpha)^{-1}\chi_R\big)\big]dx\nonumber\\
&\le \frac{\mu+\lambda}{6}\|\divv u\rho_{0\delta}^{-\frac{\alpha}{2}}\chi_R\|_{L^2}^2
+C\Big(\|\rho\rho_{0\delta}^{-\frac{\alpha}{2}}\theta\chi_R\|_{L^2}^2
+\|\sqrt{\rho}\rho_{0\delta}^{-\frac{\alpha}{2}}u\chi_R\|_{L^2}^2\Big)\nonumber\\
&\quad+C(R\delta^\alpha)^{-1}(\|\sqrt{\rho}u\|_{L^2}^2
+\|\sqrt{\rho}\theta\|_{L^2}^2), \\
%\end{align}
%By , we obtain after a direct calculation that
%\begin{align}\label{2.9}
I_5&=\int b^i\partial_ib^ju^j\rho_{0\delta}^{-\alpha}\chi_R^2dx\nonumber\\
&=-\int\partial_ib^iu^j\rho_{0\delta}^{-\alpha}\chi_R^2dx-\int b^i\partial_iu^jb^j\rho_{0\delta}^{-\alpha}\chi_R^2dx
-\int b^ib^ju^j\partial_i\big(\rho_{0\delta}^{-\alpha}\chi_R^2\big)dx\nonumber\\
&=-\int b^i\partial_iu^jb^j\rho_{0\delta}^{-\alpha}\chi_R^2dx
-\int b^ib^ju^j\partial_i\big(\rho_{0\delta}^{-\alpha}\chi_R^2\big)dx\nonumber\\
&\le C\|\rho_{0\delta}^{-\frac{\alpha}{2}}\nabla u\chi_R\|_{L^2}\|\rho_{0\delta}^{-\frac{\alpha}{2}}b\chi_R\|_{L^2}
+C\int|b|^2|u|\big(\rho_{0\delta}^{-\alpha}\sqrt{\rho_0}\chi_R^2
+(R\delta^\alpha)^{-1}\chi_R\big)dx\nonumber\\
&\le \frac{\mu}{8}\|\rho_{0\delta}^{-\frac{\alpha}{2}}\nabla u\chi_R\|_{L^2}^2
+C\big(1+\|\sqrt{\rho_0}u\|_{L^\infty}\big)\|\rho_{0\delta}^{-\frac{\alpha}{2}}b\chi_R\|_{L^2}^2
+C(R\delta^\alpha)^{-1}\|u\|_{L^\infty}\|b\|_{L^2}^2\nonumber\\
&\le \frac{\mu}{8}\|\rho_{0\delta}^{-\frac{\alpha}{2}}\nabla u\chi_R\|_{L^2}^2
+C\phi(t)\|\rho_{0\delta}^{-\frac{\alpha}{2}}b\chi_R\|_{L^2}^2
+C(R\delta^\alpha)^{-1}\|u\|_{L^\infty}\|b\|_{L^2}^2,\\
%\end{align}
%Similarly to \eqref{2.9}, one can also derive that
%\begin{align}
I_6&=\frac{1}{2}\int|b|^2\divv u\rho_{0\delta}^{-\alpha}\chi_R^2dx
+\frac12\int|b|^2u\cdot\nabla\big(\rho_{0\delta}^{-\alpha}\chi_R^2\big)dx\nonumber\\
&\le \frac{1}{2}\int|b|^2\divv u\rho_{0\delta}^{-\alpha}\chi_R^2dx
+C\|\sqrt{\rho_0}u\|_{L^\infty}\|\rho_{0\delta}^{-\frac{\alpha}{2}}b\chi_R\|_{L^2}^2
+C\delta^{-\alpha}\|b\|_{L^\infty}\|\nabla u\|_{L^2}
\|b\|_{L^2(\mathcal C_R)}\nonumber\\
&\le \frac{\mu+\lambda}{6}\|\divv u\rho_{0\delta}^{-\frac{\alpha}{2}}\chi_R\|_{L^2}^2
+C\phi(t)\|\rho_{0\delta}^{-\frac{\alpha}{2}}b\chi_R\|_{L^2}^2
+C(R\delta^\alpha)^{-1}\|u\|_{L^\infty}\|b\|_{L^2}^2. \label{2.10}
\end{align}
Putting \eqref{2.4}--\eqref{2.10} into \eqref{2.3} and applying Lemma \ref{l21}, we arrive at
\begin{align}\label{2.11}
&\frac12\frac{d}{dt}\int\rho|u|^2\rho_{0\delta}^{-\alpha}\chi_R^2dx+\frac{3\mu}{4}\int|\nabla u|^2\rho_{0\delta}^{-\alpha}\chi_R^2dx
+\frac{\mu+\lambda}{2}\int(\divv u)^2\rho_{0\delta}^{-\alpha}\chi_R^2dx\nonumber\\
&\le C\phi(t)\Big(\|\rho\rho_{0\delta}^{-\frac{\alpha}{2}}\theta\chi_R\|_{L^2}^2
+\|\sqrt{\rho}\rho_{0\delta}^{-\frac{\alpha}{2}}u\chi_R\|_{L^2}^2
+\|\rho_{0\delta}^{-\frac{\alpha}{2}}b\chi_R\|_{L^2}^2\Big)\nonumber\\
&\quad+C(R\delta^\alpha)^{-1}\big(\|\sqrt{\rho}u\|_{L^2}^2+\|\sqrt{\rho}\theta\|_{L^2}^2
+\|u\|_{L^\infty}\|\sqrt{\rho}u\|_{L^2}^2+\|u\|_{L^\infty}\|b\|_{L^2}^2\big)\nonumber\\
&\quad+C\delta^{-\alpha}\|\nabla u\|_{L^2}\|\nabla u\|_{L^2(\mathcal C_R)}.
\end{align}

2. Multiplying $\eqref{a1}_4$ by $\rho_{0\delta}^{-\alpha}b\chi_R^2$ and integrating the resultant over $\mathbb{R}^3$, we have
\begin{align}\label{2.12}
&\frac12\frac{d}{dt}\int|b|^2\rho_{0\delta}^{-\alpha}\chi_R^2dx+\nu\int|\nabla b|^2\rho_{0\delta}^{-\alpha}\chi_R^2dx\nonumber\\
&=-\nu\int\partial_ib\cdot b\partial_i\big(\rho_{0\delta}^{-\alpha}\chi_R^2\big)dx
+\int (b\cdot\nabla) u\cdot\rho_{0\delta}^{-\alpha}b\chi_R^2dx\nonumber\\
&\quad-\int (u\cdot\nabla) b\cdot\rho_{0\delta}^{-\alpha}b\chi_R^2dx-\int \divv u b\cdot\rho_{0\delta}^{-\alpha}b\chi_R^2dx\nonumber\\
& \triangleq J_1+\int (b\cdot\nabla) u\cdot\rho_{0\delta}^{-\alpha}b\chi_R^2dx
+J_2-\int\divv u|b|^2\rho_{0\delta}^{-\alpha}\chi_R^2dx.
\end{align}
By \eqref{bb} and H{\"o}lder's inequality, one gets that
\begin{align}\label{2.13}
J_1&\le C\int|b||\nabla b|\big(\rho_{0\delta}^{-\alpha}\sqrt{\rho_0}\chi_R^2+(R\delta^\alpha)^{-1}
\chi_R1_{\mathcal{C}_R}\big)dx\nonumber\\
&\le C\|\sqrt{\rho_0}\|_{L^\infty}\|\rho_{0\delta}^{-\frac{\alpha}{2}}b\chi_R\|_{L^2}
\|\rho_{0\delta}^{-\frac{\alpha}{2}}\nabla b\chi_R\|_{L^2}
+C(R\delta^\alpha)^{-1}\|\nabla b\|_{L^2(\mathcal C_R)}\|b\|_{L^6}\|\chi_R\|_{L^3}\nonumber\\
&\le \frac{\nu}{2}\|\rho_{0\delta}^{-\frac{\alpha}{2}}\nabla b\chi_R\|_{L^2}^2+C\|\rho_{0\delta}^{-\frac{\alpha}{2}}b\chi_R\|_{L^2}^2
+C\delta^{-\alpha}\|\nabla b\|_{L^2}\|\nabla b\|_{L^2(\mathcal C_R)}.
\end{align}
Due to $\divv b=0$, we obtain from integration by parts that
\begin{align*}
-\int (u\cdot \nabla) b\cdot b\rho_{0\delta}^{-\alpha}\chi_R^2dx
=\int\divv u|b|^2\rho_{0\delta}^{-\alpha}\chi_R^2dx
+\int (u\cdot\nabla)b\cdot b\rho_{0\delta}^{-\alpha}\chi_R^2dx
+\int|b|^2 u\cdot\nabla\big(\rho_{0\delta}^{-\alpha}\chi_R^2\big)dx,
\end{align*}
which combined with \eqref{bb} yields that
\begin{align}\label{2.14}
J_2&=\frac12\int\divv u|b|^2\rho_{0\delta}^{-\alpha}\chi_R^2dx
+\frac12\int |b|^2 u\cdot\nabla\big(\rho_{0\delta}^{-\alpha}\chi_R^2\big)dx\nonumber\\
&\le \frac12\int\divv u|b|^2\rho_{0\delta}^{-\alpha}\chi_R^2dx
+C\int |u||b|^2\big(\rho_{0\delta}^{-\alpha}\sqrt{\rho_0}\chi_R^2
+(R\delta^\alpha)^{-1}\chi_R1_{\mathcal{C}_R}\big)dx\nonumber\\
&\le \frac12\int\divv u|b|^2\rho_{0\delta}^{-\alpha}\chi_R^2dx
+C\|\sqrt{\rho_0}u\|_{L^\infty}\|\rho_{0\delta}^{-\frac{\alpha}{2}}b\chi_R\|_{L^2}^2
+C(R\delta^\alpha)^{-1}\|u\|_{L^\infty}\|b\|_{L^2}^2.
\end{align}
Inserting \eqref{2.13} and \eqref{2.14} into \eqref{2.12}, we arrive at
\begin{align*}
&\frac12\frac{d}{dt}\int|b|^2\rho_{0\delta}^{-\alpha}\chi_R^2dx+\frac{\nu}{2}\int|\nabla b|^2\rho_{0\delta}^{-\alpha}\chi_R^2dx\nonumber\\
&\le -\frac{1}{2}\int|b|^2\divv u\rho_{0\delta}^{-\alpha}\chi_R^2dx
+C\|\nabla u\|_{L^\infty}\|\rho_{0\delta}^{-\frac{\alpha}{2}}b\chi_R\|_{L^2}^2
+C\delta^{-\alpha}\|\nabla b\|_{L^2}\|\nabla b\|_{L^2(\mathcal C_R)}\nonumber\\
&\quad
+C(1+\|\sqrt{\rho_0}u\|_{L^\infty})\|\rho_{0\delta}^{-\frac{\alpha}{2}}b\chi_R\|_{L^2}^2
+C(R\delta^\alpha)^{-1}\|u\|_{L^\infty}\|b\|_{L^2}^2\nonumber\\
&\le C\phi(t)\|\rho_{0\delta}^{-\frac{\alpha}{2}}b\chi_R\|_{L^2}^2
+C\delta^{-\alpha}\|\nabla b\|_{L^2}\|\nabla b\|_{L^2(\mathcal C_R)}
+C(R\delta^\alpha)^{-1}\|u\|_{L^\infty}\|b\|_{L^2}^2.
\end{align*}
This together with \eqref{2.11} leads to
\begin{align}\label{2.16}
&\frac12\frac{d}{dt}\int\big(\rho|u|^2\rho_{0\delta}^{-\alpha}\chi_R^2
+|b|^2\rho_{0\delta}^{-\alpha}\chi_R^2\big)dx+\mu\int|\nabla u|^2\rho_{0\delta}^{-\alpha}\chi_R^2dx\nonumber\\
&\quad+(\mu+\lambda)\int(\divv u)^2\rho_{0\delta}^{-\alpha}\chi_R^2dx
+\frac{3\nu}{4}\int|\nabla b|^2\rho_{0\delta}^{-\alpha}\chi_R^2dx\nonumber\\
&\le C\phi(t)\Big(\|\rho\rho_{0\delta}^{-\frac{\alpha}{2}}\theta\chi_R\|_{L^2}^2
+\|\sqrt{\rho}\rho_{0\delta}^{-\frac{\alpha}{2}}u\chi_R\|_{L^2}^2
+\|\rho_{0\delta}^{-\frac{\alpha}{2}}b\chi_R\|_{L^2}^2\Big)\nonumber\\
&\quad+C(R\delta^\alpha)^{-1}\big(\|\sqrt{\rho}u\|_{L^2}^2+\|\sqrt{\rho}\theta\|_{L^2}^2
+\|u\|_{L^\infty}\|\sqrt{\rho}u\|_{L^2}^2+\|u\|_{L^\infty}\|b\|_{L^2}^2\big)\nonumber\\
&\quad+C\delta^{-\alpha}\big(\|\nabla u\|_{L^2}\|\nabla u\|_{L^2(\mathcal C_R)}
+\|\nabla b\|_{L^2}\|\nabla b\|_{L^2(\mathcal C_R)}\big).
\end{align}

3. Multiplying $\eqref{a1}_3$ by $\rho_0\rho_{0\delta}^{-\alpha}\theta\chi_R^2$ and integrating the resultant over $\mathbb{R}^3$ lead to
\begin{align}\label{2.17}
&\frac{c_v}{2}\frac{d}{dt}\int\rho\rho_0\rho_{0\delta}^{-\alpha}\theta^2\chi_R^2dx
+\kappa\int\rho_0\rho_{0\delta}^{-\alpha}|\nabla\theta|^2\chi_R^2dx\nonumber\\
&=\frac{c_v}{2}\int\rho u\theta^2\nabla\big(\rho_0\rho_{0\delta}^{-\alpha}\chi_R^2\big)dx
-\int\theta\nabla\theta\cdot\nabla\big(\rho_0\rho_{0\delta}^{-\alpha}\chi_R^2\big)dx
-\int p\divv u\rho_0\rho_{0\delta}^{-\alpha}\theta^2\chi_R^2dx\nonumber\\
&\quad
+\int\mathcal Q(\nabla u)\rho_0\rho_{0\delta}^{-\alpha}\theta^2\chi_R^2dx
+\nu\int|\curl b|^2\rho_0\rho_{0\delta}^{-\alpha}\theta^2\chi_R^2dx
\triangleq\sum_{i=1}^5K_i.
\end{align}
Due to
\begin{align*}
|\nabla(\rho_0\rho_{0\delta}^{-\alpha}\chi_R^2)|&=\Big|\nabla\rho_0\rho_{0\delta}^{-\alpha}\chi_R^2
-\alpha\rho_0\rho_{0\delta}^{-(\alpha+1)}\nabla\rho_0\chi_R^2+2\rho_0\rho_{0\delta}^{-\alpha}\chi_R\nabla\chi_R\Big|\nonumber\\
&\le C\rho_0^\frac32\rho_{0\delta}^{-\alpha}\chi_R^2+CR^{-1}\rho_0\rho_{0\delta}^{-\alpha}\chi_R1_{\mathcal C_R},
\end{align*}
it follows from \eqref{h1} and Cauchy-Schwarz inequality that
\begin{align*}
K_1&\le C\int\rho|u|\theta^2\Big(\rho_0^\frac32\rho_{0\delta}^{-\alpha}\chi_R^2
+\rho_0(R\delta^\alpha)^{-1}\chi_R1_{\mathcal C_R}\Big)dx\nonumber\\
&\le C\|\sqrt{\rho_0}u\|_{L^\infty}\|\sqrt{\rho\rho_0}\rho_{0\delta}^{-\frac{\alpha}{2}}\theta\chi_R\|_{L^2}^2
+C(R\delta^\alpha)^{-1}\|\rho_0u\|_{L^\infty}\|\sqrt{\rho}\theta\|_{L^2}\nonumber\\
&\le C\phi(t)\|\sqrt{\rho\rho_0}\rho_{0\delta}^{-\frac{\alpha}{2}}\theta\chi_R\|_{L^2}^2
+C(R\delta^\alpha)^{-1}\phi(t)\|\sqrt{\rho}\theta\|_{L^2}^2,\\
%\end{align*}
%Thanks to \eqref{2.18}, by , we have
%\begin{align*}
K_2&\le C\int\theta|\nabla\theta|\Big(\rho_0^\frac32\rho_{0\delta}^{-\alpha}\chi_R^2
+\rho_0R^{-1}\rho_0\rho_{0\delta}^{-\alpha}\chi_R1_{\mathcal C_R}\Big)dx\nonumber\\
&\le \frac{\kappa}{2}\|\sqrt{\rho_0}\rho_{0\delta}^{-\frac{\alpha}{2}}\nabla\theta\chi_R\|_{L^2}^2
+C\|\rho_0\rho_{0\delta}^{-\frac{\alpha}{2}}\theta\chi_R\|_{L^2}^2+C(R^2\delta^\alpha)^{-1}\|\sqrt{\rho_0}\theta\|_{L^2}^2,\\
K_3&\le \widetilde{R}\int\rho\theta|\divv u|\rho_0\rho_{0\delta}^{-\alpha}\theta\chi_R^2dx\\
&\le C\|\nabla u\|_{L^\infty}\|\sqrt{\rho\rho_0}\rho_{0\delta}^{-\alpha}\theta\chi_R\|_{L^2}^2\nonumber\\
&\le C\phi(t)\|\sqrt{\rho\rho_0}\rho_{0\delta}^{-\alpha}\theta\chi_R\|_{L^2}^2,\\
K_4&\le C\int|\nabla u|^2\rho_0\rho_{0\delta}^{-\alpha}\theta\chi_R^2dx\nonumber\\
&\le C\|\nabla u\|_{L^\infty}\|\rho_{0\delta}^{-\frac{\alpha}{2}}\nabla u\chi_R\|_{L^2}
\|\rho_0\rho_{0\delta}^{-\frac{\alpha}{2}}\theta\chi_R^2\|_{L^2}\nonumber\\
&\le \frac{\mu}{4}\|\rho_{0\delta}^{-\frac{\alpha}{2}}\nabla u\chi_R\|_{L^2}^2
+C\phi(t)\|\rho_0\rho_{0\delta}^{-\frac{\alpha}{2}}\theta\chi_R^2\|_{L^2}^2,\\
K_5&\le C\int|\nabla b|^2\rho_0\rho_{0\delta}^{-\alpha}\theta\chi_R^2dx\nonumber\\
&\le C\|\nabla u\|_{L^\infty}\|\rho_{0\delta}^{-\frac{\alpha}{2}}\nabla u\chi_R\|_{L^2}
\|\rho_0\rho_{0\delta}^{-\frac{\alpha}{2}}\theta\chi_R^2\|_{L^2}\nonumber\\
&\le \frac{\mu}{4}\|\rho_{0\delta}^{-\frac{\alpha}{2}}\nabla b\chi_R\|_{L^2}^2
+C\phi(t)\|\rho_0\rho_{0\delta}^{-\frac{\alpha}{2}}\theta\chi_R^2\|_{L^2}^2.
\end{align*}
Substituting the above estimates on $K_i\ (i=1, 2, \ldots, 5)$ into \eqref{2.17}, we derive from Lemma \ref{l21} that
\begin{align*}
&\frac{c_v}{2}\frac{d}{dt}\int\rho\rho_0\rho_{0\delta}^{-\alpha}\theta^2\chi_R^2dx
+\frac{\kappa}{2}\int\rho_0\rho_{0\delta}^{-\alpha}|\nabla\theta|^2\chi_R^2dx\nonumber\\
&\le \frac{\mu}{4}\|\rho_{0\delta}^{-\frac{\alpha}{2}}\nabla u\chi_R\|_{L^2}^2
+\frac{\mu}{4}\|\rho_{0\delta}^{-\frac{\alpha}{2}}\nabla b\chi_R\|_{L^2}^2
+C\phi(t)\|\sqrt{\rho\rho_0}\rho_{0\delta}^{-\frac{\alpha}{2}}\theta\chi_R\|_{L^2}^2\nonumber\\
&\quad+C\big((R\delta^\alpha)^{-1}(\phi(t)+(R^2\delta^\alpha)^{-1}\big)
\|\sqrt{\rho}\theta\|_{L^2}^2,
\end{align*}
which along with \eqref{2.16} and Lemma \ref{l21}  gives that
\begin{align}\label{2.18}
&\frac{d}{dt}\Big(\|\sqrt{\rho}\rho_{0\delta}^{-\frac{\alpha}{2}}u\chi_R\|_{L^2}^2
+c_v\|\sqrt{\rho\rho_0}\rho_{0\delta}^{-\frac{\alpha}{2}}\theta\chi_R\|_{L^2}^2
+\|\rho_{0\delta}^{-\frac{\alpha}{2}}b\chi_R\|_{L^2}^2\Big)\nonumber\\
&\quad+\|\rho_{0\delta}^{-\frac{\alpha}{2}}\nabla u\chi_R\|_{L^2}^2
+\|\rho_{0\delta}^{-\frac{\alpha}{2}}\nabla b\chi_R\|_{L^2}^2
+\|\sqrt{\rho_0}\rho_{0\delta}^{-\frac{\alpha}{2}}\nabla\theta\chi_R\|_{L^2}^2\nonumber\\
&\le C\phi(t)\Big(c_v\|\rho\rho_{0\delta}^{-\frac{\alpha}{2}}\theta\chi_R\|_{L^2}^2
+\|\sqrt{\rho}\rho_{0\delta}^{-\frac{\alpha}{2}}u\chi_R\|_{L^2}^2
+\|\rho_{0\delta}^{-\frac{\alpha}{2}}b\chi_R\|_{L^2}^2\Big)\nonumber\\
&\quad+C(R\delta^\alpha)^{-1}\big(\|\sqrt{\rho}u\|_{L^2}^2
+\|u\|_{L^\infty}\|\sqrt{\rho}u\|_{L^2}^2+\|u\|_{L^\infty}\|b\|_{L^2}^2\big)\nonumber\\
&\quad+C\delta^{-\alpha}\big(\|\nabla u\|_{L^2}\|\nabla u\|_{L^2(\mathcal C_R)}
+\|\nabla b\|_{L^2}\|\nabla b\|_{L^2(\mathcal C_R)}\big)\nonumber\\
&\quad+C\big((R\delta^\alpha)^{-1}(\phi(t)+(R^2\delta^\alpha)^{-1}\big)\|\sqrt{\rho}\theta\|_{L^2}^2.
\end{align}
Applying Gronwall's inequality to \eqref{2.18} and noticing that
\begin{align*}
&\int_0^T\big(\|\nabla u\|_{L^2}\|\nabla u\|_{L^2(\mathcal C_R)}
+\|\nabla b\|_{L^2}\|\nabla b\|_{L^2(\mathcal C_R)}\big)dt\rightarrow 0, \quad \text{as}~R\rightarrow \infty, \\
&(1+\|u\|_{L^\infty})\|\sqrt{\rho}u\|_{L^2}^2
+\|u\|_{L^\infty}\|b\|_{L^2}^2+\phi(t)\|\sqrt{\rho}\theta\|_{L^2}^2\in L^1(0, T),
\end{align*}
guaranteed by the regularities of $(\rho, u, \theta, b)$, one obtains after taking $R\rightarrow\infty$ and applying Lemma \ref{l21} that
\begin{align*}
&\sup_{0\le t\le T}\Big(\|\sqrt{\rho_0}\rho_{0\delta}^{-\frac{\alpha}{2}}u\|_{L^2}^2
+c_v\|\rho_0\rho_{0\delta}^{-\frac{\alpha}{2}}\theta\|_{L^2}^2
+\|\rho_{0\delta}^{-\frac{\alpha}{2}}b\|_{L^2}^2\Big)\nonumber\\
&\quad+\int_0^T\Big(\|\rho_{0\delta}^{-\frac{\alpha}{2}}\nabla u\|_{L^2}^2
+\|\rho_{0\delta}^{-\frac{\alpha}{2}}\nabla b\|_{L^2}^2
+\|\sqrt{\rho_0}\rho_{0\delta}^{-\frac{\alpha}{2}}\nabla\theta\|_{L^2}^2\Big)dt
\nonumber\\
&\le C\Big\|\Big(\rho_0^\frac{1-\alpha}{2}u_0,
\rho_0^{1-\frac{\alpha}{2}}\theta_0,
\rho_0^{-\frac{\alpha}{2}}b_0\Big)\Big\|_{L^2}^2.
\end{align*}
Thus, taking $\delta\rightarrow 0$ and using Lebesgue's dominated convergence theorem, one derives \eqref{2.2}.
\end{proof}

\begin{lemma}\label{l23}
There exists a positive constant $C$ depending only on $\mu$, $\lambda$, $\widetilde{R}$, $c_v$, $\kappa$, $\nu$, $K_1$, and $\Phi_T$ such that
\begin{align}
&\sup_{0\le t\le T}\Big\|\Big(\rho_0^{-\frac{\gamma}{2}}\nabla u,
\rho_0^{-\frac{\gamma+2}{2}}\nabla b, \rho_0^{-\frac{\gamma+2}{2}}|b|^2\Big)\Big\|_{L^2}^2
+\int_0^T\Big\|\Big(\rho_0^\frac{1-\gamma}{2}u_t, \rho_0^{-\frac{\gamma+2}{2}}\nabla^2b\Big)\Big\|_{L^2}^2dt\nonumber\\
&\le C\Big\|\Big(\rho_0^\frac{1-\gamma}{2}u_0,
\rho_0^{1-\frac{\gamma}{2}}\theta_0,
\rho_0^{-\frac{\gamma}{2}}b_0, \rho_0^{-\frac{\gamma}{2}}\nabla u_0,
\rho_0^{-\frac{\gamma+2}{2}}\nabla b_0\Big)\Big\|_{L^2}^2.
\end{align}
\end{lemma}
\begin{proof}[Proof]
1. Multiplying $\eqref{a1}_2$ by $\rho_0^{-\gamma}u_t$ and integrating the resulting equation over $\mathbb{R}^3$ lead to
\begin{align}\label{2.25}
&\frac12\frac{d}{dt}\int\rho_0^{-\gamma}\big[\mu|\nabla u|^2+(\mu+\lambda)(\divv u)^2\big]dx+\int\rho\rho_0^{-\gamma}|u_t|^2dx\nonumber\\
&=-\int(\nabla\rho_0^{-\gamma}\cdot\nabla)u\cdot u_tdx-\int\divv u u_t\cdot
\nabla\rho_0^{-\gamma}dx
-\int\rho u\cdot\nabla u\cdot\rho_0^{-\gamma}u_tdx\nonumber\\
&\quad-\int\nabla p\rho_0^{-\gamma}u_tdx+\int b\cdot\nabla b\cdot\rho_0^{-\gamma}u_tdx
-\frac{1}{2}\int\nabla|b|^2\cdot\rho_0^{-\gamma}u_tdx
\triangleq\sum_{i=1}^6L_i.
\end{align}
It follows from integrating by parts, \eqref{h1}, Lemma \ref{l21}, and Cauchy-Schwarz inequality that
\begin{align*}
L_1+L_2 & \le C\int|\nabla u||u_t|\rho_0^{-\gamma-1}|\nabla\rho_0|dx \\
& \le C\int\rho_0^{-\frac{\gamma}{2}}|\nabla u|\rho_0^{\frac12-\frac{\gamma}{2}}|u_t|dx
\le \frac{1}{12}\|\rho_0^\frac{1-\gamma}{2}u_t\|_{L^2}^2+C\|\rho_0^{-\frac{\gamma}{2}}\nabla u\|_{L^2}^2.
\end{align*}
Using Lemma \ref{l21} and Cauchy-Schwarz inequality, we find that
\begin{align*}
L_3\le \frac{1}{24}\|\sqrt{\rho}\rho_0^{-\frac{\gamma}{2}}u_t\|_{L^2}^2+C\|\sqrt{\rho_0}u\|_{L^\infty}^2\|\rho_0^{-\frac{\gamma}{2}}\nabla u\|_{L^2}^2
\le \frac{1}{24}\|\sqrt{\rho}\rho_0^{-\frac{\gamma}{2}}u_t\|_{L^2}^2+C\phi(t)\|\rho_0^{-\frac{\gamma}{2}}\nabla u\|_{L^2}^2.
\end{align*}
Due to $\divv b=0$, we obtain from integration by parts and Cauchy-Schwarz inequality that
\begin{align*}
L_5&=\int b^i\partial_ib^j\rho_0^{-\gamma}u_tdx=\int\partial_i(b^ib^j)\rho_0^{-\gamma}u_t^jdx\nonumber\\
&=-\int b^ib^j\partial_i\rho_0^{-\gamma}u_t^jdx-\int b^ib^j\rho_0^{-\gamma}\partial_iu_t^jdx\nonumber\\
&=-\int b^ib^j\partial_i\rho_0^{-\gamma}u_t^jdx+\int b_t^ib^j\rho_0^{-\gamma}\partial_iu^jdx
+\int b^ib_t^j\rho_0^{-\gamma}\partial_iu^jdx\nonumber\\
&=-\int b^ib^j\partial_i\rho_0^{-\gamma}u_t^jdx
+\int(b\cdot\nabla u-u\cdot\nabla b+b\divv u+\nu\Delta b)\otimes b:\rho_0^{-\gamma}\nabla udx\nonumber\\
&\quad+\int b\otimes(b\cdot\nabla u-u\cdot\nabla b+b\divv u+\nu\Delta b):\rho_0^{-\gamma}\nabla udx\nonumber\\
&\le C\int|b|^2\rho_0^{-\gamma}\sqrt{\rho_0}u_tdx+C\int|b|^2|\nabla u|^2\rho_0^{-\gamma}dx
+C\int\sqrt{\rho_0}|u||\nabla u||b||\nabla b|\rho_0^{-\gamma-\frac12}dx\nonumber\\
&\le C\|b\|_{L^\infty}\|\rho_0^{-\frac{\gamma}{2}}b\|_{L^2}\|\sqrt{\rho_0}\rho_0^{-\frac{\gamma}{2}}u_t\|_{L^2}
+C\|b\|_{L^\infty}\|\rho_0^{-\frac{\gamma}{2}}b\|_{L^2}\|\rho_0^{-\frac{\gamma}{2}}\nabla u\|_{L^2}\nonumber\\
&\quad+C\|\sqrt{\rho_0}u\|_{L^\infty}\|\rho^{-\frac{\gamma}{2}}\nabla u\|_{L^2}\||\nabla b||b|\rho^{-\frac{\gamma+1}{2}}\|_{L^2}\nonumber\\
&\le \frac{1}{24}\|\sqrt{\rho_0}\rho_0^{-\frac{\gamma}{2}}u_t\|_{L^2}^2
+C\|b\|_{L^\infty}^2\|\rho_0^{-\frac{\gamma}{2}}b\|_{L^2}^2
+C(1+\|\sqrt{\rho_0}u\|_{L^\infty}^2)\|\rho_0^{-\frac{\gamma}{2}}\nabla u\|_{L^2}^2\nonumber\\
&\quad+C\|\rho^{-\frac{\gamma+1}{2}}|\nabla b||b|\|_{L^2}^2\nonumber\\
&\le \frac{1}{24}\|\sqrt{\rho_0}\rho_0^{-\frac{\gamma}{2}}u_t\|_{L^2}^2
+C\phi(t)\Big(\|\rho_0^{-\frac{\gamma}{2}}b\|_{L^2}^2+\|\rho_0^{-\frac{\gamma}{2}}\nabla u\|_{L^2}^2\Big)
+C\|\rho^{-\frac{\gamma+1}{2}}|\nabla b||b|\|_{L^2}^2.
\end{align*}
Similarly, we have
\begin{align*}
L_6\le \frac{1}{24}\|\sqrt{\rho_0}\rho_0^{-\frac{\gamma}{2}}u_t\|_{L^2}^2
+C\phi(t)\Big(\|\rho_0^{-\frac{\gamma}{2}}b\|_{L^2}^2+\|\rho_0^{-\frac{\gamma}{2}}\nabla u\|_{L^2}^2\Big)
+C\|\rho^{-\frac{\gamma+1}{2}}|\nabla b||b|\|_{L^2}^2.
\end{align*}
For the estimate on $L_4$. Let $F$ be the effective viscous flux, i.e.,
\begin{align}\label{z2.31}
F=(2\mu+\lambda)\divv u-p-\frac12|b|^2.
\end{align}
It follows from $\eqref{a1}_1$, $\eqref{a1}_3$, and the equation of state that
\begin{align*}
-p_t=\divv(up-\kappa(\gamma-1)\nabla\theta)+(\gamma-1)\big(\divv up-\mathcal Q(\nabla u)-|\curl b|^2\big).
\end{align*}
This together with \eqref{z2.31}, integration by parts, and Lemma \ref{l21} gives that
\begin{align*}
L_4&=\int p(\nabla\rho_0^{-\gamma}\cdot u_t+\rho_0^{-\gamma}\divv u_t)dx\nonumber\\
&=\frac{d}{dt}\int\rho_0^{-\gamma}\divv updx+\int(p\nabla\rho_0^{-\gamma}\cdot u_t-\rho_0^{-\gamma}\divv u p_t)dx\nonumber\\
&=\frac{d}{dt}\int\Big(\rho_0^{-\gamma}\divv up-\frac{1}{2(2\mu+\lambda)}\rho_0^{-\gamma}p^2\Big)dx
+\int p\nabla\rho_0^{-\gamma}\cdot u_tdx\nonumber\\
&\quad-\frac{1}{2\mu+\lambda}\int\rho_0^{-\gamma}Fp_tdx-\frac{1}{2(2\mu+\lambda)}\int\rho_0^{-\gamma}|b|^2p_tdx\nonumber\\
&=\frac{d}{dt}\int\Big(\rho_0^{-\gamma}\divv up-\frac{1}{2(2\mu+\lambda)}\rho_0^{-\gamma}p^2\Big)dx
+\int p\nabla\rho_0^{-\gamma}\cdot u_tdx\nonumber\\
&\quad+\frac{1}{2\mu+\lambda}\int\rho_0^{-\gamma}\Big(F+\frac12|b|^2\Big)\divv(up-k(\gamma-1)\nabla\theta)dx\nonumber\\
&\quad+\frac{\gamma-1}{2\mu+\lambda}\int\rho_0^{-\gamma}\Big(F+\frac12|b|^2\Big)(\divv up-\mathcal Q(\nabla u)-|\curl b|^2)dx\nonumber\\
&=\frac{d}{dt}\int\Big(\rho_0^{-\gamma}\divv up-\frac{1}{2(2\mu+\lambda)}\rho_0^{-\gamma}p^2\Big)dx
+\int p\nabla\rho_0^{-\gamma}\cdot u_tdx\nonumber\\
&\quad+\int((\gamma-1)\kappa\nabla\theta-up)\cdot\Big(\rho_0^{-\gamma}
(\nabla F+\frac12\nabla|b|^2)+\nabla\rho_0^{-\gamma}(F+\frac12|b|^2)\Big)dx\nonumber\\
&\quad+\frac{\gamma-1}{2\mu+\lambda}\int\rho_0^{-\gamma}((2\mu+\lambda)\divv u-p)\big(\divv up-\mathcal Q(\nabla u)-|\curl b|^2\big)dx\nonumber\\
&\le \frac{d}{dt}\int\Big(\rho_0^{-\gamma}\divv up-\frac{1}{2(2\mu+\lambda)}\rho_0^{-\gamma}p^2\Big)dx
+C\int\rho\theta\rho_0^{-\gamma}\sqrt{\rho_0}|u_t|dx\nonumber\\
&\quad+C\int\rho_0^{-\gamma}|\nabla F||\nabla\theta|dx
+C\int\rho_0^{-\gamma}\rho\theta|u||\nabla F|dx+C\int\rho_0^{-\gamma}|b||\nabla b||\nabla\theta|dx\nonumber\\
&\quad+C\int\rho_0^{-\gamma}\rho\theta|u||b||\nabla b|dx
+C\int\rho_0^{-\gamma}|\nabla u|^2\rho\theta dx+C\int\rho_0^{-\gamma}|\nabla u|^3dx\nonumber\\
&\quad+C\int\rho_0^{-\gamma}|\nabla u||\nabla b|^2dx+C\int\rho_0^{-\gamma}\rho^2\theta^2|\nabla u|dx
+C\int\rho_0^{-\gamma}\rho\theta|\nabla u|^2dx
+C\int\rho_0^{-\gamma}\rho\theta|\nabla b|^2dx\nonumber\\
&\le  \frac{d}{dt}\int\Big(\rho_0^{-\gamma}\divv up-\frac{1}{2(2\mu+\lambda)}\rho_0^{-\gamma}p^2\Big)dx
\nonumber\\
&\quad+\frac{1}{24}\|\sqrt{\rho}\rho_0^{-\frac{\gamma}{2}}u_t\|_{L^2}^2+C\|\rho_0^{1-\frac{\gamma}{2}}\theta\|_{L^2}^2
+\eta\|\rho_0^{-\frac{\gamma+1}{2}}\nabla F\|_{L^2}^2+C\|\rho_0^\frac{1-\gamma}{2}\nabla\theta\|_{L^2}^2\nonumber\\
&\quad+C\|\sqrt{\rho_0}u\|_{L^\infty}^2\|\rho\rho_0^{-\frac{\gamma}{2}}\theta\|_{L^2}^2
+C\|\rho_0^{-\frac{\gamma+1}{2}}|b||\nabla b|\|_{L^2}^2+C\|\nabla b\|_{L^\infty}^2\|\rho_0^{-\frac{\gamma}{2}}\nabla b\|_{L^2}^2\nonumber\\
&\quad+C(1+\|\nabla u\|_{L^\infty}^2)\|\rho_0^{-\frac{\gamma}{2}}\nabla u\|_{L^2}^2
+C\|\nabla u\|_{L^\infty}\|\rho\rho_0^{-\frac{\gamma}{2}}\theta\|_{L^2}^2\nonumber\\
&\le \frac{d}{dt}\int\Big(\rho_0^{-\gamma}\divv up-\frac{1}{2(2\mu+\lambda)}\rho_0^{-\gamma}p^2\Big)dx
+C\|\rho_0^\frac{1-\gamma}{2}\nabla\theta\|_{L^2}^2\nonumber\\
&\quad+\frac{1}{24}\|\sqrt{\rho}\rho_0^{-\frac{\gamma}{2}}u_t\|_{L^2}^2+\eta\|\rho_0^{-\frac{\gamma+1}{2}}\nabla F\|_{L^2}^2
+C\phi(t)\|\rho_0^{1-\frac{\gamma}{2}}\theta\|_{L^2}^2\nonumber\\
&\quad+C\phi(t)(\|\rho_0^{-\frac{\gamma}{2}}\nabla u\|_{L^2}^2
+\|\rho_0^{-\frac{\gamma}{2}}\nabla b\|_{L^2}^2)+C\|\rho_0^{-\frac{\gamma+1}{2}}|b||\nabla b|\|_{L^2}^2.
\end{align*}
Putting the above estimates on $L_i\ (i=1, 2,\cdots, 6)$ into \eqref{2.25} and using Lemma \ref{l21}, one obtains that, for any positive $\eta$,
\begin{align}\label{2.27}
&\frac{d}{dt}\Big(\mu\|\rho_0^{-\frac{\gamma}{2}}\nabla u\|_{L^2}^2
+(\mu+\lambda)\|\rho_0^{-\frac{\gamma}{2}}\divv u\|_{L^2}^2
+\frac{1}{2\mu+\lambda}\|\rho_0^{-\frac{\gamma}{2}}p\|_{L^2}^2\Big)\nonumber\\
&\quad-2\frac{d}{dt}\int\rho_0^{-\gamma}\divv updx+\frac32\|\sqrt{\rho}\rho_0^{-\frac{\gamma}{2}}u_t\|_{L^2}^2\nonumber\\
&\le 2\eta\|\rho_0^{-\frac{\gamma+1}{2}}\nabla F\|_{L^2}^2
+C\phi(t)\Big(\|\rho_0^{1-\frac{\gamma}{2}}\theta\|_{L^2}^2
+\|\rho_0^{-\frac{\gamma}{2}}\nabla u\|_{L^2}^2
+\|\rho_0^{-\frac{\gamma}{2}}\nabla b\|_{L^2}^2\Big)\nonumber\\
&\quad+C\phi(t)\|\rho_0^{-\frac{\gamma}{2}}b\|_{L^2}^2+C\|\rho_0^\frac{1-\gamma}{2}\nabla\theta\|_{L^2}^2
+C\|\rho^{-\frac{\gamma+1}{2}}|\nabla b||b|\|_{L^2}^2.
\end{align}

2. Multiplying $\eqref{a1}_4$ by $\rho_0^{-(\gamma+2)}\Delta b$ and integrating the resultant over $\mathbb{R}^3$ lead to
\begin{align}\label{2.28}
&\frac12\frac{d}{dt}\int|\nabla b|^2\rho_0^{-(\gamma+2)}dx+\nu\int|\Delta b|^2\rho_0^{-(\gamma+2)}dx\nonumber\\
&=\int(u\cdot\nabla b+b\divv u-b\cdot\nabla u)\cdot\Delta b\rho_0^{-(\gamma+2)}dx
-\int b_t\partial_ib\partial_i\rho_0^{-(\gamma+2)}dx\nonumber\\
&=\int\big(u\cdot\nabla b+b\divv u-b\cdot\nabla u\big)\cdot\Delta b\rho_0^{-(\gamma+2)}dx\nonumber\\
&\quad-\int(b\cdot\nabla u-u\cdot\nabla b-b\divv u+\nu\Delta b)\partial_ib\partial_i\rho_0^{-(\gamma+2)}dx\nonumber\\
&=\int u^i\partial_ib\partial_{jj}b\rho_0^{-(\gamma+2)}dx+C\int|b||\nabla u||\Delta b|\rho_0^{-(\gamma+2)}dx\nonumber\\
&\quad+C\int\big(|b||\nabla u|+|u||\nabla b|+|\Delta b|\big)|\nabla b|
\rho_0^{-(\gamma+2)}\sqrt{\rho_0}dx\nonumber\\
&\le \frac{\nu}{2}\int|\Delta b|^2\rho_0^{-(\gamma+2)}dx
+\int u^i\partial_ib\partial_{jj}b\rho_0^{-(\gamma+2)}dx
+C\|\nabla u\|_{L^\infty}^2\int|b|^2\rho_0^{-(\gamma+2)}dx\nonumber\\
&\quad+C\big(1+\|\sqrt{\rho_0}u\|_{L^\infty}\big)\int|\nabla b|^2\rho_0^{-(\gamma+2)}dx\nonumber\\
&\le \frac{\nu}{2}\int|\Delta b|^2\rho_0^{-(\gamma+2)}dx
+C\phi(t)\int|\nabla b|^2\rho_0^{-(\gamma+2)}dx+C\phi(t)\int|b|^2\rho_0^{-(\gamma+2)}dx,
\end{align}
where we have used the following fact
\begin{align}\label{z2.28}
\int u^i\partial_ib\partial_{jj}b\rho_0^{-(\gamma+2)}dx
&=-\int\partial_ju^i\partial_ib\partial_jb\rho_0^{-(\gamma+2)}dx
-\int u\partial_{ji}b\partial_jb\rho_0^{-(\gamma+2)}dx
-\int u^i\partial_ib\partial_jb\partial_j\rho_0^{-(\gamma+2)}dx\nonumber\\
&=-\int\partial_ju^i\partial_ib\partial_jb\rho_0^{-(\gamma+2)}dx
+\frac12\int \divv u|\nabla b|^2\rho_0^{-(\gamma+2)}dx\nonumber\\
&\quad+\frac12\int u|\nabla b|^2
\cdot\nabla\rho_0^{-(\gamma+2)}dx
-\int u^i\partial_ib\partial_jb\partial_j\rho_0^{-(\gamma+2)}dx\nonumber\\
&\le C\int|\nabla u||\nabla b|^2\rho_0^{-(\gamma+2)}dx
+C\int|u||\nabla b|^2\rho_0^{-(\gamma+2)}\sqrt{\rho_0}dx\nonumber\\
&\le  C\big(\|\nabla u\|_{L^\infty}+\|\sqrt{\rho_0}u\|_{L^\infty}\big)\int|\nabla b|^2\rho_0^{-(\gamma+2)}dx\nonumber\\
&\le C\phi(t)\int|\nabla b|^2\rho_0^{-(\gamma+2)}dx.
\end{align}
Integration by parts together with \eqref{h1} yields that
\begin{align*}
\int|\nabla^2b|^2\rho_0^{-(\gamma+2)}dx
&=\int\partial_{ij}b\cdot\partial_{ij}b\rho_0^{-(\gamma+2)}dx\nonumber\\
&=\int|\Delta b|^2\rho_0^{-(\gamma+2)}dx-\int\partial_{ij}b\cdot\partial_jb
\partial_i\rho_0^{-(\gamma+2)}dx
+\int\partial_{ii}b\cdot\partial_jb\partial_j\rho_0^{-(\gamma+2)}dx\nonumber\\
&\le \int|\Delta b|^2\rho_0^{-(\gamma+2)}dx
+C\int|\nabla^2b||\nabla b|\rho_0^{-(\gamma+2)}\sqrt{\rho_0}dx\nonumber\\
&\le \int|\Delta b|^2\rho_0^{-(\gamma+2)}dx+\frac12\int|\nabla^2b|^2\rho_0^{-(\gamma+2)}dx
+C\int|\nabla b|^2\rho_0^{-(\gamma+2)}dx,
\end{align*}
that is
\begin{align}\label{2.29}
\|\rho_0^{-\frac{\gamma+2}{2}}\nabla^2b\|_{L^2}^2\le 2\|\rho_0^{-\frac{\gamma+2}{2}}\Delta b\|_{L^2}^2
+C\|\rho_0^{-\frac{\gamma+2}{2}}\nabla b\|_{L^2}^2.
\end{align}
Combining \eqref{2.28} and \eqref{2.29}, we derive that
\begin{align*}
\frac{d}{dt}\|\rho_0^{-\frac{\gamma+2}{2}}\nabla b\|_{L^2}^2+\|\rho_0^{-\frac{\gamma+2}{2}}\nabla^2b\|_{L^2}^2
\le C\phi(t)\Big(\|\rho_0^{-\frac{\gamma+2}{2}}\nabla b\|_{L^2}^2+\|\rho_0^{-\frac{\gamma+2}{2}}b\|_{L^2}^2\Big).
\end{align*}
This together with \eqref{2.27} leads to
\begin{align}\label{2.31}
&\frac{d}{dt}\Big(\mu\|\rho_0^{-\frac{\gamma}{2}}\nabla u\|_{L^2}^2
+(\mu+\lambda)\|\rho_0^{-\frac{\gamma}{2}}\divv u\|_{L^2}^2+\|\rho_0^{-\frac{\gamma+2}{2}}\nabla b\|_{L^2}^2
+\frac{1}{2\mu+\lambda}\|\rho_0^{-\frac{\gamma}{2}}p\|_{L^2}^2\Big)\nonumber\\
&\quad-2\frac{d}{dt}\int\rho_0^{-\gamma}\divv updx+\frac32\|\sqrt{\rho}\rho_0^{-\frac{\gamma}{2}}u_t\|_{L^2}^2
+\|\rho_0^{-\frac{\gamma+2}{2}}\nabla^2b\|_{L^2}^2\nonumber\\
&\le 2\eta\|\rho_0^{-\frac{\gamma+1}{2}}\nabla F\|_{L^2}^2
+C\phi(t)\Big(\|\rho_0^{1-\frac{\gamma}{2}}\theta\|_{L^2}^2
+\|\rho_0^{-\frac{\gamma}{2}}\nabla u\|_{L^2}^2
+\|\rho_0^{-\frac{\gamma+1}{2}}\nabla b\|_{L^2}^2\Big)\nonumber\\
&\quad+C\phi(t)\|\rho_0^{-\frac{\gamma}{2}}b\|_{L^2}^2+C\|\rho_0^\frac{1-\gamma}{2}\nabla\theta\|_{L^2}^2
+C\|\rho^{-\frac{\gamma+1}{2}}|\nabla b||b|\|_{L^2}^2.
\end{align}

3. We need to estimate $\|\rho_0^{-\frac{\gamma+1}{2}}\nabla F\|_{L^2}^2$. It follows from $\eqref{a1}_2$ that $F$ satisfies
\begin{align}\label{2.32}
\Delta F=\divv(\rho\dot{u})+\divv(b\otimes b).
\end{align}
%where $\dot{u}\triangleq u_t+u\cdot\nabla u$.
Multiplying \eqref{2.32} by $\rho_0^{-(\gamma+1)}F$ and
integrating the resultant over $\mathbb{R}^3$, we then infer from \eqref{h1} and Cauchy-Schwarz inequality that
\begin{align}
\|\rho_0^{-\frac{\gamma+1}{2}}\nabla F\|_{L^2}^2&=-\int F\nabla\rho_0^{-(\gamma+1)}\cdot\nabla Fdx
+\int\divv(\rho\dot{u})\rho_0^{-(\gamma+1)}Fdx
+\int\divv(b\otimes b)\rho_0^{-(\gamma+1)}Fdx\nonumber\\
&=-\int F\rho_0^{-(\gamma+2)}\nabla\rho_0\cdot\nabla Fdx
-\int\rho\dot{u}\cdot\big(\rho_0^{-(\gamma+1)}\nabla F+\nabla\rho_0^{-(\gamma+1)}F\big)dx\nonumber\\
&\quad-\int b\otimes b:\big(\rho_0^{-(\gamma+1)}\nabla F+\nabla\rho_0^{-(\gamma+1)}F\big)dx\nonumber\\
&\le C\int\rho_0^{-\frac{2\gamma+1}{2}}|F||\nabla F|dx
+C\int\rho\rho_0^{-(\gamma+1)}|u_t||\nabla F|dx\nonumber\\
&\quad+C\int\rho\rho_0^{-(\gamma+1)}|u||\nabla u||\nabla F|dx
+C\int\rho\rho_0^{-\frac{2\gamma+1}{2}}|u_t||F|dx\nonumber\\
&\quad+C\int\rho|u||\nabla u|\rho_0^{-\frac{2\gamma+1}{2}}|F|dx
+C\int|b|^2\rho_0^{-(\gamma+1)}|\nabla F|dx
+C\int|b|^2\rho_0^{-\frac{2\gamma+1}{2}}|F|dx\nonumber\\
&\le \frac12\|\rho_0^{-\frac{\gamma+1}{2}}\nabla F\|_{L^2}^2+C\|\rho_0^{-\frac{\gamma}{2}}F\|_{L^2}^2
+C\|\sqrt{\rho}\rho_0^{-\frac{\gamma}{2}}u_t\|_{L^2}^2\nonumber\\
&\quad+C\|\sqrt{\rho_0}u\|_{L^\infty}^2\|\rho_0^{-\frac{\gamma}{2}}\nabla u\|_{L^2}^2
+C\|\rho_0^{-\frac{\gamma+1}{2}}|b|^2\|_{L^2}^2\nonumber\\
&\le \frac12\|\rho_0^{-\frac{\gamma+1}{2}}\nabla F\|_{L^2}^2
+C\phi(t)\Big(\|\rho_0^{-\frac{\gamma}{2}}\nabla u\|_{L^2}^2+\|\rho_0^{1-\frac{\gamma}{2}}\theta\|_{L^2}^2
+\|\rho_0^{-\frac{\gamma}{2}}b\|_{L^2}^2\Big)\nonumber\\
&\quad+C\|\sqrt{\rho}\rho_0^{-\frac{\gamma}{2}}u_t\|_{L^2}^2
+C\|\rho_0^{-\frac{\gamma+1}{2}}|b|^2\|_{L^2}^2,
\end{align}
due to
\begin{align*}
\|\rho_0^{-\frac{\gamma}{2}}F\|_{L^2}^2&\le C\|\rho_0^{-\frac{\gamma}{2}}\nabla u\|_{L^2}^2
+C\|\rho\rho_0^{-\frac{\gamma}{2}}\theta\|_{L^2}^2+C\|\rho_0^{-\frac{\gamma}{2}}|b|^2\|_{L^2}^2\nonumber\\
&\le C\|\rho_0^{-\frac{\gamma}{2}}\nabla u\|_{L^2}^2+C\|\rho_0^{1-\frac{\gamma}{2}}\theta\|_{L^2}^2
+C\|b\|_{L^\infty}^2\|\rho_0^{-\frac{\gamma}{2}}b\|_{L^2}^2.
\end{align*}
Thus, it is not hard to see that
\begin{align}\label{2.34}
\|\rho_0^{-\frac{\gamma+1}{2}}\nabla F\|_{L^2}^2
&\le C\phi(t)\big(\|\rho_0^{-\frac{\gamma}{2}}\nabla u\|_{L^2}^2+\|\rho_0^{1-\frac{\gamma}{2}}\theta\|_{L^2}^2
+\|\rho_0^{-\frac{\gamma}{2}}b\|_{L^2}^2\big) \nonumber\\
& \quad +C\|\sqrt{\rho}\rho_0^{-\frac{\gamma}{2}}u_t\|_{L^2}^2
+C\|\rho_0^{-\frac{\gamma+1}{2}}|b|^2\|_{L^2}^2.
\end{align}
Then, by \eqref{2.31} and \eqref{2.34} and choosing $\eta$ suitably small, we have
\begin{align}
&\frac{d}{dt}\Big(\mu\|\rho_0^{-\frac{\gamma}{2}}\nabla u\|_{L^2}^2
+(\mu+\lambda)\|\rho_0^{-\frac{\gamma}{2}}\divv u\|_{L^2}^2+\|\rho_0^{-\frac{\gamma+2}{2}}\nabla b\|_{L^2}^2
+\frac{1}{2\mu+\lambda}\|\rho_0^{-\frac{\gamma}{2}}p\|_{L^2}^2\Big)\nonumber\\
&\quad+\|\sqrt{\rho}\rho_0^{-\frac{\gamma}{2}}u_t\|_{L^2}^2
+\|\rho_0^{-\frac{\gamma+2}{2}}\nabla^2b\|_{L^2}^2\nonumber\\
&\le 2\frac{d}{dt}\int\rho_0^{-\gamma}\divv updx
+C\phi(t)\Big(\|\rho_0^{1-\frac{\gamma}{2}}\theta\|_{L^2}^2
+\|\rho_0^{-\frac{\gamma}{2}}\nabla u\|_{L^2}^2
+\|\rho_0^{-\frac{\gamma+2}{2}}\nabla b\|_{L^2}^2\Big)\nonumber\\
&\quad+C\phi(t)\|\rho_0^{-\frac{\gamma}{2}}b\|_{L^2}^2+C\|\rho_0^\frac{1-\gamma}{2}\nabla\theta\|_{L^2}^2
+C\|\rho^{-\frac{\gamma+2}{2}}|\nabla b||b|\|_{L^2}^2.
\end{align}
We infer from Lemma \ref{l22}, Lemma \ref{l21}, and  \eqref{h1} that
\begin{align}
\sup_{0\le t\le T}\|\rho_0^{-\frac{\gamma}{2}}p\|_{L^2}^2
\le C\sup_{0\le t\le T}\|\rho \rho_0^{-\frac{\gamma}{2}}\theta\|_{L^2}^2\le
C\Big\|\Big(\rho_0^\frac{1-\gamma}{2}u_0,
\rho_0^{1-\frac{\gamma}{2}}\theta_0,
\rho_0^{-\frac{\gamma}{2}}b_0\Big)\Big\|_{L^2}^2,
\end{align}
and
\begin{align}
\Big|\int\rho_0^{-\gamma}\divv updx\Big|\le \frac{\mu+\lambda}{2}\|\rho_0^{-\frac{\gamma}{2}}\divv u\|_{L^2}^2
+C\|\rho_0^{-\frac{\gamma}{2}}p\|_{L^2}^2.
\end{align}

4. Multiplying $\eqref{a1}_4$ by $4|b|^2b\rho_0^{-(\gamma+2)}$ and integration the resultant over $\mathbb{R}^3$, we get from \eqref{h1} and Sobolev's inequality that
\begin{align}\label{2.23}
&\frac{d}{dt}\int|b|^4\rho_0^{-(\gamma+2)}dx+4\int\big(|\nabla b|^2|b|^2+2|\nabla|b||^2|b|^2\big)\rho_0^{-(\gamma+2)}dx\nonumber\\
&=4\int\big(b\cdot\nabla u-u\cdot\nabla b-b\divv u\big)\cdot|b|^2b\rho_0^{-(\gamma+2)}dx\nonumber\\
&=4\int (b\cdot\nabla) u\cdot b|b|^2\rho_0^{-(\gamma+2)}dx-3\int|b|^4\divv u\rho_0^{-(\gamma+2)}dx\nonumber\\
&\quad-\int\partial_ib\cdot|b|^2b\partial_i\rho_0^{-(\gamma+2)}dx
+\int|b|^4 u\cdot\nabla\rho_0^{-(\gamma+2)}dx\nonumber\\
&\le C\|\nabla u\|_{L^\infty}\int|b|^4\rho_0^{-(\gamma+2)}dx
+C\int|u||b|^4\rho_0^{-(\gamma+2)}\sqrt{\rho_0}dx
+C\int|b|^3|\nabla b|\rho_0^{-(\gamma+2)}\sqrt{\rho_0}dx\nonumber\\
&\le C(1+\|\nabla u\|_{L^\infty}+\|\sqrt{\rho_0}u\|_{L^\infty}
)\int|b|^4\rho_0^{-(\gamma+2)}dx
+2\int|\nabla b|^2|b|^2\rho_0^{-(\gamma+2)}dx\nonumber\\
&\le 2\int|\nabla b|^2|b|^2\rho_0^{-(\gamma+2)}dx
+C\phi(t)\int|b|^4\rho_0^{-(\gamma+2)}dx,
\end{align}
due to
\begin{align*}
& \int\Delta b\cdot|b|^2b\rho_0^{-(\gamma+2)}dx \notag \\
&=\int\partial_{ii}b\cdot|b|^2b\rho_0^{-(\gamma+2)}dx\nonumber\\
&=-\int\partial_ib\cdot\partial_i|b|^2b\rho_0^{-(\gamma+2)}dx
-\int\partial_ib\cdot|b|^2\partial_ib\rho_0^{-(\gamma+2)}dx
-\int\partial_ib\cdot|b|^2b\partial_i\rho_0^{-(\gamma+2)}dx\nonumber\\
&=-2\int|\nabla|b||^2|b|^2\rho_0^{-(\gamma+2)}dx-\int|\nabla b|^2|b|^2\rho_0^{-(\gamma+2)}dx
-\int\partial_ib\cdot|b|^2b\partial_i\rho_0^{-(\gamma+2)}dx,
\end{align*}
and
\begin{align*}
-\int (u\cdot\nabla) b\cdot|b|^2b\rho_0^{-(\gamma+2)}dx
&=-\int u^i\partial_ib\cdot|b|^2b\rho_0^{-(\gamma+2)}dx\nonumber\\
&=\int\divv u|b|^4\rho_0^{-(\gamma+2)}dx
+\int (u\cdot\nabla)b\cdot|b|^2b\rho_0^{-(\gamma+2)}dx\nonumber\\
&\quad+2\int (u\cdot\nabla)b\cdot|b|^2b\rho_0^{-(\gamma+2)}dx
+\int|b|^4 u\cdot\nabla\rho_0^{-(\gamma+2)}dx.
\end{align*}
Thus, we infer from \eqref{2.23} and Gronwall's inequality that
\begin{align}\label{2.33}
\sup_{0\le t\le T}\Big\|\rho_0^{-\frac{\gamma+2}{2}}|b|^2\Big\|_{L^2}^2
+\int_0^T\Big\|\Big(\rho_0^{-\frac{\gamma+2}{2}}|\nabla b||b|,
\rho_0^{-\frac{\gamma+2}{2}}|\nabla|b|||b|\Big)\Big\|_{L^2}^2dt
\le C\Big\|\rho_0^{-\frac{\gamma+2}{2}}|b_0|^2\Big\|_{L^2}^2.
\end{align}
Thanks to the above inequalities, and applying Gronwall's inequality, the conclusion follows from Lemmas \ref{l21}--\ref{l22}.
\end{proof}

\begin{lemma}\label{l24}
There exists a positive constant $C$ depending only on $\mu$, $\lambda$, $\widetilde{R}$, $c_v$, $\kappa$, $\nu$, $K_1$, and $\Phi_T$ such that
\begin{align}\label{z3.24}
&\sup_{0\le t\le T}\|\rho_0^{1-\frac{\gamma}{2}}\nabla\theta\|_{L^2}^2
+\int_0^T\|\rho_0^\frac{3-\gamma}{2}\theta_t\|_{L^2}^2dt\nonumber\\
&\le C\Big\|\Big(\rho_0^\frac{1-\gamma}{2}u_0,
\rho_0^{1-\frac{\gamma}{2}}\theta_0,
\rho_0^{-\frac{\gamma}{2}}b_0, \rho_0^{-\frac{\gamma+2}{2}}|b_0|^2, \rho_0^{-\frac{\gamma}{2}}\nabla u_0,
\rho_0^{-\frac{\gamma+2}{2}}\nabla b_0, \rho_0^{-\frac{\gamma}{2}}\nabla\theta_0\Big)\Big\|_{L^2}^2.
\end{align}
\end{lemma}
\begin{proof}[Proof]
Multiplying $\eqref{a1}_3$ by $\rho_0^{2-\gamma}\theta_t$ and integrating the resultant over $\mathbb{R}^3$, we have
\begin{align*}
&\frac{\kappa}{2}\frac{d}{dt}\|\rho_0^{1-\frac{\gamma}{2}}\nabla\theta\|_{L^2}^2
+c_v\|\sqrt{\rho}\rho_0^{1-\frac{\gamma}{2}}\theta_t\|_{L^2}^2\nonumber\\
&=\kappa\int\nabla\theta\cdot\nabla\rho_0^{2-\gamma}\theta_tdx
+\int(\mathcal Q(\nabla u)+|\curl b|^2-p\divv u)\rho_0^{2-\gamma}\theta_tdx
-c_v\int\rho u\cdot\nabla\theta\rho_0^{2-\gamma}\theta_tdx\nonumber\\
&\le \frac{c_v}{2}\|\sqrt{\rho}\rho_0^{1-\frac{\gamma}{2}}\theta_t\|_{L^2}^2
+C\|\rho_0^{1-\frac{\gamma}{2}}\nabla\theta\|_{L^2}^2+C\|\nabla u\|_{L^\infty}^2\|\rho_0^\frac{1-\gamma}{2}\nabla u\|_{L^2}^2
+C\|\nabla b\|_{L^\infty}^2\|\rho_0^\frac{1-\gamma}{2}\nabla b\|_{L^2}^2\nonumber\\
&\quad +C\|\nabla u\|_{L^\infty}^2\|\rho_0^\frac{3-\gamma}{2}\theta\|_{L^2}^2
+C\|\sqrt{\rho_0}u\|_{L^\infty}^2\|\rho_0^{1-\frac{\gamma}{2}}\nabla\theta\|_{L^2}^2\nonumber\\
&\le \frac{c_v}{2}\|\sqrt{\rho}\rho_0^{1-\frac{\gamma}{2}}\theta_t\|_{L^2}^2
+C\phi(t)\Big\|\Big(\rho_0^\frac{1-\gamma}{2}\nabla u, \rho_0^{-\frac{\gamma+2}{2}}\nabla b,
\rho_0^\frac{3-\gamma}{2}\theta, \rho_0^{1-\frac{\gamma}{2}}\nabla\theta\Big)\Big\|_{L^2}^2.
\end{align*}
Thus, we get that
\begin{align*}
\kappa\frac{d}{dt}\|\rho_0^{1-\frac{\gamma}{2}}\nabla\theta\|_{L^2}^2
+c_v\|\sqrt{\rho}\rho_0^{1-\frac{\gamma}{2}}\theta_t\|_{L^2}^2
\le C\phi(t)\Big\|\Big(\rho_0^\frac{1-\gamma}{2}\nabla u, \rho_0^{-\frac{\gamma+2}{2}}\nabla b,
\rho_0^\frac{3-\gamma}{2}\theta, \rho_0^{1-\frac{\gamma}{2}}\nabla\theta\Big)\Big\|_{L^2}^2,
\end{align*}
which combined with Gronwall's inequality and Lemmas \ref{l21}--\ref{l23} yields \eqref{z3.24}.
\end{proof}

\begin{lemma}\label{l25}
There exists a positive constant $C$ depending only on $\mu$, $\lambda$, $\widetilde{R}$, $c_v$, $\kappa$, $\nu$, $K_1$, and $\Phi_T$ such that
\begin{align}\label{2.40}
&\sup_{0\le t\le T}\Big(\|\rho_0^{1-\frac{\gamma}{2}}\dot{u}\|_{L^2}^2+\|\rho_0^{-\frac{\gamma+2}{2}}|\nabla b||b|\|_{L^2}^2\Big)
+\int_0^T\Big(\|\rho_0^\frac{1-\gamma}{2}\nabla\dot{u}\|_{L^2}^2+\|\rho_0^{-\frac{\gamma+2}{2}}|\Delta b||b|\|_{L^2}^2\Big)dt\nonumber\\
&\le C\Big\|\Big(\rho_0^\frac{1-\gamma}{2}u_0,
\rho_0^{1-\frac{\gamma}{2}}\theta_0,
\rho_0^{-\frac{\gamma}{2}}b_0, \rho_0^{-\frac{\gamma+2}{2}}|b_0|^2, \rho_0^{-\frac{\gamma}{2}}\nabla u_0,
\rho_0^{-\frac{\gamma+2}{2}}\nabla b_0, \rho_0^{-\frac{\gamma}{2}}\nabla\theta_0, \rho_0^{-\frac{\gamma}{2}}\wp_0\Big)\Big\|_{L^2}^2,
\end{align}
where
\begin{align*}
\wp_0\triangleq\mu\Delta u_0+(\mu+\lambda)\nabla \divv u_0-\widetilde{R}\nabla(\rho_0\theta_0)+b_0\cdot\nabla b_0-\frac12\nabla|b_0|^2.
\end{align*}
\end{lemma}
\begin{proof}[Proof]
1. Motivated by \cite{LSX2016}, for $a_1$, $a_2$, $a_3$ $\in\{-1, 0, 1\}$, we denote by
\begin{align*}
\tilde{b}(a_1, a_2, a_3)=a_1b^1+a_2b^2+a_3b^3, \quad \tilde{u}(a_1, a_2, a_3)=a_1u^1+a_2u^2+a_3u^3.
\end{align*}
Then it follows from $\eqref{a1}_4$ that
\begin{align}\label{2.41}
\tilde{b}_t-\nu\Delta\tilde{b}= b\cdot\nabla\tilde{u}-u\cdot\nabla\tilde{b}-\tilde{b}\divv u.
\end{align}
Multiplying \eqref{2.41} by $\rho_0^{-(\gamma+2)}\tilde{b}\Delta|\tilde{b}|^2$ and integrating the resulting equation over $\mathbb{R}^3$, we derive that
\begin{align}\label{z2.42}
& -\int\tilde{b}_t\cdot\rho_0^{-(\gamma+2)}\tilde{b}\Delta|\tilde{b}|^2dx
+\nu\int\Delta\tilde{b}\cdot\rho_0^{-(\gamma+2)}\tilde{b}\Delta|\tilde{b}|^2dx \notag \\
& =-\int(b\cdot\nabla\tilde{u}-\tilde{b}\divv u)\cdot\rho_0^{-(\gamma+2)}\tilde{b}\Delta|\tilde{b}|^2dx
+\int u\cdot\nabla \tilde{b}\cdot\rho^{-(\gamma+2)}\tilde{b}\Delta|\tilde{b}|^2dx.
\end{align}
Integration by parts gives that
\begin{align}\label{2.42}
& -\int\tilde{b}_t\cdot\rho_0^{-(\gamma+2)}\tilde{b}\Delta|\tilde{b}|^2dx \notag \\
&=-\frac12\int\partial_t|\tilde{b}|^2\partial_{ii}|\tilde{b}|^2\rho_0^{-(\gamma+2)}dx\nonumber\\
&=\frac12\int\partial_t(\partial_i|\tilde{b}|^2)\partial_i|\tilde{b}|^2\rho_0^{-(\gamma+2)}dx
+\int\tilde{b}_t\cdot\tilde{b}\partial_i|\tilde{b}|^2\partial_i\rho_0^{-(\gamma+2)}dx\nonumber\\
&=\frac14\frac{d}{dt}\int|\nabla|\tilde{b}|^2|^2\rho_0^{-(\gamma+2)}dx
+\int(\nu\Delta\tilde{b}+b\cdot\nabla\tilde{u}-u\cdot\nabla\tilde{b}-\tilde{b}
\divv u)\cdot\tilde{b}\partial_i|\tilde{b}|^2\partial_i\rho_0^{-(\gamma+2)}dx.
\end{align}
Owing to $\Delta\tilde{b}\cdot\tilde{b}=\frac12(\Delta|\tilde{b}|^2-2|\nabla\tilde{b}|^2)$, we have
\begin{align}\label{2.43}
\nu\int\Delta\tilde{b}\cdot\rho_0^{-(\gamma+2)}\tilde{b}\Delta|\tilde{b}|^2dx
=\frac{\nu}{2}\int|\Delta|\tilde{b}|^2|^2\rho_0^{-(\gamma+2)}dx
-\nu\int|\nabla\tilde{b}|^2\Delta|\tilde{b}|^2\rho_0^{-(\gamma+2)}dx.
\end{align}
Inserting \eqref{2.42} and \eqref{2.43} into \eqref{z2.42} implies that
\begin{align}\label{z2.43}
&\frac14\frac{d}{dt}\|\rho_0^{-\frac{\gamma+2}{2}}\nabla|\tilde{b}|^2\|_{L^2}^2
+\frac{\nu}{2}\|\rho_0^{-\frac{\gamma+2}{2}}\Delta|\tilde{b}|^2\|_{L^2}^2\nonumber\\
& = -\int(\nu\Delta\tilde{b}+b\cdot\nabla\tilde{u}-u\cdot\nabla\tilde{b}-\tilde{b}
\divv u)\cdot\tilde{b}\partial_i|\tilde{b}|^2\partial_i\rho_0^{-(\gamma+2)}dx
+\nu\int|\nabla\tilde{b}|^2\Delta|\tilde{b}|^2\rho_0^{-(\gamma+2)}dx \notag \\
& \quad -\int(b\cdot\nabla\tilde{u}-\tilde{b}\divv u)\cdot\rho_0^{-(\gamma+2)}\tilde{b}\Delta|\tilde{b}|^2dx
+\int u\cdot\nabla \tilde{b}\cdot\rho^{-(\gamma+2)}\tilde{b}\Delta|\tilde{b}|^2dx
\triangleq \bar{J}_1+\bar{J}_2+\bar{J}_3+\bar{J}_4.
\end{align}
It follows from integration by parts and Cauchy-Schwarz inequality that
\begin{align*}
\bar{J}_1&=-\frac{\nu}{2}\int\partial_i^2|\tilde{b}|^2\partial_i|\tilde{b}|^2
\partial_i\rho_0^{-(\gamma+2)}dx
+\nu\int|\partial_i|\tilde{b}||^2\partial_i|\tilde{b}|^2\partial_i\rho_0^{-(\gamma+2)}dx
\nonumber\\
&\quad+\int b_i\partial_i\tilde{u}\tilde{b}\partial_i|\tilde{b}|^2
\partial_i\rho_0^{-(\gamma+2)}dx
-\int|\tilde{b}|^2\partial_iu^i\partial_i|\tilde{b}|^2\partial_i\rho_0^{-(\gamma+2)}dx-\frac12\int u^i\partial_i|\tilde{b}|^2\partial_i|\tilde{b}|^2\partial_i\rho_0^{-(\gamma+2)}dx
\nonumber\\
&\le \frac{\nu}{16}\|\rho_0^{-\frac{\gamma+2}{2}}\Delta|\tilde{b}|^2\|_{L^2}^2
+C\|\rho_0^{-\frac{\gamma+2}{2}}\nabla|\tilde{b}|^2\|_{L^2}^2\nonumber\\
&\quad+C\|\nabla\tilde{b}\|_{L^\infty}\|\rho_0^{-\frac{\gamma+2}{2}}\nabla\tilde{b}\|_{L^2}
\|\rho_0^{-\frac{\gamma+2}{2}}\nabla|\tilde{b}|^2\|_{L^2}
+C\|\sqrt{\rho_0}u\|_{L^\infty}\|\rho_0^{-\frac{\gamma+2}{2}}\nabla|\tilde{b}|^2\|_{L^2}^2\nonumber\\
&\quad+C\|\nabla\tilde{u}\|_{L^\infty}\|\rho_0^{-\frac{\gamma+2}{2}}|\tilde{b}|^2\|_{L^2}
\|\rho_0^{-\frac{\gamma+2}{2}}\nabla\tilde{b}\|_{L^2},\\
\bar{J}_2&\le C\|\nabla \tilde{b}\|_{L^\infty}\|\rho_0^{-\frac{\gamma+2}{2}}\nabla\tilde{b}\|_{L^2}
\|\rho_0^{-\frac{\gamma+2}{2}}\Delta|\tilde{b}|^2\|_{L^2} \notag \\
&\le \frac{\nu}{12}\|\rho_0^{-\frac{\gamma+2}{2}}\Delta|\tilde{b}|^2\|_{L^2}^2
+C\|\nabla\tilde{b}\|_{L^\infty}^2\|\rho_0^{-\frac{\gamma+2}{2}}\nabla\tilde{b}\|_{L^2}^2,\\
\bar{J}_3&\le C\|\nabla u\|_{L^\infty}\int|b|^2\rho_0^{-\frac{\gamma+2}{2}}
|\Delta|\tilde{b}|^2|\rho_0^{-\frac{\gamma+2}{2}}dx\nonumber\\
&\le \frac{\nu}{12}\|\Delta|\tilde{b}|^2\rho_0^{-\frac{\gamma+2}{2}}\|_{L^2}^2
+C\|\nabla u\|_{L^\infty}^2\||b|^2\rho_0^{-\frac{\gamma+2}{2}}\|_{L^2}^2,\\
\bar{J}_4&=\frac12\int u^i\partial_i|\tilde{b}|^2\partial_{jj}|\tilde{b}|^2\rho_0^{-(\gamma+2)}dx\nonumber\\
&=-\frac12\int\partial_ju^i\partial_i|\tilde{b}|^2\partial_j|\tilde{b}|^2\rho_0^{-(\gamma+2)}dx
-\frac12\int u^i\partial_j\partial_i|\tilde{b}|^2\partial_j|\tilde{b}|^2\rho_0^{-(\gamma+2)}dx\nonumber\\
&\quad-\frac12\int u^i\partial_i|\tilde{b}|^2\partial_j|\tilde{b}|^2\partial_j\rho_0^{-(\gamma+2)}dx\nonumber\\
&=-\frac12\int\partial_ju^i\partial_i|\tilde{b}|^2\partial_j|\tilde{b}|^2\rho_0^{-(\gamma+2)}dx
-\frac12\int u^i\partial_i|\tilde{b}|^2\partial_j|\tilde{b}|^2\partial_j\rho_0^{-(\gamma+2)}dx\nonumber\\
&\quad+\frac14\int\partial_iu^i|\nabla|\tilde{b}|^2|^2\rho_0^{-(\gamma+2)}dx
+\frac14\int u^i|\nabla|\tilde{b}|^2|^2\partial_i\rho_0^{-(\gamma+2)}dx\nonumber\\
&\le C(\|\nabla u\|_{L^\infty}+\|\sqrt{\rho_0}u\|_{L^\infty})\|\rho_0^{-\frac{\gamma+2}{2}}\nabla|\tilde{b}|^2\|_{L^2}^2.
\end{align*}
Substituting the above estimates into \eqref{z2.43}, one gets from \eqref{h1} and Lemma \ref{l21} that
\begin{align}\label{2.46}
&\frac{d}{dt}\|\rho_0^{-\frac{\gamma+2}{2}}\nabla|\tilde{b}|^2\|_{L^2}^2
+\nu\|\rho_0^{-\frac{\gamma+2}{2}}\Delta|\tilde{b}|^2\|_{L^2}^2\nonumber\\
&\le C\|\rho_0^{-\frac{\gamma+2}{2}}\nabla|\tilde{b}|^2\|_{L^2}^2
+C\|\nabla\tilde{b}\|_{L^\infty}\|\rho_0^{-\frac{\gamma+2}{2}}\nabla\tilde{b}\|_{L^2}
\|\rho_0^{-\frac{\gamma+2}{2}}\nabla|\tilde{b}|^2\|_{L^2}\nonumber\\
&\quad+\|\rho_0^{-\frac{\gamma+2}{2}}\nabla\tilde{b}\|_{L^2}^2
+C\|\nabla\tilde{b}\|_{L^\infty}^2\|\rho_0^{-\frac{\gamma+2}{2}}\nabla\tilde{b}\|_{L^2}^2
+C\|\nabla u\|_{L^\infty}^2\||b|^2\rho_0^{-\frac{\gamma+2}{2}}\|_{L^2}^2\nonumber\\
&\quad+C(\|\nabla u\|_{L^\infty}+\|\sqrt{\rho_0}u\|_{L^\infty})
\|\rho_0^{-\frac{\gamma+2}{2}}\nabla|\tilde{b}|^2\|_{L^2}^2\nonumber\\
&\le C\phi(t)\Big\|\Big(\rho_0^{-\frac{\gamma+2}{2}}\nabla\tilde{b},
 \rho_0^{-\frac{\gamma+2}{2}}\nabla|\tilde{b}|^2, |b|^2\rho_0^{-\frac{\gamma+2}{2}}\Big)\Big\|_{L^2}^2
+C\|\rho_0^{-\frac{\gamma+2}{2}}\nabla|\tilde{b}|^2\|_{L^2}^2.
\end{align}
Noticing that
\begin{align*}
\|\rho_0^{-\frac{\gamma+2}{2}}|\nabla b||b|\|_{L^2}^2&\le C\|\rho_0^{-\frac{\gamma+2}{2}}\nabla |\tilde b (1, 0, 1)|^2\|_{L^2}^2
+C\|\rho_0^{-\frac{\gamma+2}{2}}\nabla|\tilde b (1, 0, -1)|^2\|_{L^2}^2
+\|\rho_0^{-\frac{\gamma+2}{2}}\nabla |\tilde b (1, 1, 0)|^2\|_{L^2}^2\nonumber\\
&\quad+C\|\rho_0^{-\frac{\gamma+2}{2}}\nabla |\tilde b (0, 1, 1)|^2\|_{L^2}^2
+C\|\rho_0^{-\frac{\gamma+2}{2}}\nabla |\tilde b(1, 1, 1)|^2\|_{L^2}^2
+C\|\rho_0^{-\frac{\gamma+2}{2}}\nabla|\tilde b (1, -1, 0)|^2\|_{L^2}^2\nonumber\\
&\quad+C\|\rho_0^{-\frac{\gamma+2}{2}}\nabla|\tilde b (0, 1, -1)|^2\|_{L^2}^2,
\end{align*}
and
\begin{align*}
\|\rho_0^{-\frac{\gamma+2}{2}}|\Delta b||b|\|_{L^2}^2&\le C\|\rho_0^{-\frac{\gamma+2}{2}}|\nabla b|^2\|_{L^2}^2+
C\|\rho_0^{-\frac{\gamma+2}{2}}\Delta |\tilde b (1, 0, 1)|^2\|_{L^2}^2
+C\|\rho_0^{-\frac{\gamma+2}{2}}\Delta|\tilde b (1, 0, -1)|^2\|_{L^2}^2
\nonumber\\
&\quad+C\|\rho_0^{-\frac{\gamma+2}{2}}\Delta |\tilde b (0, 1, 1)|^2\|_{L^2}^2
+C\|\rho_0^{-\frac{\gamma+2}{2}}\Delta |\tilde b(1, 1, 1)|^2\|_{L^2}^2
+C\|\rho_0^{-\frac{\gamma+2}{2}}\Delta|\tilde b (1, -1, 0)|^2\|_{L^2}^2\nonumber\\
&\quad+C\|\rho_0^{-\frac{\gamma+2}{2}}\Delta|\tilde b (0, 1, -1)|^2\|_{L^2}^2
+\|\rho_0^{-\frac{\gamma+2}{2}}\Delta |\tilde b (1, 1, 0)|^2\|_{L^2}^2\nonumber\\
&\le C\|\nabla b\|_{L^\infty}^2\|\rho_0^{-\frac{\gamma+2}{2}}\nabla b\|_{L^2}^2
+C\|\rho_0^{-\frac{\gamma+2}{2}}\Delta |\tilde b (1, 0, 1)|^2\|_{L^2}^2
+C\|\rho_0^{-\frac{\gamma+2}{2}}\Delta|\tilde b (1, 0, -1)|^2\|_{L^2}^2
\nonumber\\
&\quad+C\|\rho_0^{-\frac{\gamma+2}{2}}\Delta |\tilde b (0, 1, 1)|^2\|_{L^2}^2
+C\|\rho_0^{-\frac{\gamma+2}{2}}\Delta |\tilde b(1, 1, 1)|^2\|_{L^2}^2
+C\|\rho_0^{-\frac{\gamma+2}{2}}\Delta|\tilde b (1, -1, 0)|^2\|_{L^2}^2\nonumber\\
&\quad+C\|\rho_0^{-\frac{\gamma+2}{2}}\Delta|\tilde b (0, 1, -1)|^2\|_{L^2}^2
+\|\rho_0^{-\frac{\gamma+2}{2}}\Delta |\tilde b (1, 1, 0)|^2\|_{L^2}^2.
\end{align*}
Thanks to the above two estimates, and applying Gronwall's inequality to \eqref{2.46}, we deduce that
\begin{align}\label{2.47}
&\sup_{0\le t\le T}\|\rho_0^{-\frac{\gamma+2}{2}}|\nabla b||b|\|_{L^2}^2
+\int_0^T\|\rho_0^{-\frac{\gamma+2}{2}}(|\Delta b||b|, \Delta|b|^2)\|_{L^2}^2dt\nonumber\\
&\le C\Big\|\Big(\rho_0^\frac{1-\gamma}{2}u_0,
\rho_0^{1-\frac{\gamma}{2}}\theta_0,
\rho_0^{-\frac{\gamma}{2}}b_0, \rho_0^{-\frac{\gamma+2}{2}}|b_0|^2, \rho_0^{-\frac{\gamma}{2}}\nabla u_0,
\rho_0^{-\frac{\gamma+2}{2}}\nabla b_0, \rho_0^{-\frac{\gamma}{2}}\wp_0\Big)\Big\|_{L^2}^2.
\end{align}

2. Operating $\partial_t+\divv(u\cdot)$ to the $j$th-component of $\eqref{a1}_2$ and multiplying the resulting equation
by $\rho_0^{1-\gamma}\dot{u}^j$, one gets by some calculations that
\begin{align}\label{2.48}
\frac12\frac{d}{dt}\int\rho\rho_0^{1-\gamma}|\dot{u}^j|^2dx&=\frac12\int\rho u\nabla\rho_0^{1-\gamma}|\dot{u}|^2dx
-\int\rho_0^{1-\gamma}\dot{u}^j\big(\partial_jp_t+
\divv(u\partial_jp)\big)dx
\nonumber\\
&\quad+\mu\int\rho_0^{1-\gamma}\dot{u}^j\big(\Delta u_t+\divv(u\Delta u^j)\big)dx\nonumber\\
&\quad+(\mu+\lambda)\int\rho_0^{1-\gamma}\dot{u}^j(\partial_j\divv u_t+\divv(u\partial_j\divv u))dx\nonumber\\
&\quad-\frac12\int\rho_0^{1-\gamma}\dot{u}^j\big(\partial_t\partial_j|b|^2+\divv(u\partial_j|b|^2)\big)dx\nonumber\\
&\quad+\int\rho_0^{1-\gamma}\dot{u}^j\big(\partial_t\partial_i(b^ib^j)+\divv(u\partial_i(b^ib^j))\big)dx\nonumber\\
&\le C\phi(t)\|\sqrt{\rho}\rho_0^\frac{1-\gamma}{2}\dot{u}\|_{L^2}^2+\sum_{i=1}^5U_i.
\end{align}
It follows from integration by parts and Young's inequality that
\begin{align}\label{2.49}
U_1&=-\int\big(\partial_j(p_t+\divv(up))-\divv(\partial_jup)\big)\rho_0^{1-\gamma}\dot{u}^jdx\nonumber\\
&=\widetilde{R}\int[(\rho\theta)_t+\divv(\rho u\theta)]\big(\partial_j\rho_0^{1-\gamma}\dot{u}^j+\rho_0^{1-\gamma}\partial_j\dot{u}^j\big)dx
-\widetilde{R}\int\rho\theta\partial_ju\cdot\big(\nabla\rho_0^{1-\gamma}\dot{u}^j
+\rho_0^{1-\gamma}\nabla\dot{u}^j\big)dx\nonumber\\
&\le C\int\rho|\dot{\theta}|\rho_0^{-\gamma}\rho_0^\frac32|\dot{u}|dx+C\int\rho\rho_0^{1-\gamma}|\dot{\theta}||\nabla\dot{u}|dx\nonumber\\
&\quad+C\int\rho\rho_0^{-\gamma}\rho_0^\frac32\theta|\nabla u||\dot{u}|dx
+C\int\rho\rho_0^{1-\gamma}\theta|\nabla u||\nabla\dot{u}|dx\nonumber\\
&\le C\|\rho_0^{1-\frac{\gamma}{2}}\dot{u}\|_{L^2}^2+C\|\rho_0^\frac{3-\gamma}{2}\dot{\theta}\|_{L^2}^2
+C\|\nabla u\|_{L^\infty}^2\|\rho_0^\frac{3-\gamma}{2}\theta\|_{L^2}^2
+C\|\rho_0^\frac{1-\gamma}{2}\nabla\dot{u}\|_{L^2}^2\nonumber\\
&\le C\|\rho_0^{1-\frac{\gamma}{2}}\dot{u}\|_{L^2}^2+C\|\rho_0^\frac{3-\gamma}{2}\theta_t\|_{L^2}^2
+C\|\nabla u\|_{L^\infty}^2\|\rho_0^\frac{3-\gamma}{2}\theta\|_{L^2}^2\nonumber\\
&\quad+\frac{\mu}{16}\|\rho_0^\frac{1-\gamma}{2}\nabla\dot{u}\|_{L^2}^2
+C\|\sqrt{\rho_0}u\|_{L^\infty}^2\|\rho_0^{1-\frac{\gamma}{2}}\nabla\theta\|_{L^2}^2.
\end{align}
Integration by parts together with \eqref{h1} and Young's inequality yields that
\begin{align}
U_2&=\mu\int\big(\Delta\dot{u}^j+\divv(u\Delta u^j)-\Delta(u\cdot\nabla u^j)\big)\rho_0^{1-\gamma}\dot{u}^jdx\nonumber\\
&=\mu\int(\partial_{kk}\dot{u}^j+\partial_i(u^i\partial_k^2u^j)-\partial_k^2(u^i\partial_iu^j))\rho_0^{1-\gamma}\dot{u}^jdx\nonumber\\
&=\mu\int\big[\partial_{kk}\dot{u}^j+\partial_i\partial_k(u^i\partial_ku^j)-\partial_i(\partial_ku^i\partial_ku^j)
-\partial_k(u^i\partial_k\partial_iu^j)
-\partial_k(\partial_ku^i\partial_iu^j)\big]\rho_0^{1-\gamma}\dot{u}^jdx\nonumber\\
&=\mu\int\big[\partial_{kk}\dot{u}^j+\partial_k(\partial_iu^i\partial_ku^j)-\partial_i(\partial_ku^i\partial_ku^j)
-\partial_k(\partial_ku^i\partial_iu^j)\big]\rho_0^{1-\gamma}\dot{u}^jdx\nonumber\\
&\le -\mu\int\rho_0^{1-\gamma}|\nabla\dot{u}|dx-\int\partial_k\dot{u}^j\partial_k\rho_0^{\gamma-1}\dot{u}^jdx
+C\int|\nabla\rho_0^{1-\gamma}||\dot{u}||\nabla u|^2dx
+\int\rho_0^{1-\gamma}|\nabla\dot{u}||\nabla u|^2dx\nonumber\\
&\le -\frac{3\mu}{4}\|\rho_0^\frac{\gamma-1}{2}\nabla\dot{u}\|_{L^2}^2
+C\|\rho_0^{1-\frac{\gamma}{2}}\dot{u}\|_{L^2}^2+C\|\nabla u\|_{L^\infty}^2\|\rho_0^\frac{1-\gamma}{2}\nabla u\|_{L^2}^2.
\end{align}
Similarly, one finds that
\begin{align}
U_3&=(\mu+\lambda)\int\big(\partial_j\divv\dot{u}-\partial_j\divv(u\cdot\nabla u)
+\divv(u\partial_j\divv u)\big)\rho_0^{1-\gamma}\dot{u}^jdx\nonumber\\
&=(\mu+\lambda)\int\big(\partial_j\divv\dot{u}+\divv\partial_j(u\divv u)-\divv(\partial_ju\divv u)\big)\rho_0^{1-\gamma}\dot{u}^jdx\nonumber\\
&\quad-(\mu+\lambda)\int\big(\partial_j(u\cdot\nabla\divv u+\nabla u_k\cdot\partial_ku)\big)\rho_0^{1-\gamma}\dot{u}^jdx\nonumber\\
&=(\mu+\lambda)\int\big[\partial_j\divv\dot{u}-\partial_j(\nabla u_k\cdot\partial_ku)
-\divv(\partial_ju\divv u)\big]\rho_0^{1-\gamma}\dot{u}^jdx\nonumber\\
&\le -(\mu+\lambda)\|\rho_0^\frac{1-\gamma}{2}\divv u\|_{L^2}^2
-\int\divv\dot{u}\partial_j\rho_0^{1-\gamma}\dot{u}^jdx+C\int|\nabla\rho_0^{1-\gamma}||\dot{u}||\nabla u|^2dx\nonumber\\
&\quad+\int\rho_0^{1-\gamma}|\nabla\dot{u}||\nabla u|^2dx\nonumber\\
&\le -\frac{(\mu+\lambda)}{2}\|\rho_0^\frac{1-\gamma}{2}\divv u\|_{L^2}^2
+\frac{\mu}{16}\|\rho_0^\frac{\gamma-1}{2}\nabla\dot{u}\|_{L^2}^2
+C\|\rho_0^{1-\frac{\gamma}{2}}\dot{u}\|_{L^2}^2
+C\|\nabla u\|_{L^\infty}^2\|\rho_0^\frac{1-\gamma}{2}\nabla u\|_{L^2}^2.
\end{align}
Furthermore, integration by parts together with \eqref{h1} gives that
\begin{align}
U_4&=\int\rho_0^{1-\gamma}\partial_j\dot{u}^jb\cdot b_tdx+\int\partial_j\rho_0^{1-\gamma}\dot{u}^jb\cdot b_tdx
+\frac12\int\rho_0^{1-\gamma}\dot{u}^j\partial_i(u^i\partial_j|b|^2)dx\nonumber\\
&=-\frac12\int u^i\partial_j|b|^2\big(\partial_i\rho_0^{1-\gamma}\dot{u}^j
+\rho_0^{1-\gamma}\partial_i\dot{u}^j\big)dx\nonumber\\
&\quad+\int\rho_0^{1-\gamma}\partial_j\dot{u}^jb\cdot\big(b\cdot\nabla u-u\cdot\nabla b+b\divv u+\nu\Delta b\big)dx\nonumber\\
&\quad+\int\partial_j\rho_0^{1-\gamma}\dot{u}^jb\cdot\big(b\cdot\nabla u-u\cdot\nabla b+b\divv u+\nu\Delta b\big)dx\nonumber\\
&\le C\int\rho_0^{-\gamma}\rho_0^\frac32|u||b||\nabla b||\dot{u}|dx+C\int\rho_0^{1-\gamma}|u||b||\nabla b||\nabla\dot{u}|dx\nonumber\\
&\quad+C\int\rho_0^{1-\gamma}|\nabla\dot{u}||b|^2|\nabla u|dx+C\int\rho_0^{1-\gamma}|\nabla\dot{u}||b||\Delta b|dx\nonumber\\
&\quad+C\int\rho_0^{-\gamma}\rho_0^\frac32|\dot{u}||b|^2|\nabla u|dx
+C\int\rho_0^{-\gamma}\rho_0^\frac32|\dot{u}||u||b||\nabla b|dx\nonumber\\
&\quad+C\int\rho_0^{-\gamma}\rho_0^\frac32|\dot{u}||b||\Delta b|dx\nonumber\\
&\le \frac{\mu}{16}\|\rho_0^\frac{1-\gamma}{2}\nabla\dot{u}\|_{L^2}^2
+C\big(\|\nabla u\|_{L^\infty}^2+\|\sqrt{\rho_0}u\|_{L^\infty}^2+1\big)\|\rho_0^{1-\frac{\gamma}{2}}\dot{u}\|_{L^2}^2\nonumber\\
&\quad+C(1+\|\sqrt{\rho_0}u\|_{L^\infty}^2)\|\rho_0^{-\frac{\gamma}{2}}|b||\nabla b|\|_{L^2}^2
+\|\rho_0^\frac{1-\gamma}{2}|b||\Delta b|\|_{L^2}^2\nonumber\\
&\quad+C(\|\sqrt{\rho_0}u\|_{L^\infty}^2+1)\|\rho_0^{-\frac{\gamma}{2}}|b||\nabla b|\|_{L^2}^2.
\end{align}
Similarly, we obtain that
\begin{align}\label{2.53}
|U_5|&\le \frac{\mu}{16}\|\rho_0^\frac{1-\gamma}{2}\nabla\dot{u}\|_{L^2}^2
+C\big(\|\nabla u\|_{L^\infty}^2+\|\sqrt{\rho_0}u\|_{L^\infty}^2+1\big)\|\rho_0^{1-\frac{\gamma}{2}}\dot{u}\|_{L^2}^2
+\|\rho_0^\frac{1-\gamma}{2}|b||\Delta b|\|_{L^2}^2\nonumber\\
&\quad+C(1+\|\sqrt{\rho_0}u\|_{L^\infty}^2)\|\rho_0^{-\frac{\gamma}{2}}|b||\nabla b|\|_{L^2}^2
+C\big(\|\sqrt{\rho_0}u\|_{L^\infty}^2+1\big)\|\rho_0^{-\frac{\gamma}{2}}|b||\nabla b|\|_{L^2}^2.
\end{align}
Putting \eqref{2.49}--\eqref{2.53} into \eqref{2.48}, one infers from \eqref{h1} and Lemma \ref{l21} that
\begin{align}\label{z2.49}
&\frac{d}{dt}\|\sqrt{\rho}\rho_0^\frac{1-\gamma}{2}\dot{u}\|_{L^2}^2
+\mu\|\rho_0^\frac{1-\gamma}{2}\nabla\dot{u}\|_{L^2}^2+(\mu+\lambda)\|\rho_0^\frac{1-\gamma}{2}\divv \dot{u}\|_{L^2}^2\nonumber\\
&\le C\big(\|\nabla u\|_{L^\infty}^2+\|\sqrt{\rho_0}u\|_{L^\infty}^2+1\big)\|\rho_0^{1-\frac{\gamma}{2}}\dot{u}\|_{L^2}^2
+C\|\rho_0^\frac{1-\gamma}{2}|b||\Delta b|\|_{L^2}^2\nonumber\\
&\quad+C\big(1+\|\sqrt{\rho_0}u\|_{L^\infty}^2\big)\|\rho_0^{-\frac{\gamma}{2}}|b||\nabla b|\|_{L^2}^2
+C\|\rho_0^\frac{3-\gamma}{2}\theta_t\|_{L^2}^2
+C\|\nabla u\|_{L^\infty}^2\|\rho_0^\frac{3-\gamma}{2}\theta\|_{L^2}^2\nonumber\\
&\quad+C\|\nabla u\|_{L^\infty}^2\|\rho_0^\frac{1-\gamma}{2}\nabla u\|_{L^2}^2
+C\|\sqrt{\rho_0}u\|_{L^\infty}^2\|\rho_0^{1-\frac{\gamma}{2}}\nabla\theta\|_{L^2}^2\nonumber\\
&\le C\|\rho_0^\frac{3-\gamma}{2}\theta_t\|_{L^2}^2
+C\|\rho_0^{-\frac{\gamma+2}{2}}|b||\Delta b|\|_{L^2}^2
+C\phi(t)\Big(\|\rho_0^{1-\frac{\gamma}{2}}\dot{u}\|_{L^2}^2
+\|\rho_0^\frac{3-\gamma}{2}\theta\|_{L^2}^2\Big)\nonumber\\
&\quad+C\phi(t)\Big(\|\rho_0^{-\frac{\gamma}{2}}|b||\nabla b|\|_{L^2}^2
+\|\rho_0^{1-\frac{\gamma+2}{2}}\nabla\theta\|_{L^2}^2+\|\rho_0^\frac{1-\gamma}{2}\nabla u\|_{L^2}^2\Big),
\end{align}
which together with Gronwall's inequality, \eqref{2.47}, and Lemmas \ref{l22}--\ref{l24} gives \eqref{2.40}.
\end{proof}

\begin{lemma}\label{l26}
There exists a positive constant $C$ depending only on $\mu$, $\lambda$, $\widetilde{R}$, $c_v$, $\kappa$, $\nu$, $K_1$, $K_2$, and $\Phi_T$ such that
\begin{align}\label{z2.50}
&\sup_{0\le t\le T}\Big\|\Big(\rho_0^{1-\frac{\gamma}{2}}u_t, \rho_0^{-\frac{\gamma}{2}}b_t\Big)\Big\|_{L^2}^2
+\int_0^T\Big\|\Big(\rho_0^\frac{1-\gamma}{2}\nabla u_t, \rho_0^{-\frac{\gamma}{2}}\nabla b_t\Big)\Big\|_{L^2}^2dt\nonumber\\
&\le C\Big\|\Big(\rho_0^\frac{1-\gamma}{2}u_0,
\rho_0^{1-\frac{\gamma}{2}}\theta_0,
\rho_0^{-\frac{\gamma}{2}}b_0, \rho_0^{-\frac{\gamma+2}{2}}|b_0|^2, \rho_0^{-\frac{\gamma}{2}}\nabla u_0,
\rho_0^{-\frac{\gamma+2}{2}}\nabla b_0, \rho_0^{-\frac{\gamma}{2}}\nabla\theta_0, \rho_0^{-\frac{\gamma}{2}}\wp_0\Big)\Big\|_{L^2}^2,
\end{align}
where $\wp_0$ is defined as in Lemma \ref{l25}.
\end{lemma}
\begin{proof}[Proof]
1. Differentiating $\eqref{a1}_2$ with respect to $t$ gives that
\begin{align}\label{2.56}
&\rho u_{tt}+\rho u\cdot\nabla u_t-\mu\Delta u_t-(\mu+\lambda)\nabla\divv u_t\nonumber\\
&=-\rho_t(u_t+u\cdot\nabla u)-\rho u_t\cdot\nabla u-\nabla p_t+\Big(b\cdot\nabla b-\frac12\nabla|b|^2\Big)_t.
\end{align}
Multiplying \eqref{2.56} by $\rho_0^{1-\gamma}u_t$ and integrating the resulting equation over $\mathbb{R}^3$,
we obtain from $\eqref{a1}_1$ that
\begin{align}\label{2.57}
&\frac12\frac{d}{dt}\int\rho_0^{1-\gamma}\rho|u_t|^2dx+\mu\int|\nabla u_t|^2\rho_0^{1-\gamma}dx
+(\mu+\lambda)\int(\divv u_t)^2\rho_0^{1-\gamma}dx\nonumber\\
&=-\mu\int (\nabla\rho_0^{1-\gamma}\cdot\nabla) u_t\cdot u_tdx
-(\mu+\lambda)\int \divv u_t u_t\cdot\nabla\rho_0^{1-\gamma}dx\nonumber\\
&\quad-\int\rho_t|u_t|^2\rho_0^{1-\gamma}dx-\int\rho_tu\cdot\nabla u\cdot u_t\rho_0^{1-\gamma}dx
+\int p_t\divv u_t\rho_0^{1-\gamma}dx\nonumber\\
&\quad+\int p_t u_t\cdot\nabla\rho_0^{1-\gamma}dx-\int\rho u_t\cdot\nabla u\cdot u_t\rho_0^{1-\gamma}dx
+\int\Big(b\cdot\nabla b-\frac12\nabla|b|^2\Big)_t\cdot u_t\rho_0^{1-\gamma}dx
\triangleq\sum_{i=1}^8\bar{J}_i,
\end{align}
where we have used the following simple facts
\begin{align*}
&\frac12\int\rho_0^{1-\gamma}\rho\frac{d}{dt}|u_t|^2dx+\int\rho_0^{1-\gamma}\rho u \cdot\nabla u_t\cdot u_tdx
=\frac12\frac{d}{dt}\int\rho_0^{1-\gamma}\rho|u_t|^2dx,\\
&-\int\Delta u_t\cdot u_t\rho_0^{1-\gamma}dx=\int|\nabla u_t|^2\rho_0^{1-\gamma}dx+\int (\nabla\rho_0^{1-\gamma}\cdot\nabla) u_t\cdot u_tdx,\\
&-\int\nabla\divv u_t \rho_0^{1-\gamma}u_tdx=\int(\divv u_t)^2\rho_0^{1-\gamma}dx+\int \divv u_t u_t\cdot\nabla\rho_0^{1-\gamma}dx.
\end{align*}
It follows from \eqref{h1} and Cauchy-Schwarz inequality that
\begin{align*}
\bar{J}_1+\bar{J}_2&\le \frac{\mu}{10}\|\rho_0^\frac{1-\gamma}{2}\nabla u_t\|_{L^2}^2
+\frac{\mu+\lambda}{4}\|\rho_0^\frac{1-\gamma}{2}\divv u_t\|_{L^2}^2+C\|\rho_0^{1-\frac{\gamma}{2}}u_t\|_{L^2}^2.
\end{align*}
By \eqref{h1}, \eqref{h2}, and Lemma \ref{l21}, we obtain after integration by parts that
\begin{align*}
\bar{J}_3&=\int\divv(\rho u)|u_t|^2\rho_0^{1-\gamma}dx\nonumber\\
&=-\int\rho u\cdot\nabla |u_t|^2\rho_0^{1-\gamma}dx
-\int\rho u|u_t|^2\cdot\nabla\rho_0^{1-\gamma}dx\nonumber\\
&\le \frac{\mu}{10}\|\rho_0^\frac{1-\gamma}{2}\nabla u_t\|_{L^2}^2
+C\|\sqrt{\rho_0}u\|_{L^\infty}^2\|\rho_0^{1-\frac{\gamma}{2}}u_t\|_{L^2}^2,\\
\bar{J}_4&=\int\divv(\rho u)(u\cdot\nabla u)\cdot u_t\rho_0^{1-\gamma}dx\nonumber\\
&=-\int\rho u^j\partial_ju^k\partial_ku^iu^i_t\rho_0^{1-\gamma}dx-\int\rho u^ju^k\partial_j\partial_ku^iu^i_t\rho_0^{1-\gamma}dx\nonumber\\
&\quad-\int\rho u^ju^k\partial_ku^i\partial_ju^i_t\rho_0^{1-\gamma}dx-\int\rho u^ju^k\partial_ku^iu^i_t\partial_j\rho_0^{1-\gamma}dx\nonumber\\
&\le -\int\rho u^j\partial_ju^k\partial_ku^iu^i_t\rho_0^{1-\gamma}dx-C\int\rho_0u^ju^k\partial_j\partial_ku^iu^i_t\rho_0^{1-\gamma}dx\nonumber\\
&\quad-\int\rho u^ju^k\partial_ku^i\partial_ju^i_t\rho_0^{1-\gamma}dx-\int\rho u^ju^k\partial_ku^iu^i_t\partial_j\rho_0^{1-\gamma}dx\nonumber\\
&=-\int\rho u^j\partial_ju^k\partial_ku^iu^i_t\rho_0^{1-\gamma}dx-\int\rho u^ju^k\partial_ku^i\partial_ju^i_t\rho_0^{1-\gamma}dx\nonumber\\
&\quad-\int\rho u^ju^k\partial_ku^iu^i_t\partial_j\rho_0^{1-\gamma}dx+C\int\partial_k\rho_0u^ju^k\partial_ju^iu^i_t\rho_0^{1-\gamma}dx\nonumber\\
&\quad+C\int\rho_0u^j\partial_ku^k\partial_ju^iu^i_t\rho_0^{1-\gamma}dx
+C\int\rho_0u^ju^k\partial_ju^i\partial_ku^i_t\rho_0^{1-\gamma}dx\nonumber\\
&\quad+C\int\rho_0u^ju^k\partial_ju^iu^i_t\partial_k\rho_0^{1-\gamma}dx\nonumber\\
&\le C\int\rho\rho_0^{1-\gamma}|u||\nabla u|^2|u_t|dx+C\int\rho\rho_0^{1-\gamma}|u|^2|\nabla u||\nabla u_t|dx\nonumber\\
&\quad+C\int\rho|u|^2|\nabla u||u_t||\nabla\rho_0^{1-\gamma}|dx+C\int\rho_0^{1-\gamma}|\nabla\rho_0||u|^2|\nabla u||u_t|dx\nonumber\\
&\quad+C\int\rho_0^{2-\gamma}|u||\nabla u|^2|u_t|dx+C\int\rho_0^{2-\gamma}|u|^2|\nabla u||\nabla u_t|dx\nonumber\\
&\quad+C\int\rho_0|\nabla\rho_0^{1-\gamma}||u|^2|\nabla u||u_t|dx\nonumber\\
&\le \frac{\mu}{10}\|\rho_0^\frac{1-\gamma}{2}\nabla u_t\|_{L^2}^2
+C\|\nabla u\|_{L^\infty}^2\Big(\|\rho_0^{1-\frac{\gamma}{2}}u\|_{L^2}^2
+\|\rho_0^{1-\frac{\gamma}{2}}u_t\|_{L^2}^2\Big)\nonumber\\
&\quad+C\|\sqrt{\rho_0}u\|_{L^\infty}^2\Big(\|\rho_0^{1-\frac{\gamma}{2}}u_t\|_{L^2}^2+C\|\rho_0^\frac{1-\gamma}{2}\nabla u\|_{L^2}^2\Big)\nonumber\\
&\quad+C\|\nabla u\|_{L^\infty}\|u\|_{L^6}\|\rho^\frac{3-\gamma}{2}u\|_{L^6}\|\rho_0^\frac{1-\gamma}{2}\nabla u_t\|_{L^2}\nonumber\\
&\le \frac{\mu}{10}\|\rho_0^\frac{1-\gamma}{2}\nabla u_t\|_{L^2}^2
+C\phi(t)\Big(\|\rho_0^{1-\frac{\gamma}{2}}u\|_{L^2}^2+\|\rho_0^{1-\frac{\gamma}{2}}u_t\|_{L^2}^2
+\|\rho_0^\frac{1-\gamma}{2}\nabla u\|_{L^2}^2\Big),\\
\bar{J}_5&=\widetilde{R}\int(\rho_t\theta+\rho\theta_t)\divv u_t\rho_0^{1-\gamma}dx\nonumber\\
&=-\widetilde{R}\int\divv(\rho u)\theta\divv u_t\rho_0^{1-\gamma}dx
+\widetilde{R}\int\rho\theta_t\divv u_t\rho_0^{1-\gamma}dx\nonumber\\
&=-\widetilde{R}\int\partial_j(\rho u^j)\theta\partial_ku^k_t\rho_0^{1-\gamma}dx
+\widetilde{R}\int\rho\theta_t\divv u_t\rho_0^{1-\gamma}dx\nonumber\\
&=\widetilde{R}\int\rho u^j\partial_j\theta\partial_ku^k_t\rho_0^{1-\gamma}dx+\widetilde{R}\int\rho u^j\theta\partial_j\partial_ku^k_t\rho_0^{1-\gamma}dx\nonumber\\
&\quad+\widetilde{R}\int\rho u^j\theta\partial_ku^k_t\partial_j\rho_0^{1-\gamma}dx
+\widetilde{R}\int\rho\theta_t\divv u_t\rho_0^{1-\gamma}dx\nonumber\\
&\le \widetilde{R}\int\rho u^j\partial_j\theta\partial_ku^k_t\rho_0^{1-\gamma}dx+C\widetilde{R}\int\rho_0 u^j\theta\partial_j\partial_ku^k_t\rho_0^{1-\gamma}dx\nonumber\\
&\quad+\widetilde{R}\int\rho u^j\theta\partial_ku^k_t\partial_j\rho_0^{1-\gamma}dx
+\widetilde{R}\int\rho\theta_t\divv u_t\rho_0^{1-\gamma}dx\nonumber\\
&=\widetilde{R}\int\rho u^j\partial_j\theta\partial_ku^k_t\rho_0^{1-\gamma}dx+\widetilde{R}\int\rho u^j\theta\partial_ku^k_t\partial_j\rho_0^{1-\gamma}dx\nonumber\\
&\quad+\widetilde{R}\int\rho\theta_t\divv u_t\rho_0^{1-\gamma}dx-C\widetilde{R}\int\partial_k\rho_0 u^j\theta\partial_ju^k_t\rho_0^{1-\gamma}dx\nonumber\\
&\quad-C\widetilde{R}\int\rho_0\partial_ku^j\theta\partial_ju^k_t\rho_0^{1-\gamma}dx
-C\widetilde{R}\int\rho_0u^j\partial_k\theta\partial_ju^k_t\rho_0^{1-\gamma}dx\nonumber\\
&\quad-C\widetilde{R}\int\rho_0\partial_k\rho_0^{1-\gamma}u^j\theta\partial_ju^kdx\nonumber\\
&\le C\int\rho\rho_0^{1-\gamma}|u||\nabla\theta||\nabla u_t|dx
+C\int\rho\theta|u||\nabla u_t||\nabla\rho_0^{1-\gamma}|dx\nonumber\\
&\quad+C\int\rho\theta_t|\divv u_t|\rho_0^{1-\gamma}dx
+C\int\rho_0^{1-\gamma}|\theta||u||\nabla\rho_0||\nabla u_t|dx\nonumber\\
&\quad+C\int\rho_0^{2-\gamma}|\theta||\nabla u||\nabla u_t|dx
+C\int\rho_0^{2-\gamma}|u||\nabla\theta||\nabla u_t|dx\nonumber\\
&\quad+C\int\rho|\nabla\rho_0^{1-\gamma}||u||\theta||\nabla u|dx\nonumber\\
&\le \frac{\mu}{10}\|\rho_0^\frac{1-\gamma}{2}\nabla u_t\|_{L^2}^2
+C\|\sqrt{\rho_0}u\|_{L^\infty}^2\Big(\|\rho_0^{1-\frac{\gamma}{2}}\nabla\theta\|_{L^2}^2
+\|\rho_0^{1-\frac{\gamma}{2}}\theta\|_{L^2}^2\Big)\nonumber\\
&\quad+\frac{\mu+\lambda}{4}\|\rho_0^\frac{1-\gamma}{2}\divv u\|_{L^2}^2
+C\|\rho_0^\frac{3-\gamma}{2}\theta_t\|_{L^2}^2
+C\|\nabla u\|_{L^\infty}^2\|\rho_0^{1-\frac{\gamma}{2}}\theta\|_{L^2}^2
+C\|\rho_0^{1-\frac{\gamma}{2}}\nabla u\|_{L^2}^2,\\
\bar{J}_6&=\widetilde{R}\int(\rho_t\theta+\rho\theta_t)\cdot u_t\nabla\rho_0^{1-\gamma}dx\nonumber\\
&=-\widetilde{R}\int\divv(\rho u)\theta u_t\cdot\nabla \rho_0^{1-\gamma}dx
+\widetilde{R}\int\rho \theta_tu_t\cdot\nabla \rho_0^{1-\gamma}dx\nonumber\\
&=-\widetilde{R}\int\partial_j(\rho u^j)\theta u_t^k\partial_k\rho_0^{1-\gamma}dx
+\widetilde{R}\int\rho \theta_t u_t\cdot\nabla \rho_0^{1-\gamma}dx\nonumber\\
&=\widetilde{R}\int\rho u^j\partial_j\theta u_t^k\partial_k\rho_0^{1-\gamma}dx
+\widetilde{R}\int\rho u^j\theta\partial_j u_t^k\partial_k\rho_0^{1-\gamma}dx\nonumber\\
&\quad+\widetilde{R}\int\rho u^j\theta u_t^k\partial_j \partial_k\rho_0^{1-\gamma}dx+\widetilde{R}\int\rho \theta_t u_t\cdot\nabla \rho_0^{1-\gamma}dx\nonumber\\
&\le \frac{\mu}{10}\|\rho_0^\frac{1-\gamma}{2}\nabla u_t\|_{L^2}^2+C\|\sqrt{\rho_0}u\|_{L^\infty}^2\|\sqrt{\rho_0}^{1-\frac{\gamma}{2}}u_t\|_{L^2}^2
+C\|\rho_0^{1-\frac{\gamma}{2}}\nabla\theta\|_{L^2}^2\nonumber\\
&\quad+C\|\sqrt{\rho_0}u\|_{L^\infty}^2\|\rho_0^{1-\frac{\gamma}{2}}\theta\|_{L^2}^2
+C\|\rho_0^\frac{3-\gamma}{2}\theta_t\|_{L^2}^2+C\|\rho_0^{1-\frac{\gamma}{2}}u_t\|_{L^2}^2,\\
\bar{J}_7&\le C\|\nabla u\|_{L^\infty}\|\sqrt{\rho}\rho_0^\frac{1-\gamma}{2}u_t\|_{L^2}^2
\le C\phi(t)\|\sqrt{\rho}\rho_0^\frac{1-\gamma}{2}u_t\|_{L^2}^2,\\
\bar{J}_8&=\int\Big(b\cdot\nabla b-\frac12\nabla|b|^2\Big)_t\cdot u_t\rho_0^{1-\gamma}dx\nonumber\\
&=\int \divv(b\otimes b)_t\cdot u_t\rho_0^{1-\gamma}dx+\int b\cdot b_t\divv u_t\rho_0^{1-\gamma}dx
+\int b\cdot b_t\cdot u_t\cdot\nabla\rho_0^{1-\gamma}dx\nonumber\\
&=-\int (b\otimes b)_t:\nabla u_t\rho_0^{1-\gamma}dx-\int (b\otimes b):u_t\otimes\nabla\rho_0^{1-\gamma}\cdot dx
+\int b\cdot b_t\divv u_t\rho_0^{1-\gamma}dx\nonumber\\
&\quad+\int b\cdot b_t\cdot u_t\cdot\nabla\rho_0^{1-\gamma}dx\nonumber\\
&\le -\int b\otimes(\nu\Delta b+b\cdot\nabla u-u\cdot\nabla b-b\divv u):\nabla u_t\rho_0^{1-\gamma}dx\nonumber\\
&\quad-\int(\nu\Delta b+b\cdot\nabla u-u\cdot\nabla b-b\divv u)\otimes b:\nabla u_t\rho_0^{1-\gamma}dx\nonumber\\
&\quad+\int b\cdot(\nu\Delta b+b\cdot\nabla u-u\cdot\nabla b-b\divv u)\divv u_t\rho_0^{1-\gamma}dx\nonumber\\
&\quad+\int b\cdot(\nu\Delta b+b\cdot\nabla u-u\cdot\nabla b-b\divv u)\cdot u_t\cdot\nabla\rho_0^{1-\gamma}dx\nonumber\\
&\quad-\int (b\otimes b):\nabla\rho_0^{1-\gamma}\otimes u_tdx\nonumber\\
&\le \frac{\mu}{10}\|\rho_0^\frac{1-\gamma}{2}\nabla u_t\|_{L^2}^2
+C\||b||\Delta b|\rho_0^\frac{1-\gamma}{2}\|_{L^2}^2+C\|b\|_{L^\infty}^2\|\rho_0^\frac{1-\gamma}{2}\nabla u\|_{L^2}^2\nonumber\\
&\quad+C\|\sqrt{\rho_0}u\|_{L^\infty}^2\||b||\nabla b|\rho_0^{-\frac{\gamma}{2}}\|_{L^2}^2
+C\|b\|_{L^\infty}^2\|\rho_0^{1-\frac{\gamma}{2}}u_t\|_{L^2}^2.
\end{align*}
Inserting the above estimates on $\bar{J}_i\ (i=1, 2,\cdots, 8)$ into \eqref{2.57}, we arrive at
\begin{align}\label{2.58}
&\frac{d}{dt}\|\sqrt{\rho}\rho_0^\frac{1-\gamma}{2}u_t\|_{L^2}^2
+\mu\|\rho_0^\frac{1-\gamma}{2}\nabla u_t\|_{L^2}^2
+(\mu+\lambda)\|\rho_0^\frac{1-\gamma}{2}\divv u_t\|_{L^2}^2\nonumber\\
&\le C\phi(t)\Big(\|\rho_0^{1-\frac{\gamma}{2}}u\|_{L^2}^2
+\||b||\nabla b|\rho_0^{-\frac{\gamma}{2}}\|_{L^2}^2
+\|\rho_0^\frac{1-\gamma}{2}\nabla u\|_{L^2}^2+\|\rho_0^{1-\frac{\gamma}{2}}\nabla\theta\|_{L^2}^2
+\|\rho_0^{1-\frac{\gamma}{2}}\theta\|_{L^2}^2\Big)\nonumber\\
&\quad+C\phi(t)\|\sqrt{\rho}\rho_0^\frac{1-\gamma}{2}u_t\|_{L^2}^2
+C\||b||\Delta b|\rho_0^\frac{1-\gamma}{2}\|_{L^2}^2+C\|\rho_0^\frac{3-\gamma}{2}\theta_t\|_{L^2}^2.
\end{align}

2. Differentiating $\eqref{a1}_4$ with respect to $t$ shows that
\begin{align}\label{2.59}
b_{tt}-b_t\cdot\nabla u-b\cdot\nabla u_t+u_t\cdot\nabla b
+u\cdot\nabla b_t+b_t\divv u+b\divv u_t=\nu\Delta b_t.
\end{align}
Multiplying \eqref{2.59} by $b_t\rho_0^{-\gamma}$ and integrating the resulting equation over $\mathbb{R}^3$ yield that
\begin{align}\label{2.60}
&\frac12\frac{d}{dt}\int|b_t|^2\rho_0^{-\gamma}dx+\nu\int|\nabla b_t|^2\rho_0^{-\gamma}dx\nonumber\\
&=-\nu\int b_t\cdot\nabla b_t\cdot\nabla\rho_0^{-\gamma}dx+\int b\cdot\nabla u_t\cdot b_t\rho_0^{-\gamma}dx
-\int u_t\cdot\nabla b\cdot b_t\rho_0^{-\gamma}dx\nonumber\\
&\quad-\int b\cdot b_t\divv u_t\rho_0^{-\gamma}dx+\int b_t\cdot \nabla u\cdot b_t\rho_0^{-\gamma}dx
-\int u\cdot\nabla b_t\cdot b_t\rho_0^{-\gamma}dx\nonumber\\
&\quad-\int |b_t|^2\divv u\rho_0^{-\gamma}dx\triangleq\sum_{i=1}^7\bar{K}_i.
\end{align}
By \eqref{h1} and Cauchy-Schwarz inequality, we arrive at
\begin{align*}
\bar{K}_1+\bar{K}_5+\bar{K}_7
&\le C\|\rho_0\|_{L^\infty}^\frac12\|\rho_0^{-\frac{\gamma}{2}}\nabla b_t\|_{L^2}
\|\rho_0^{-\frac{\gamma}{2}}b_t\|_{L^2}
+C\|\nabla u\|_{L^\infty}
\|\rho_0^{-\frac{\gamma}{2}}b_t\|_{L^2}^2\nonumber\\
&\le \frac{\nu}{4}\|\rho_0^{-\frac{\gamma}{2}}\nabla b_t\|_{L^2}^2
+C(\|\nabla u\|_{L^\infty}+1)\|\rho_0^{-\frac{\gamma}{2}}b_t\|_{L^2}^2\nonumber\\
&\le \frac{\nu}{6}\|\rho_0^{-\frac{\gamma}{2}}\nabla b_t\|_{L^2}^2+C\phi(t)\|\rho_0^{-\frac{\gamma}{2}}b_t\|_{L^2}^2.
\end{align*}
Using $\eqref{a1}_4$, H\"older's inequality, and Young's inequality, we deduce that
\begin{align*}
\bar{K}_2+\bar{K}_4
&=\int b\cdot\nabla u_t\cdot(\nu\Delta b+b\cdot\nabla u-u\cdot\nabla b-b\divv u)\rho_0^{-\gamma}dx\nonumber\\
& \quad +\int b\cdot(\nu\Delta b+b\cdot\nabla u-u\cdot\nabla b-b\divv u)\divv u_t\rho_0^{-\gamma}dx \notag \\
&\le C\|\rho_0^\frac{1-\gamma}{2}\nabla u_t\|_{L^2}\||b||\Delta b|\rho_0^{-\frac{1+\gamma}{2}}\|_{L^2}
+C\|\nabla u\|_{L^\infty}\|\rho_0^\frac{1-\gamma}{2}\nabla u_t\|_{L^2}\||b|^2\rho_0^{-\frac{1+\gamma}{2}}\|_{L^2}\nonumber\\
&\quad+C\|\sqrt{\rho_0}u\|_{L^\infty}\|\rho_0^\frac{1-\gamma}{2}\nabla u_t\|_{L^2}\||b||\nabla b|\rho_0^{-\frac{2+\gamma}{2}}\|_{L^2}\nonumber\\
&\le \frac{\mu}{6}\|\rho_0^\frac{1-\gamma}{2}\nabla u_t\|_{L^2}^2+C\||b||\Delta b|\rho_0^{-\frac{1+\gamma}{2}}\|_{L^2}^2
+C\|\nabla u\|_{L^\infty}^2\||b|^2\rho_0^{-\frac{1+\gamma}{2}}\|_{L^2}^2\nonumber\\
&\quad+C\|\sqrt{\rho_0}u\|_{L^\infty}^2\||b||\nabla b|\rho_0^{-\frac{2+\gamma}{2}}\|_{L^2}^2\nonumber\\
&\le \frac{\mu}{6}\|\rho_0^\frac{1-\gamma}{2}\nabla u_t\|_{L^2}^2
+C\phi(t)\Big(\||b|^2\rho_0^{-\frac{1+\gamma}{2}}\|_{L^2}^2
+\||b||\nabla b|\rho_0^{-\frac{2+\gamma}{2}}\|_{L^2}^2\Big)
+C\||b||\Delta b|\rho_0^{-\frac{1+\gamma}{2}}\|_{L^2}^2.
\end{align*}
We obtain from integration by parts that
\begin{align}\label{z2.55}
& \int u_t\cdot\nabla b\cdot\Delta b\rho_0^{-\gamma}dx \notag \\
& =\int u_t^i\partial_i b^j\partial_{kk} b^j\rho_0^{-\gamma}dx \notag \\
& =-\int \partial_{k}u_t^i\partial_i b^j\partial_{k} b^j\rho_0^{-\gamma}dx
-\int u_t^i\partial_{ik} b^j\partial_{k} b^j\rho_0^{-\gamma}dx
-\int u_t^i\partial_i b^j\partial_{k} b^j\partial_{k}\rho_0^{-\gamma}dx \notag \\
& =-\int \partial_{k}u_t^i\partial_i b^j\partial_{k} b^j\rho_0^{-\gamma}dx
+\frac12\int \divv u_t|\nabla b|^2\rho_0^{-\gamma}dx
-\int u_t^i\partial_i b^j\partial_{k} b^j\partial_{k}\rho_0^{-\gamma}dx \notag \\
& \quad -\frac12\int u_t|\nabla b|^2\cdot\nabla\rho_0^{-\gamma}dx \notag \\
& \leq C\int |\nabla u_t||\nabla b|^2\rho_0^{-\gamma}dx+C\int |u_t||\nabla b|^2|\nabla\rho_0^{-\gamma}|dx
\end{align}
due to
\begin{align*}
-\int u_t^i\partial_{ik} b^j\partial_{k} b^j\rho_0^{-\gamma}dx
& =\int \divv u_t|\nabla b|^2\rho_0^{-\gamma}dx
+\int u_t^i\partial_{ik} b^j\partial_{k} b^j\rho_0^{-\gamma}dx
-\int u_t|\nabla b|^2\cdot\nabla\rho_0^{-\gamma}dx,
\end{align*}
which combined with $\eqref{a1}_4$ and \eqref{h1} leads to
\begin{align*}
\bar{K}_3&=\int u_t\cdot\nabla b\cdot(b\cdot\nabla u-b\divv u-u\cdot\nabla b+\nu\Delta b)\rho_0^{-\gamma}dx\nonumber\\
&\le C\int |u_t||\nabla b||b||\nabla u|\rho_0^{-\gamma}dx
+C\int|u_t||u||\nabla b|^2\rho_0^{-\gamma}dx\nonumber\\
&\quad+C\int |\nabla u_t||\nabla b|^2\rho_0^{-\gamma}dx
+C\int |u_t||\nabla b|^2|\nabla\rho_0^{-\gamma}|dx\nonumber\\
&\le C\|\nabla u\|_{L^\infty}\|\rho_0^{1-\frac{\gamma}{2}}u_t\|_{L^2}
\|\rho_0^{-\frac{\gamma+2}{2}}|b||\nabla b|\|_{L^2}
+C\|\nabla b\|_{L^\infty}^2\|\rho_0^{1-\frac{\gamma}{2}}u_t\|_{L^2}
\|\rho_0^{-1-\frac{\gamma}{2}}u\|_{L^2}\nonumber\\
&\quad
+C\|\nabla b\|_{L^\infty}\|\rho_0^{-\frac{\gamma}{2}}\nabla u_t\|_{L^2}\|\rho_0^{-\frac{\gamma}{2}}\nabla b\|_{L^2}
+C\|\nabla b\|_{L^\infty}\|\rho_0^{1-\frac{\gamma}{2}}u_t\|_{L^2}\|\rho_0^{-\frac{\gamma+1}{2}}\nabla b\|_{L^2}\nonumber\\
&\le C\phi(t)\Big(\|\rho_0^{1-\frac{\gamma}{2}}u_t\|_{L^2}^2
+\|\rho_0^{-1-\frac{\gamma}{2}}u\|_{L^2}^2
+\|\rho_0^{-\frac{\gamma}{2}}\nabla b\|_{L^2}^2\Big)
+C\|\rho_0^{-\frac{\gamma+2}{2}}|b||\nabla b|\|_{L^2}^2
+C\|\rho_0^{-\frac{\gamma+1}{2}}\nabla b\|_{L^2}.
\end{align*}
Integration by parts gives that
\begin{align*}
\bar{K}_6
=\int \divv u|b_t|^2\rho_0^{-\gamma}dx-\bar{K}_6
+\int (u\cdot b_t) (b_t\cdot\nabla \rho_0^{-\gamma})dx,
\end{align*}
which along with \eqref{h1} implies that
\begin{align*}
\bar{K}_6
&\le C\big(\|\nabla u\|_{L^\infty}+\|\sqrt{\rho_0}u\|_{L^\infty}\big)
\|\rho_0^{-\frac{\gamma}{2}}b_t\|_{L^2}^2
\le C\phi(t)
\|\rho_0^{-\frac{\gamma}{2}}b_t\|_{L^2}^2.
\end{align*}
Substituting the above estimates on $\bar{K}_i\ (i=1, 2,\cdots, 6)$ into \eqref{2.60} leads to
\begin{align*}
&\frac{d}{dt}\|\rho_0^{-\frac{\gamma}{2}}b_t\|_{L^2}^2+\nu\|\rho_0^{-\frac{\gamma}{2}}\nabla b_t\|_{L^2}^2\nonumber\\
&\le C\phi(t)\Big(\|\rho_0^{1-\frac{\gamma}{2}}u_t\|_{L^2}^2
+\|\rho_0^{-\frac{\gamma+1}{2}}\nabla b\|_{L^2}
+\|\rho_0^{-\frac{\gamma+2}{2}}|b||\nabla b|\|_{L^2}^2+\||b|^2\rho_0^{-\frac{1+\gamma}{2}}\|_{L^2}^2
+\|\rho_0^{-\frac{\gamma}{2}}b_t\|_{L^2}^2\Big)\nonumber\\
&\quad+\frac{\mu}{2}\|\rho_0^\frac{1-\gamma}{2}\nabla u_t\|_{L^2}^2+C\|\rho_0^{-\frac{\gamma}{2}}\nabla b\|_{L^2}^2
+C\|\rho_0^{-\frac{\gamma}{2}}\nabla^2b\|_{L^2}^2
+C\||b||\Delta b|\rho_0^{-\frac{1+\gamma}{2}}\|_{L^2}^2.
\end{align*}
This together with \eqref{2.58} yields that
\begin{align*}
&\frac{d}{dt}\Big(\|\sqrt{\rho}\rho_0^\frac{1-\gamma}{2}u_t\|_{L^2}^2
+\|\rho_0^{-\frac{\gamma}{2}}b_t\|_{L^2}^2\Big)+\|\rho_0^\frac{1-\gamma}{2}\nabla u_t\|_{L^2}^2
+\|\rho_0^{-\frac{\gamma}{2}}\nabla b_t\|_{L^2}^2\nonumber\\
&\le C\phi(t)\Big(\|\sqrt{\rho}\rho_0^\frac{1-\gamma}{2}u_t\|_{L^2}^2
+\|\rho_0^{-\frac{\gamma}{2}}b_t\|_{L^2}^2\Big)
+C\|\rho_0^{-\frac{\gamma+2}{2}}\nabla^2b\|_{L^2}^2
+C\||b||\Delta b|\rho_0^{-\frac{\gamma+2}{2}}\|_{L^2}^2+C\|\rho_0^\frac{3-\gamma}{2}\theta_t\|_{L^2}^2
\notag \\ & \quad +C\phi(t)\Big(\|\rho_0^{1-\frac{\gamma}{2}}u\|_{L^2}^2
+\|\rho_0^\frac{1-\gamma}{2}\nabla u\|_{L^2}^2+\|\rho_0^{1-\frac{\gamma}{2}}\nabla\theta\|_{L^2}^2
+\|\rho_0^{1-\frac{\gamma}{2}}\theta\|_{L^2}^2\Big)\nonumber\\
&\quad+C\phi(t)\Big(\|\rho_0^{-\frac{\gamma+2}{2}}\nabla b\|_{L^2}
+\|\rho_0^{-\frac{\gamma+2}{2}}|b||\nabla b|\|_{L^2}^2+\||b|^2\rho_0^{-\frac{2+\gamma}{2}}\|_{L^2}^2\Big),
\end{align*}
which combined with Gronwall's inequality and Lemmas \ref{l21}--\ref{l25} implies \eqref{z2.50}.
\end{proof}

\begin{lemma}\label{l27}
There exists a positive constant $C$ depending only on $\mu$, $\lambda$, $\widetilde{R}$, $c_v$, $\kappa$, $\nu$, $K_1$, $K_2$, and $\Phi_T$ such that
\begin{align}
&\|\nabla(\rho_0^{-\frac{\gamma}{2}}u)\|_{L^6}+\|\rho_0^{-\frac{\gamma}{2}}\nabla u\|_{L^6}
+\|\nabla(\rho_0^{-\frac{\gamma}{2}}b)\|_{L^6}+\|\rho_0^{-\frac{\gamma}{2}}\nabla b\|_{L^6}\nonumber\\
&\le C\Big\|\Big(\rho_0^\frac{1-\gamma}{2}u_0,
\rho_0^{1-\frac{\gamma}{2}}\theta_0,
\rho_0^{-\frac{\gamma}{2}}b_0, \rho_0^{-\frac{\gamma+2}{2}}|b_0|^2, \rho_0^{-\frac{\gamma}{2}}\nabla u_0,
\rho_0^{-\frac{\gamma+2}{2}}\nabla b_0, \rho_0^{-\frac{\gamma}{2}}\nabla\theta_0, \rho_0^{-\frac{\gamma}{2}}\wp_0\Big)\Big\|_{L^2}^2,
\end{align}
where $\wp_0$ is defined as in Lemma \ref{l25}.
\end{lemma}
\begin{proof}[Proof]
It follows from $\eqref{a1}_2$ and $\eqref{a1}_4$ that
\begin{align}\label{2.64}
&\mu\Delta\Big(\rho_0^{-\frac{\gamma}{2}}u\Big)
+(\mu+\lambda)\nabla\divv\Big(\rho_0^{-\frac{\gamma}{2}}u\Big)\nonumber\\
&=\mu\rho_0^{-\frac{\gamma}{2}}\Delta u+\mu\divv\Big(u\otimes\nabla\rho_0^{-\frac{\gamma}{2}}\Big)
+\mu\nabla u\cdot\nabla\rho_0^{-\frac{\gamma}{2}}
+(\mu+\lambda)\rho_0^{-\frac{\gamma}{2}}\nabla\divv u\nonumber\\
&\quad+(\mu+\lambda)\nabla\Big(u\cdot\nabla \rho_0^{-\frac{\gamma}{2}}\Big)
+(\mu+\lambda)\divv\nabla\rho_0^{-\frac{\gamma}{2}}\nonumber\\
&=\rho_0^{-\frac{\gamma}{2}}\rho\dot{u}-\rho_0^{-\frac{\gamma}{2}}\curl b\times b
+\divv\Big(\mu u\otimes\nabla\rho_0^{-\frac{\gamma}{2}}+(\mu+\lambda)u\cdot\nabla\rho_0^{-\frac{\gamma}{2}}I
+\rho_0^{-\frac{\gamma}{2}}pI\Big)\nonumber\\
&\quad+(\mu\nabla u+(\mu+\lambda)\divv uI-pI)\nabla\rho_0^{-\frac{\gamma}{2}},
\end{align}
and
\begin{align}\label{2.65}
\nu\Delta\Big(\rho_0^{-\frac{\gamma}{2}}b\Big)
&=\nu\rho_0^{-\frac{\gamma}{2}}\Delta b+\nu\divv\Big(b\otimes\nabla\rho_0^{-\frac{\gamma}{2}}\Big)
+\nu\nabla b\cdot\nabla\rho_0^{-\frac{\gamma}{2}}\nonumber\\
&=\rho_0^{-\frac{\gamma}{2}}(b_t-b\cdot\nabla u+u\cdot\nabla b+b\divv u)
+\nu\nabla b\cdot\nabla\rho_0^{-\frac{\gamma}{2}}
+\nu\divv\Big(b\otimes\nabla\rho_0^{-\frac{\gamma}{2}}\Big).
\end{align}
Applying the standard $L^2$-estimate to systems \eqref{2.64} and \eqref{2.65}, respectively, then we derive from Lemma \ref{l21}, \eqref{h1}, Sobolev's inequality, and Young's inequality that
\begin{align}\label{2.66}
\Big\|\nabla\Big(\rho_0^{-\frac{\gamma}{2}}u\Big)\Big\|_{L^6}
&\le C\Big\|\rho_0^{-\frac{\gamma}{2}}|b||\nabla b|\Big\|_{L^2}
+C\Big\|(\mu\nabla u+(\mu+\lambda)\divv uI-pI)\nabla\rho_0^{-\frac{\gamma}{2}}\Big\|_{L^2}\nonumber\\
&\quad+C\|\rho_0^{-\frac{\gamma}{2}}\rho\dot{u}\|_{L^2}
+C\Big\|\mu u\otimes\nabla\rho_0^{-\frac{\gamma}{2}}+(\mu+\lambda)u\cdot\nabla\rho_0^{-\frac{\gamma}{2}}I
+\rho_0^{-\frac{\gamma}{2}}pI\Big\|_{L^6}\nonumber\\
&\le C\Big\|\rho_0^{1-\frac{\gamma}{2}}\dot{u}\Big\|_{L^2}
+C\Big\|\rho_0^{-\frac{\gamma}{2}}|b||\nabla b|\Big\|_{L^2}
+C\Big\|\rho_0^\frac{1-\gamma}{2}\nabla u\Big\|_{L^2}
+C\Big\|\rho_0^\frac{3-\gamma}{2}\theta\Big\|_{L^2}\nonumber\\
&\quad+C\Big\|\rho_0^\frac{1-\gamma}{2}u\Big\|_{L^6}
+C\Big\|\rho_0^{1-\frac{\gamma}{2}}\theta\Big\|_{L^6}\nonumber\\
&\le C\Big\|\rho_0^{1-\frac{\gamma}{2}}\dot{u}\Big\|_{L^2}
+C\Big\|\rho_0^{-\frac{\gamma}{2}}|b||\nabla b|\Big\|_{L^2}
+C\Big\|\rho_0^\frac{1-\gamma}{2}\nabla u\Big\|_{L^2}
+C\Big\|\rho_0^\frac{3-\gamma}{2}\theta\Big\|_{L^2}\nonumber\\
&\quad+C\Big\|\rho_0^{1-\frac{\gamma}{2}}u\Big\|_{L^2}
+C\Big\|\rho_0^{1-\frac{\gamma}{2}}\nabla\theta\Big\|_{L^2},
\end{align}
which along with Sobolev's inequality yields that
\begin{align}\label{2.67}
\|\rho_0^{-\frac{\gamma}{2}}\nabla u\|_{L^6}
&=\|\nabla(\rho_0^{-\frac{\gamma}{2}}u)-\nabla\rho_0^{-\frac{\gamma}{2}}u\|_{L^6}
\le C\|\nabla(\rho_0^{-\frac{\gamma}{2}}u)\|_{L^6}
+C\|\nabla\rho_0^{-\frac{\gamma}{2}}u\|_{L^6}\nonumber\\
&\le C\|\nabla(\rho_0^{-\frac{\gamma}{2}}u)\|_{L^6}
+C\|\rho_0^{1-\frac{\gamma}{2}}u\|_{L^2}+C\|\rho_0^\frac{1-\gamma}{2}\nabla u\|_{L^2}\nonumber\\
&\le C\Big\|\Big(\rho_0^{1-\frac{\gamma}{2}}\dot{u}, \rho_0^{-\frac{\gamma}{2}}|b||\nabla b|,
\rho_0^\frac{1-\gamma}{2}\nabla u, \rho_0^\frac{3-\gamma}{2}\theta, \rho_0^{1-\frac{\gamma}{2}}u,
\rho_0^{1-\frac{\gamma}{2}}\nabla\theta\Big)\Big\|_{L^2}.
\end{align}
Similarly, by Gagliardo-Nirenberg inequality, \eqref{h1}, and \eqref{2.67}, it follows from the estimates of elliptic equations that
\begin{align}\label{2.68}
\|\nabla (\rho_0^{-\frac{\gamma}{2}}b)\|_{L^6}
&\le C\Big\|\rho_0^{-\frac{\gamma}{2}}(b_t-b\cdot\nabla u+u\cdot\nabla b+b\divv u)\Big\|_{L^2}
+C\|\rho_0^\frac{1-\gamma}{2}\nabla b\|_{L^2}+C\|\rho_0^\frac{1-\gamma}{2}b\|_{L^6}\nonumber\\
&\le C\|\rho_0^{-\frac{\gamma}{2}}b_t\|_{L^2}+C\|b\|_{L^4}\|\rho_0^{-\frac{\gamma}{2}}\nabla u\|_{L^4}
+C\|\rho_0^{-\frac{\gamma}{2}}u\|_{L^6}\|u\|_{L^6}\|\rho_0^{-\frac{\gamma}{2}}\nabla b\|_{L^2}\|\nabla b\|_{L^6}\nonumber\\
&\quad+C\|\rho_0^\frac{1-\gamma}{2}\nabla b\|_{L^2}
+C\|\nabla\rho_0^\frac{1-\gamma}{2}b\|_{L^2}\nonumber\\
&\le C\|\rho_0^{-\frac{\gamma}{2}}b_t\|_{L^2}+C\|b\|_{L^2}^\frac14\|\nabla b\|_{L^2}^\frac34\Big(\|\rho_0^{-\frac{\gamma}{2}}\nabla u\|_{L^2}
+\|\rho_0^{-\frac{\gamma}{2}}\nabla u\|_{L^6}\Big)\nonumber\\
&\quad+C\|\rho_0^\frac{1-\gamma}{2}\nabla b\|_{L^2}
+C\Big(\|\rho_0^\frac{1-\gamma}{2}u\|_{L^2}+\|\rho_0^{-\frac{\gamma}{2}}\nabla u\|_{L^2}\Big)\|\nabla u\|_{L^2}
\|\nabla^2b\|_{L^2}\|\rho_0^{-\frac{\gamma}{2}}\nabla b\|_{L^2}\nonumber\\
&\le C\Big\|\Big(\rho_0^{-\frac{\gamma}{2}}b_t, \rho_0^{-\frac{\gamma}{2}}\nabla u,
\rho_0^\frac{1-\gamma}{2}\nabla b\Big)\Big\|_{L^2}
+C\Big\|\Big(\rho_0^\frac{1-\gamma}{2}u, \rho_0^{-\frac{\gamma}{2}}\nabla u, \rho_0^{-\frac{\gamma}{2}}\nabla b\Big)\Big\|_{L^2}^2
+C\|\rho_0^{-\frac{\gamma}{2}}\nabla u\|_{L^6}\nonumber\\
&\le C\Big\|\Big(\rho_0^{-\frac{\gamma}{2}}b_t, \rho_0^{-\frac{\gamma}{2}}\nabla u,
\rho_0^\frac{1-\gamma}{2}\nabla b, \rho_0^{1-\frac{\gamma}{2}}\dot{u}, \rho_0^{-\frac{\gamma+2}{2}}|b||\nabla b|,
\rho_0^\frac{3-\gamma}{2}\theta, \rho_0^{1-\frac{\gamma}{2}}u,
\rho_0^{1-\frac{\gamma}{2}}\nabla\theta\Big)\Big\|_{L^2}\nonumber\\
&\quad+C\Big\|\Big(\rho_0^\frac{1-\gamma}{2}u, \rho_0^{-\frac{\gamma}{2}}\nabla u, \rho_0^{-\frac{\gamma+2}{2}}\nabla b\Big)\Big\|_{L^2}^2,
\end{align}
and
\begin{align}\label{2.69}
\|\rho_0^{-\frac{\gamma}{2}}\nabla b\|_{L^6}&=\|\nabla(\rho_0^{-\frac{\gamma}{2}}b)-\nabla\rho_0^{-\frac{\gamma}{2}}b\|_{L^6}
\le C\|\nabla(\rho_0^{-\frac{\gamma}{2}}b)\|_{L^6}
+C\|\rho_0^\frac{1-\gamma}{2}b\|_{L^6}\nonumber\\
&\le C\Big\|\Big(\rho_0^{-\frac{\gamma}{2}}b_t, \rho_0^{-\frac{\gamma}{2}}\nabla u,
\rho_0^\frac{1-\gamma}{2}\nabla b, \rho_0^{1-\frac{\gamma}{2}}\dot{u}, \rho_0^{-\frac{\gamma+2}{2}}|b||\nabla b|,
\rho_0^\frac{3-\gamma}{2}\theta, \rho_0^{1-\frac{\gamma}{2}}u,
\rho_0^{1-\frac{\gamma}{2}}\nabla\theta\Big)\Big\|_{L^2}\nonumber\\
&\quad+C\Big\|\Big(\rho_0^\frac{1-\gamma}{2}u, \rho_0^{-\frac{\gamma}{2}}\nabla u, \rho_0^{-\frac{\gamma+2}{2}}\nabla b\Big)\Big\|_{L^2}^2.
\end{align}
Combining \eqref{2.66}--\eqref{2.69} altogether and applying Lemmas \ref{l22}--\ref{l26}, the conclusion follows.
\end{proof}

\section{The De Giorgi iterations}\label{sec3}
This section is devoted to carrying out suitable De Giorgi iterations, which are
preparations for proving the lower and upper bounds of the entropy in the next
section. Moreover, the iterations are applied to different equations in establishing
the lower and upper bounds of the entropy: in dealing with the lower bound, a De
Giorgi iteration is applied to the entropy equation itself, while in dealing with the
upper bound, it is applied to the temperature equation.

Set
\begin{align}\label{3.1}
M_T\triangleq\frac{\kappa(\gamma-1)}{c_v}e^{C_*\Phi_T}\big[(1+|\gamma-2|)K_1^2+K_2\big],
\end{align}
where $C_*$ is the positive constant stated in Lemma \ref{l21} and $\Phi_T$ is given by \eqref{2.1}.

The following De Giorgi type iteration will be used to get the uniform lower bound
of the entropy.
\begin{lemma}\label{lem31}
Let $M_T$ be given by \eqref{3.1} and define
\begin{align*}
\tilde{s}=\log\theta-(\gamma-1)\log\rho_0+M_Tt, \quad \underline{\tilde{s}}_0=\frac{\underline{s}_0}{c_v}-\log\frac{\widetilde{R}}{A}.
\end{align*}
Then the following statements hold.

(i) For any $\ell\le \underline{\tilde{s}}_0$, there exists a positive constant $C$ depending only on $\mu$, $\lambda$, $\widetilde{R}$, $c_v$, $\kappa$, $\nu$, $K_1$, $K_2$, $\Phi_T$, and the initial data such that
\begin{align*}
\sup_{0\le t\le T}\|(\tilde{s}-\ell)_-\|_{L^2}^2
+\int_0^T\Big\|\frac{\nabla(\tilde{s}-\ell)_-}{\sqrt{\rho_0}}\Big\|_{L^2}^2dt\le C.
\end{align*}

(ii) Set
\begin{align*}
\mathcal Y_\ell=\mathcal Y_\ell(T)\triangleq\sup_{0\le t\le T}\|(\tilde{s}-\ell)_-\|_{L^2}^2+\int_0^T\|\nabla(\tilde{s}-\ell)_-\|_{L^2}^2dt,
\quad \forall \ell\le \underline{\tilde{s}}_0,
\end{align*}
where $f_-\triangleq-\min\{f, 0\}$. Then there exists a positive constant $C$ depending only on $c_v$, $\mu$, $\lambda$, $\kappa$, $\nu$, $K_1$, $K_2$, and $\Phi_T$ such that
\begin{align*}
\mathcal Y_\ell\le \frac{C}{(m-\ell)^3}\mathcal Y_m^\frac32, \quad \forall \ell<m\le \underline{\tilde{s}}_0.
\end{align*}
\end{lemma}
\begin{proof}[Proof]
It follows from \eqref{1.3} and the definition of $\tilde{s}$ that
\begin{align}\label{3.2}
c_v\rho(\tilde{s}_t+u\cdot\nabla\tilde{s})-\kappa\Delta\tilde{s}&=
c_vM_T\rho+\kappa(\gamma-1)\Delta\log\rho_0-\widetilde{R}\rho u\cdot\nabla\log\rho_0\nonumber\\
&\quad-\widetilde{R}\rho\divv u+\kappa\Big|\frac{\nabla\theta}{\theta}\Big|^2+\frac{\mathcal Q(\nabla u)}{\theta}
+\nu\frac{|\curl b|^2}{\theta}.
\end{align}
Multiplying \eqref{3.2} by $-\frac{(\tilde{s}-\ell)_-}{\rho_0}$ and integrating the resultant over $\mathbb{R}^3$, we obtain that
\begin{align}
&\frac{c_v}{2}\frac{d}{dt}\Big\|\sqrt{\frac{\rho}{\rho_0}}(\tilde{s}-\ell)_-\Big\|_{L^2}^2
+\kappa\Big\|\frac{\nabla(\tilde{s}-\ell)_-}{\sqrt{\rho_0}}\Big\|_{L^2}^2\nonumber\\
&=-\frac{c_v}{2}\int\rho u\cdot\nabla\rho_0^{-1}(\tilde{s}-\ell)_-^2dx
-\kappa\int\nabla(\tilde{s}-\ell)_-\nabla\rho_0^{-1}(\tilde{s}-\ell)_-dx\nonumber\\
&\quad-\int[c_vM_T\rho+\kappa(\gamma-1)\Delta\log\rho_0]\rho_0^{-1}(\tilde{s}-\ell)_-dx\nonumber\\
&\quad+\widetilde{R}\int\rho u\cdot\nabla\log\rho_0\rho_0^{-1}(\tilde{s}-\ell)_-dx
+\widetilde{R}\int\rho\divv u\rho_0^{-1}(\tilde{s}-\ell)_-dx\nonumber\\
&\quad-\int\Big(\kappa\Big|\frac{\nabla\theta}{\theta}\Big|^2+\frac{\mathcal Q(\nabla u)}{\theta}
+\nu\frac{|\curl b|^2}{\theta}\Big)\rho_0^{-1}(\tilde{s}-\ell)_-dx\nonumber\\
&\le -\frac{c_v}{2}\int\rho u\cdot\nabla\rho_0^{-1}(\tilde{s}-\ell)_-^2dx
-\kappa\int\nabla(\tilde{s}-\ell)_-\nabla\rho_0^{-1}(\tilde{s}-\ell)_-dx\nonumber\\
&\quad-\int[c_vM_T\rho+\kappa(\gamma-1)\Delta\log\rho_0]\rho_0^{-1}(\tilde{s}-\ell)_-dx\nonumber\\
&\quad+\widetilde{R}\int\rho u\cdot\nabla\log\rho_0\rho_0^{-1}(\tilde{s}-\ell)_-dx
+\widetilde{R}\int\rho\divv u\rho_0^{-1}(\tilde{s}-\ell)_-dx.
\end{align}
Noticing that the term $\frac{|\curl b|^2}{\theta}$ doesn't work for estimating the uniform lower bound
of the entropy, thus we can adopt the arguments as those in \cite[Proposition 3.1]{LX23} to get the conclusions of Lemma \ref{lem31}. Here we omit the details for simplicity.
\end{proof}

Next, we need the following De Giorgi
type iteration to get the uniform upper bound of the entropy.
\begin{lemma}
Let $M_T$ be given by \eqref{3.1} and define
\begin{align*}
&\overline{S}_0=\frac{A}{\widetilde{R}}e^{\frac{\overline{s}_0}{c_v}}, \quad\theta_\ell=\theta-\ell e^{M_Tt}\rho_0^{\gamma-1},
\quad \forall\ell\in\mathbb{R},\\
&\mathcal Z_\ell=\mathcal Z_\ell(T)\triangleq\sup_{0\le t\le T}\|\rho_0^{1-\gamma}(\theta_\ell)_+\|_{L^2}^2
+\int_0^T\|\rho_0^{\frac12-\gamma}\nabla(\theta_\ell)_+\|_{L^2}^2dt, \quad \forall\ell\ge\overline{S}_0,
\end{align*}
where $f_+\triangleq\max\{f, 0\}$. Then there exists a positive constant $C$ depending only on
$\mu$, $\lambda$, $\widetilde{R}$, $c_v$, $\kappa$, $\nu$, $K_1$, $K_2$, $\Phi_T$, and the initial data such that
\begin{align*}
&\mathcal Z_\ell\le C(\ell^2+1), \quad \forall\ell\ge \overline{S}_0,\\
&\mathcal Z_\ell\le \frac{C\ell^2}{(\ell-m)^3}\mathcal Z_m^\frac32, \quad \ell>m\ge \overline{S}_0.
\end{align*}
\end{lemma}
\begin{proof}[Proof]
It follows from $\eqref{a1}_3$ that
\begin{align}\label{3.4}
c_v\rho(\partial_t\theta_\ell+u\cdot\nabla \theta_\ell)-\kappa\Delta\theta_\ell
&=-c_v\ell e^{M_Tt}\rho u\cdot\nabla\rho_0^{\gamma-1}-\widetilde{R}\rho\big(\theta_\ell+\ell e^{M_Tt}\rho_0^{\gamma-1}\big)
\divv u\nonumber\\
&+\ell e^{M_Tt}\big(\kappa\Delta\rho_0^{\gamma-1}-c_vM_T\rho_0^{\gamma-1}\rho\big)+\mathcal Q(\nabla u)
+\nu|\curl b|^2.
\end{align}
Multiplying \eqref{3.4} by $\rho_0^{1-2\gamma}(\theta_\ell)_+$ and integration by parts, one has that
\begin{align}\label{3.5}
&\frac{c_v}{2}\frac{d}{dt}\int\rho\rho_0^{1-2\gamma}(\theta_\ell)_+^2dx
+\kappa\int\rho_0^{1-2\gamma}|\nabla(\theta_\ell)_+|^2dx\nonumber\\
&=-\frac{c_v}{2}\int \rho u\cdot\nabla\rho_0^{1-2\gamma}(\theta_\ell)_+^2dx
-\kappa\int \nabla(\theta_\ell)_+\cdot\nabla\rho_0^{1-2\gamma}(\theta_\ell)_+dx\nonumber\\
&\quad-c_v\ell e^{M_Tt}\int\rho u\cdot\nabla\rho_0^{\gamma-1}\rho_0^{1-2\gamma}(\theta_\ell)_+dx
-\widetilde{R}\int\rho\theta_\ell\divv u\rho_0^{1-2\gamma}(\theta_\ell)_+dx\nonumber\\
&\quad-\widetilde{R}\ell e^{M_Tt}\int\rho \rho_0^{-\gamma}\divv u(\theta_\ell)_+dx
+\ell e^{M_Tt}\int(\kappa\Delta\rho_0^{\gamma-1}
-c_vM_T\rho_0^{\gamma-1}\rho)\rho_0^{1-2\gamma}(\theta_\ell)_+dx\nonumber\\
&\quad+\int\mathcal Q(\nabla u)\rho_0^{1-2\gamma}(\theta_\ell)_+dx
+\nu\int|\curl b|^2\rho_0^{1-2\gamma}(\theta_\ell)_+dx\nonumber\\
&\le -\frac{c_v}{2}\int \rho u\cdot\nabla\rho_0^{1-2\gamma}(\theta_\ell)_+^2dx
-\kappa\int \nabla(\theta_\ell)_+\cdot\nabla\rho_0^{1-2\gamma}(\theta_\ell)_+dx\nonumber\\
&\quad-c_v\ell e^{M_Tt}\int\rho u\cdot\nabla\rho_0^{\gamma-1}\rho_0^{1-2\gamma}(\theta_\ell)_+dx
-\widetilde{R}\int\rho\theta_\ell\divv u\rho_0^{1-2\gamma}(\theta_\ell)_+dx\nonumber\\
&\quad-\widetilde{R}\ell e^{M_Tt}\int\rho \rho_0^{-\gamma}\divv u(\theta_\ell)_+dx
+\int\mathcal Q(\nabla u)\rho_0^{1-2\gamma}(\theta_\ell)_+dx\nonumber\\
&\quad+\nu\int|\curl b|^2\rho_0^{1-2\gamma}(\theta_\ell)_+dx\triangleq\sum_{i=1}^7\bar{I}_i,
\end{align}
due to
\begin{align}
&\ell e^{M_Tt}\int\big(\kappa\Delta\rho_0^{\gamma-1}
-c_vM_T\rho_0^{\gamma-1}\rho\big)\rho_0^{1-2\gamma}(\theta_\ell)_+dx\nonumber\\
&=\ell e^{M_Tt}\int\rho_0^\gamma\Big(\kappa(\gamma-1)\frac{\Delta\rho_0}{\rho_0^2}+
\kappa(\gamma-1)(\gamma-2)\frac{|\nabla\rho_0|^2}{\rho_0^3}-c_vM_T\frac{\rho}{\rho_0}\Big)
\rho_0^{1-2\gamma}(\theta_\ell)_+dx\nonumber\\
&\le \ell e^{M_Tt}\int\rho_0^\gamma\Big(\kappa(\gamma-1)K_2+\kappa(\gamma-1)|\gamma-2|K_1^2
-c_vM_Te^{-C_*\Phi_T}\Big)
\rho_0^{1-2\gamma}(\theta_\ell)_+dx\nonumber\\
&=-\kappa(\gamma-1)\ell e^{M_Tt}\int\rho_0^\gamma K_1^2\rho_0^{1-2\gamma}(\theta_\ell)_+dx\le 0,
\end{align}
where we have used \eqref{h1}, \eqref{h2}, \eqref{3.1}, and Lemma \ref{l21}.

It remains to estimate $\bar{I}_i\ (i=1, 2, \cdots, 7)$. We divide into the following two case:

(i) Using \eqref{h1}, Lemma \ref{l27}, Lemma \ref{l23}, H\"older's, Gagliardo-Nirenberg, and Young's inequalities, we derive
from the regularity of $(\rho, u, \theta, b)$ that
\begin{align*}
\sum_{i=1}^6\bar{I}_i&\le \frac{\kappa}{2}\|\rho_0^{\frac12-\gamma}\nabla(\theta_\ell)_+\|_{L^2}^2
+C\|\sqrt{\rho_0}u\|_{L^\infty}\|\sqrt{\rho}\rho_0^{\frac12-\gamma}(\theta_\ell)_+\|_{L^2}^2
\nonumber\\
&\quad+C\ell\|\sqrt{\rho_0}u\|_{L^2}\|\rho_0^{1-\gamma}(\theta_\ell)_+\|_{L^2}
+C\|\nabla u\|_{L^\infty}\|\rho_0^{1-\gamma}(\theta_\ell)_+\|_{L^2}^2\nonumber\\[3pt]
&\quad+C\ell\|\nabla u\|_{L^2}\|\rho_0^{1-\gamma}(\theta_\ell)_+\|_{L^2}+C\|\rho_0^{-\frac{\gamma}{2}}\nabla u\|_{L^4}^4
\|\rho_0^{1-\gamma}(\theta_\ell)_+\|_{L^2}\nonumber\\[3pt]
&\quad+C\|\rho_0^{-\frac{\gamma}{2}}\nabla b\|_{L^4}^4
\|\rho_0^{1-\gamma}(\theta_\ell)_+\|_{L^2}+C\|\rho_0^{1-\gamma}(\theta_\ell)_+\|_{L^2}^2\nonumber\\
&\le \frac{\kappa}{2}\|\rho_0^{\frac12-\gamma}\nabla(\theta_\ell)_+\|_{L^2}^2
+C\phi(t)\|\rho_0^{1-\gamma}(\theta_\ell)_+\|_{L^2}^2+C\|\rho_0^{-\frac{\gamma}{2}}\nabla u\|_{L^2}^2+C\|\rho_0^{-\frac{\gamma}{2}}\nabla u\|_{L^6}^6\nonumber\\[3pt]
&\quad+C\|\rho_0^{-\frac{\gamma}{2}}\nabla b\|_{L^6}^6+C\|\rho_0^{-\frac{\gamma}{2}}\nabla b\|_{L^2}^2+C\ell^2\big(\|\sqrt{\rho_0}u\|_{L^2}^2+\|\nabla u\|_{L^2}^2\big)\nonumber\\
&\le \frac{\kappa}{2}\|\rho_0^{\frac12-\gamma}\nabla(\theta_\ell)_+\|_{L^2}^2
+C\phi(t)\|\rho_0^{1-\gamma}(\theta_\ell)_+\|_{L^2}^2
+C(\ell^2+1),
\end{align*}
from which, and \eqref{3.5}, one has
\begin{align*}
c_v\frac{d}{dt}\|\rho_0^{1-\gamma}(\theta_\ell)_+\|_{L^2}^2
+\kappa\|\rho_0^{\frac12-\gamma}\nabla(\theta_\ell)_+\|_{L^2}^2\le C\phi(t)\|\rho_0^{1-\gamma}(\theta_\ell)_+\|_{L^2}^2
+C(\ell^2+1).
\end{align*}
This together with Gronwall's inequality, and using Lemma \ref{l21}, and noticing that
$(\theta_\ell)_+|_{t=0}=0$, for any $\ell\ge \overline{S}_0$, the first conclusion follows.

(ii) By \eqref{h1}, \eqref{h2}, Lemma \ref{l21}, and Lemma \ref{l27}, we obtain that
\begin{align*}
\bar{I}_1+\bar{I}_2+\bar{I}_4&\le \frac{\kappa}{4}\|\rho_0^{\frac12-\gamma}\nabla(\theta_\ell)_+\|_{L^2}^2
+C(\|\sqrt{\rho_0}u\|_{L^\infty}+\|\nabla u\|_{L^\infty}+1)
\|\sqrt{\rho}\rho_0^{\frac12-\gamma}(\theta_\ell)_+\|_{L^2}^2\nonumber\\
&\le \frac{\kappa}{4}\|\rho_0^{\frac12-\gamma}\nabla(\theta_\ell)_+\|_{L^2}^2+C\phi(t)
\|\sqrt{\rho}\rho_0^{\frac12-\gamma}(\theta_\ell)_+\|_{L^2}^2,\\
\bar{I}_3+\bar{I}_5&\le C\ell(\|\sqrt{\rho_0}u\|_{L^\infty}
+\|\nabla u\|_{L^\infty})\int\rho_0^{1-\gamma}(\theta_\ell)_+^2dx\nonumber\\
&\le C\ell^2\int_{\{\theta_\ell>0\}}1 dx+C\phi(t)\|\rho_0^{1-\gamma}(\theta_\ell)_+\|_{L^2}^2,\\
\bar{I}_6+\bar{I}_7&\le C(\|\rho_0^{-\frac{\gamma}{2}}\nabla u\|_{L^6}^2+\|\rho_0^{-\frac{\gamma}{2}}\nabla b\|_{L^6}^2)
\|\rho_0^{1-\gamma}(\theta_\ell)_+\|_{L^6}\Big(\int_{\{\theta_\ell>0\}}1 dx\Big)^\frac12\nonumber\\
&\le C(\|\rho_0^{1-\gamma}\nabla(\theta_\ell)_+\|_{L^2}+\|\nabla\rho_0^{1-\gamma}(\theta_\ell)_+\|_{L^2})
\Big(\int_{\{\theta_\ell>0\}}1 dx\Big)^\frac12\nonumber\\
&\le C(\|\rho_0^{1-\gamma}\nabla(\theta_\ell)_+\|_{L^2}+\|\rho_0^{\frac32-\gamma}(\theta_\ell)_+\|_{L^2})
\Big(\int_{\{\theta_\ell>0\}}1 dx\Big)^\frac12\nonumber\\
&\le \frac{\kappa}{4}\|\rho_0^{\frac12-\gamma}\nabla(\theta_\ell)_+\|_{L^2}^2
+C\|\rho_0^{1-\gamma}(\theta_\ell)_+\|_{L^2}^2+C\ell^2\int_{\{\theta_\ell>0\}}1 dx.
\end{align*}
Substituting the above estimates on $\bar{I}_i\ (i=1, 2,\cdots, 7)$ into \eqref{3.5}, we arrive at
\begin{align}\label{3.7}
c_v\frac{d}{dt}\|\rho_0^{1-\gamma}(\theta_\ell)_+\|_{L^2}^2
+\kappa\|\rho_0^{\frac12-\gamma}\nabla(\theta_\ell)_+\|_{L^2}^2\le
C\phi(t)\|\rho_0^{1-\gamma}(\theta_\ell)_+\|_{L^2}^2+C\ell^2\int_{\{\theta_\ell>0\}}1 dx.
\end{align}
Since $M_T>0$, one can check that for any $\ell>m$, it holds that
\begin{align*}
1\le e^{-M_Tt}\rho_0^{1-\gamma}\frac{\theta_m^+}{\ell-m}\le \rho_0^{1-\gamma}\frac{\theta_m^+}{\ell-m},
\quad\text{on}~\{\theta_\ell>0\}\subseteq\{\theta_m>0\},
\end{align*}
which together with \eqref{h1} and Gagliardo-Nirenberg inequality gives that
\begin{align}\label{3.8}
\int_{\{\theta_\ell>0\}}1 dx&\le
\int_{\{\theta_\ell>0\}}\Big|\frac{\rho_0^{1-\gamma}\theta_m^+}{\ell-m}\Big|^3dx
\le \frac{C}{(\ell-m)^3}\|\rho^{1-\gamma}\theta_m^+\|_{L^2}^\frac32
\|\nabla(\rho_0^{1-\gamma}\theta_m^+)\|_{L^2}^\frac32\nonumber\\
&\le \frac{C}{(\ell-m)^3}\|\rho^{1-\gamma}\theta_m^+\|_{L^2}^\frac32
\big(\|\rho_0^{1-\gamma}\nabla\theta_m^+\|_{L^2}
+\|\nabla\rho_0^{1-\gamma}\theta_m^+\|_{L^2}\big)^\frac32\nonumber\\
&\le \frac{C}{(\ell-m)^3}\|\rho^{1-\gamma}\theta_m^+\|_{L^2}^\frac32
\big(\|\rho_0^{1-\gamma}\nabla\theta_m^+\|_{L^2}
+\|\rho_0^{1-\gamma}\theta_m^+\|_{L^2}\big)^\frac32.
\end{align}
Thanks to \eqref{3.7} and \eqref{3.8}, one has that
\begin{align*}
&c_v\frac{d}{dt}\|\rho_0^{1-\gamma}(\theta_\ell)_+\|_{L^2}^2
+\kappa\|\rho_0^{\frac12-\gamma}\nabla(\theta_\ell)_+\|_{L^2}^2\nonumber\\
&\le\frac{C}{(\ell-m)^3}\|\rho^{1-\gamma}\theta_m^+\|_{L^2}^\frac32(\|\rho_0^{1-\gamma}\nabla\theta_m^+\|_{L^2}
+\|\rho_0^{1-\gamma}\theta_m^+\|_{L^2})^\frac32+
C\phi(t)\|\rho_0^{1-\gamma}(\theta_\ell)_+\|_{L^2}^2,
\end{align*}
from which, applying Gronwall's inequality, Lemma \ref{l21}, and noticing that
$(\theta_\ell)_+|_{t=0}=0$, for any $\ell\ge \overline{S}_0$, and the following fact
\begin{align*}
&\int_0^T\|\rho^{1-\gamma}\theta_m^+\|_{L^2}^\frac32\big(\|\rho_0^{1-\gamma}\nabla\theta_m^+\|_{L^2}
+\|\rho_0^{1-\gamma}\theta_m^+\|_{L^2}\big)^\frac32dt\nonumber\\
&\le C\int_0^T\big(\|\rho^{1-\gamma}\theta_m^+\|_{L^2}^2\big)^\frac32dt
+\int_0^T\|\rho^{1-\gamma}\theta_m^+\|_{L^2}^\frac32\|\rho_0^{1-\gamma}\nabla\theta_m^+\|_{L^2}^\frac32dt\nonumber\\
&\le C\Big(\sup_{0\le t\le T}\|\rho^{1-\gamma}\theta_m^+\|_{L^2}^2\Big)^\frac32
+C\Big(\int_0^T\|\rho^{1-\gamma}\theta_m^+\|_{L^2}^6dt\Big)^\frac14\Big(\int_0^T\|\rho_0^{1-\gamma}\nabla\theta_m^+\|_{L^2}^2dt\Big)^\frac34\nonumber\\
&\le C\Big(\sup_{0\le t\le T}\|\rho^{1-\gamma}\theta_m^+\|_{L^2}^2\Big)^\frac32
+C\Big(\int_0^T\|\rho_0^{1-\gamma}\nabla\theta_m^+\|_{L^2}^2dt\Big)^\frac32
\le C\mathcal Z_m^\frac32.
\end{align*}
Thus, the second conclusion follows.
\end{proof}

\section{Proof of Theorem \ref{thm1}}\label{sec4}

Combining Lemmas \ref{l21}--\ref{l27} and Lemmas \ref{3.1}--\ref{3.2} altogether, and using the similar arguments of Li and Xin \cite{LX23}, it is not difficult to get the uniform lower and upper boundedness of the entropy by a modified De Giorgi-type iteration (see Lemma \ref{l41}). Here we omit the details for simplicity.
\hfill $\Box$

\section*{Appendix}
In this appendix, we state the following De Giorgi type iterative lemma (cf. \cite{LX23}), which is important to derive the uniform lower and upper bounds of the entropy.
\begin{lemma}\label{l41}
Let $m_0\in [0, \infty)$ be given and $f$ be a nonnegative non-increasing function on $[m_0, \infty)$ satisfying
\begin{align*}
f(\ell)\le \frac{M_0(\ell+1)^\alpha}{(\ell-m)^\beta}f^\sigma(m), \quad \forall\ell>m\ge m_0,
\end{align*}
for some nonnegative constants $M_0$, $\alpha$, $\beta$, and $\sigma$, with $0\le\alpha<\beta$ and $\sigma>1$. Then,
\begin{align*}
f(m_0+d)=0,
\end{align*}
where
\begin{align*}
d=\Big[2f^\sigma(m_0)(m_0+M_0+2)^{\frac{2\alpha+2\beta+1}{\sigma-1}+\frac{\beta}{(\sigma-1)^2}+2\alpha+\beta+1}\Big]^\frac{1}{\beta-\alpha}
+2.
\end{align*}
\end{lemma}

\end{document}